\documentclass[review,3p,square]{elsarticle}
\usepackage{amssymb}
\usepackage{amsmath}
\usepackage{amsthm}
\usepackage{stmaryrd}
\usepackage{natbib}
\usepackage{bbm}
\usepackage{amsmath}
\allowdisplaybreaks[4]
\usepackage{enumitem}
\usepackage{fancyhdr}
\usepackage{hyperref}
\hypersetup{%
  pdftitle={BSDEs},%
  pdfsubject={BSDEs},%
  pdfauthor={XinYing Li, ShengJun FAN},%
  pdfkeywords={BSDEs},%
  pdfstartview=FitH,%
  CJKbookmarks=true,%
  bookmarksnumbered=true,%
  bookmarksopen=true,%
  colorlinks=true, linkcolor=blue, urlcolor=blue, citecolor=blue, %
}

\usepackage{cleveref}
\crefname{thm}{Theorem}{Theorems}
\crefname{pro}{Proposition}{Propositions}
\crefname{lem}{Lemma}{Lemmas}
\crefname{rmk}{Remark}{Remarks}
\crefname{cor}{Corollary}{Corollaries}
\crefname{dfn}{Definition}{Definitions}
\crefname{ex}{Example}{Examples}
\crefname{section}{Section}{Sections}
\crefname{subsection}{Subsection}{Subsections}

\newcommand{\as}{{\rm d}\mathbb{P}\times{\rm d} t-a.e.}

\newcommand{\F}{\mathcal{F}}
\newcommand{\E}{\mathbb{E}}

\newcommand{\T}{[0,T]}

\newcommand{\R}{{\mathbb R}}

\newcommand{\RE}{\forall}

\newcommand {\Dis}{\displaystyle}

\newtheorem{thm}{Theorem}[section]
\newtheorem{lem}[thm]{Lemma}
\newtheorem{pro}[thm]{Proposition}
\newtheorem{rmk}[thm]{Remark}
\newtheorem{cor}[thm]{Corollary}
\newtheorem{dfn}[thm]{Definition}
\newtheorem{ex}[thm]{Example}

\journal{arXiv}

\begin{document}
\begin{frontmatter}

\title{{Weighted solutions of random time horizon BSDEs with stochastic monotonicity and general growth generators and related PDEs}\tnoteref{found}}
\tnotetext[found]{Partially supported by National Natural Science Foundation of China (No. 12171471).
\vspace{0.2cm}}

\author{{Xinying Li \qquad Yaqi Zhang}\qquad Shengjun Fan$^{*}$ \vspace{0.3cm} \\ \textit{School of Mathematics, China University of Mining and Technology, Xuzhou 221116, PR China \vspace{-0.5cm}}}

\cortext[cor1]{Corresponding author\vspace{0.2cm}}

\cortext[cor2]{E-mail: lixinyingcumt@163.com (X. Li), TS22080028A31@cumt.edu.cn (Y. Zhang), shengjunfan@cumt.edu.cn (S. Fan)}

\vspace{0.2cm}
\begin{abstract}
This study focuses on a multidimensional backward stochastic differential equation (BSDE) with a general random terminal time $\tau$ taking values in $[0,+\infty]$. The generator $g$ satisfies a stochastic monotonicity condition in the first unknown variable $y$ and a stochastic Lipschitz continuity condition in the second unknown variable $z$, and it can have a more general growth with respect to $y$ than the classical one stated in (H5) of \cite{Briand2003}. Without imposing any restriction of finite moment on the stochastic coefficients, we establish a general existence and uniqueness result for the weighted solution of such BSDE in a proper weighted $L^2$-space with a suitable weighted factor. This result is proved via some innovative ideas and delicate analytical techniques, and it unifies and strengthens some existing works on BSDEs with stochastic monotonicity generators, BSDEs with stochastic Lipschitz generators, and BSDEs with deterministic Lipschitz/monotonicity generators. Then, a continuous dependence property and a stability theorem for the weighted $L^2$-solutions are given. We also derive the nonlinear Feynman-Kac formulas for both parabolic and elliptic PDEs in our context. \vspace{0.2cm}
\end{abstract}

\begin{keyword}
Backward stochastic differential equation\sep Weighted solution\sep Existence and uniqueness\sep \\
\hspace*{1.85cm} Random terminal time \sep Stochastic monotonicity condition\sep \\
\hspace*{1.85cm} Stochastic Lipschitz condition \sep Nonlinear Feynman-Kac formulas\vspace{0.2cm}

\MSC[2021] 60H10, 60H30\vspace{0.2cm}
\end{keyword}

\end{frontmatter}
\vspace{-0.4cm}

\section{Introduction}
Linear backward stochastic differential equation (BSDE in short) appeared in \cite{Bismut1973} as an adjoint process in the maximum principle of stochastic control, and nonlinear BSDE was initially studied in \cite{PardouxPeng1990SCL} laying the foundation on the study of the BSDE theory and its applications. During the past over thirty years, BSDEs have been intensively investigated with broader applications, and they have gradually become a powerful tool in various fields including PDEs, mathematical finance, optimal control and so on, see for example \cite{Peng1991SSR}, \cite{Tang1994}, \cite{DarlingandPardoux1997}, \cite{KarouiPengQuenez1997}, \cite{HuPeng1997}, \cite{BriandandHu1998}, \cite{Pardoux1999}, \cite{HuYing2005}, \cite{BriandandConfortola2008}, \cite{DelbaenTang2010}, \cite{PardouxandRascanu(2014)}, \cite{Bahlali2015}, \cite{FanHu2021SPA} and \cite{Tian2023SIAM} among others for more details.

In this study, we are concerned with the following typical multidimensional BSDE:
\begin{align}\label{BSDE1.1}
  y_t=\xi+\int_t^\tau g(s,y_s,z_s){\rm d}s-\int_t^\tau z_s{\rm d}B_s, \ \ t\in[0,\tau],
\end{align}
where $(B_t)_{t\geq0}$ is a standard $d$-dimensional Brownian motion defined on a complete probability space $(\Omega,\F,\mathbb{P})$ and generates an augmented $\sigma$-algebra filtration $(\F_t)_{t\geq0}$ with $\F:=\F_\tau$ and $\tau$ being a general $(\F_t)$-stopping time taking values in $[0,+\infty]$ called the terminal time, $\xi$ is an $\F_\tau$-measurable $k$-dimensional random vector called the terminal value, and the random function
$$
g(\omega,t,y,z): \Omega\times[0,\tau]\times \R^k\times \R^{k\times d}\mapsto \R^k
$$
is $(\F_t)$-progressively measurable for each $(y,z)$ called the generator of BSDE \eqref{BSDE1.1}. BSDE \eqref{BSDE1.1} with parameters $(\xi,\tau,g)$ is often denoted as BSDE$(\xi,\tau,g)$. An adapted solution of BSDE \eqref{BSDE1.1} is a pair of $(\F_t)$-progressively measurable processes $(y_t,z_t)_{t\in [0,\tau]}$ taking values in $\R^k\times\R^{k\times d}$ such that $\mathbb{P}-a.s.$, $Y_t$ is continuous, $\int_0^\tau (|g(t,y_t,z_t)|+|z_t|^2){\rm d}t<+\infty$, and \eqref{BSDE1.1} holds. Notably, for any adapted solution $(y_t,z_t)_{t\in[0,\tau]}$ of BSDE \eqref{BSDE1.1}, we have
\begin{align}\label{eq:1.2}
\lim\limits_{t\rightarrow \tau} y_t=\xi\ \ {\rm on\  the\ set\ of}\ \{\tau=+\infty\}.
\end{align}
The following BSDE is closely related to BSDE \eqref{BSDE1.1}: for each $T>0$,
\begin{equation}\label{BSDE1.1*}
\left\{
\begin{array}{l}
\Dis y_t=y_{T\wedge\tau}+\int_{t\wedge\tau}^{T\wedge\tau} g(s,y_s,z_s){\rm d}s-\int_{t\wedge\tau}^{T\wedge\tau} z_s{\rm d}B_s, \ \ t\in [0,T];\vspace{0.2cm}\\
\Dis y_t=\xi\ \ {\rm on\  the\ set\ of}\ \{t\geq \tau\}.
\end{array}
\right.
\end{equation}
By a solution of BSDE \eqref{BSDE1.1*}, we mean a pair of $(\F_t)$-progressively measurable processes $(y_t,z_t)_{t\geq 0}$ taking values in $\R^k\times\R^{k\times d}$ such that $\mathbb{P}-a.s.$, $Y_{t\wedge T}$ is continuous, ${\bf 1}_{t\geq \tau}z_t\equiv 0$, $\int_0^{T\wedge\tau} (|g(t,y_t,z_t)|+|z_t|^2){\rm d}t<+\infty$, and \eqref{BSDE1.1*} holds for each $T>0$. This type of BSDE was intensively investigated for example in \cite{DarlingandPardoux1997}, \cite{BriandandHu1998}, \cite{Pardoux1999}, \cite{Royer2004}, \cite{Aman2012}, \cite{PardouxandRascanu(2014)}, \cite{O2020} and \cite{Lin2020}. We would emphasize that if $\tau$ is finite stopping time, i.e., $\mathbb{P}(\tau<+\infty)=1$, then BSDE \eqref{BSDE1.1*} is equivalent to BSDE \eqref{BSDE1.1}. However, if $\mathbb{P}(\tau=+\infty)>0$, then \eqref{BSDE1.1*} indicates that the $y_\cdot$ has nothing to do with the terminal value $\xi$ on the set $\{\tau=+\infty\}$, having a significant difference with \eqref{eq:1.2}. In order to obtain an adapted solution of BSDE \eqref{BSDE1.1*} with unbounded terminal time $\tau$, i.e., $\mathbb{P}(\tau\geq M)>0$ for any $M>0$, a natural strategy is to use the solutions of approximated BSDEs with constant terminal time to approach the desired solution. More specifically, for each $n\geq1$, let $(y_t^n,z_t^n)_{t\geq0}$ be an adapted solution of the following BSDE:
\begin{equation}\label{eq:1.4}
\left\{
\begin{array}{l}
\Dis y_t^n=\xi_n+\int_{t}^{n} {\bf 1}_{s\leq \tau}g(s,y_s^n,z_s^n){\rm d}s-\int_{t}^{n} z_s^n{\rm d}B_s, \ \ 0\leq t\leq n; \\
\Dis y_t^n=\xi_t\ \ {\rm and}\ \ z_t^n=\eta_t,\ \ t\geq n
\end{array}
\right.
\end{equation}
with two processes $\xi_\cdot$ and $\eta_\cdot$ such that $\xi=\xi_t+\int_t^{\infty}\eta_s{\rm d}B_s$ for each $t\geq 0$ via the martingale representation theorem. Then, $\{(y_t^n,z_t^n)_{t\geq0}\}_{n=1}^{+\infty}$ is shown to be a Cauchy sequence in a proper space under some additional integrability conditions on $g(t,\xi_t,\eta_t)$, and the limit is a desired solution of BSDE \eqref{BSDE1.1*}.

Due to \eqref{eq:1.2}, it seems to be more natural to study BSDE \eqref{BSDE1.1} than BSDE \eqref{BSDE1.1*}. On the other hand, it also seems to be more convenient to unify those works on BSDEs with the bounded terminal time, unbounded terminal time and infinite terminal time via the study of BSDE \eqref{BSDE1.1} since it has a unified form regardless of $\tau(\omega)$ being finite or infinite. Let us quickly illustrate what we want to do in this study. It is well known that by combining the martingale representation theorem with the contract mapping argument, \cite{PardouxPeng1990SCL} initially investigated existence and uniqueness of the adapted solution for BSDE \eqref{BSDE1.1} with constant terminal time under the uniform Lipschitz continuity assumption of the generator $g$ with respect to the unknown variables $(y,z)$: there exists two nonnegative constants $\mu$ and $\nu$ such that for each $y,y_1, y_2\in\R^k$ and $z,z_1, z_2\in\R^{k\times d}$,
\begin{align}\label{a1.2}
|g(\omega,t,y_1,z)-g(\omega,t,y_2,z)|\leq \mu |y_1-y_2|\vspace{-0.2cm}
\end{align}
and\vspace{-0.2cm}
\begin{align}\label{a1.3}
|g(\omega,t,y,z_1)-g(\omega,t,y,z_2)|\leq \nu |z_1-z_2|.
\end{align}
Based on this result, by virtue of the convolution approaching idea and the weak convergence method together with the a priori estimate technique, the uniform Lipschitz continuity condition \eqref{a1.2} of the generator $g$ in $y$ was successfully weakened in \cite{Pardoux1999} to the following monotonicity condition
 \begin{align}\label{a1.333}
\left\langle y_1-y_2,g(\omega,t,y_1,z)-g(\omega,t,y_2,z)\right\rangle\leq \mu|y_1-y_2|^2,
\end{align}
combined with a general growth condition: there exists a continuous increasing function $\varphi:\R_+\rightarrow\R_+$ with $\varphi(0)=0$ such that for each $y\in\R^k$ and $z \in\R^{k\times d}$
\begin{align}\label{a1.5}
|g(\omega,t,y,0)|\leq |g(\omega,t,0,0)|+\varphi(|y|).
\end{align}
This growth condition extends the polynomial growth condition used in \cite{BriandCarmona2000} and it was further loosed the classical general growth condition: for each $r\in \R_+$,
\begin{align}\label{a1.6}
\sup_{|y|\leq r}|g(\cdot,y,0)-g(\cdot,0,0)|\in L^1(\Omega\times \T),
\end{align}
see \cite{Briand2003} for more details, where a delicate truncation argument for the generator $g$ is used to make fully use of the general growth condition \eqref{a1.6}. In this study, we are interested in solving BSDEs with a general random terminal time under more general assumptions on the generator $g$, roughly speaking, a BSDE with stochastic coefficients, and aim to advance the theory and applications of BSDEs. It should be mentioned that BSDEs with stochastic coefficients arise naturally in mathematical finance, stochastic control and some other topics, as indicated for example in \cite{KarouiPengQuenez1997}, \cite{BriandandConfortola2008} and \cite{Bahlali2015}. More specifically, we will study BSDE \eqref{BSDE1.1} with random terminal time and suppose that the generator $g$ satisfies the stochastic monotonicity condition in $y$ and the stochastic Lipschitz continuity condition in $z$, i.e., \eqref{a1.333} and \eqref{a1.3} hold with two nonnegative processes $\mu_\cdot$ and $\nu_\cdot$ instead of the constants $\mu$ and $\nu$, and it can have a more general growth in $y$ than that stated in \eqref{a1.6}. The coefficients $\mu_\cdot$ and $\nu_\cdot$ are only assumed to satisfy the condition of
\begin{align}\label{a1.7}
\int_0^{\tau(\omega)} (\mu_t(\omega)+\nu_t^2(\omega)){\rm d}t<+\infty.
\end{align}

Next, let us further review some existing results related closely to ours and enlightened our idea. First of all, by subdividing the time interval via stopping times and using the martingale representation theorem and the contract mapping argument, \cite{Liu2020} established existence and uniqueness for the adapted solutions of infinite time horizon BSDE \eqref{BSDE1.1}, where the generator $g$ satisfies the stochastic Lipschitz continuity condition in $(y,z)$, i.e., \eqref{a1.2} and \eqref{a1.3} hold with the constants $\mu$ and $\nu$ being replaced with two nonnegative processes $\mu_\cdot$ and $\nu_\cdot$, and it is also required that
\begin{align}\label{a1.8}
\int_0^{\tau(\omega)} (\mu_t(\omega)+\nu^2_t(\omega)){\rm d}t\leq M
\end{align}
for a constant $M>0$. Readers are also referred to \cite{Chen1998} and \cite{ZChenBWang2000JAMS} for the case that the processes $\mu_\cdot$ and $\nu_\cdot$ do not depend on the sample $\omega$. \cite{Yong2006} considered the solvability of linear BSDEs with constant terminal time $T$ and unbounded stochastic coefficients $\mu_\cdot$ and $\nu_\cdot$, where $\int_0^{T}\mu_t(\omega) {\rm d}t$ and $\int_0^{T}\nu_t^2(\omega) {\rm d}t$ are assumed to have a certain exponential moment. BSDEs with random terminal time and stochastic Lipschitz continuity generators were also investigated by \cite{KarouiHuang1997}, \cite{BenderKohlmann2000}, \cite{WangRanChen2007} and \cite{Li2023}, where the restrictive condition \eqref{a1.8} is weakened to \eqref{a1.7}, and the adapted solutions and terminal conditions of BSDEs are all required to lie in some certain weighted spaces with a weighted factor $e^{\int_0^t (\mu_s(\omega)+\nu_s^2(\omega)){\rm d}s}$. Notably, it is very natural that the integrability of adapted solutions of BSDEs should depend on not only integrability of the terminal value $\xi$, but also that of the coefficients $\mu_\cdot$ and $\nu_\cdot$.

Furthermore, under the assumptions that the generator $g$ satisfies \eqref{a1.3}, \eqref{a1.6} and a weaker monotonicity condition than \eqref{a1.333}, by using the convolution approaching idea, the a priori estimate technique and the truncation argument used in \cite{Pardoux1999} and \cite{Briand2003}, and the uniform continuity of a continuous function on a compact set instead of the weak convergence method, \cite{FanJiang2013} proved an existence and uniqueness result for the adapted solution of multidimensional BSDE \eqref{BSDE1.1} with constant terminal time. This result was further extended in \cite{Xiao2015} and \cite{XiaoandFan2017} to the case of infinite time horizon BSDEs, where the constants $\mu$ and $\nu$ in \eqref{a1.333} and \eqref{a1.3} need to be replaced with two deterministic functions $\mu_t$ and $\nu_t$ such that
\begin{align}\label{a1.9}
\int_0^{+\infty} (\mu_t+\nu_t^2){\rm d}t<+\infty.
\end{align}
We mention that \cite{XiaoandFan2017} developed a better truncation argument to dominate the growth of the generator $g$ and its gradient such that the proof procedure was greatly simplified. Readers are also referred to \cite{LiXuFan2021PUQR} and \cite{LiFan2023CSTM} for infinite time horizon BSDEs with the general growth condition  \eqref{a1.6} and stochastic coefficients $\mu_\cdot$ and $\nu_\cdot$ satisfying \eqref{a1.8}. Additionally, under the assumptions that the generator $g$ has a sub-quadratic growth in $y$ and a super-linear growth in $z$, \cite{Bahlali2015} dealt with the solvability of multidimensional BSDEs with constant terminal time $T$ and stochastic coefficients $\mu_\cdot$ and $\nu_\cdot$ under some local conditions, where the $\int_0^{T}(\mu_t(\omega)+\nu_t^2(\omega)){\rm d}t$ is forced to have a certain exponential moment, see Example 3 in \cite{Bahlali2015} for more details.

Inspired by the above works, in this study we establish a general existence and uniqueness result for the adapted solution of multidimensional BSDE \eqref{BSDE1.1} with a general random terminal time, see \cref{thm:3.1} in Section 3 for details. The generator $g$ of the BSDE is required to satisfy the stochastic monotonicity condition in $y$ and the stochastic Lipschitz continuity condition in $z$ with stochastic coefficients $\mu_\cdot$ and $\nu_\cdot$ only satisfying \eqref{a1.7} and without any other restrictions, and it can have a more general growth in the state variable $y$ than that stated in \eqref{a1.6}. The desired adapted solution is required to lie in a weighted $L^2$-space with a weighted factor $e^{\int_0^t (\beta \mu_s(\omega)+\frac{\rho}{2}\nu_s^2(\omega)){\rm d}s}$ for any given $\beta\geq 1$ and $\rho>1$. This weighted space is different from those used in \cite{KarouiHuang1997}, \cite{BenderKohlmann2000} and \cite{Li2023}. The ranges of $\beta$ and $\rho$ is the key point. A motivational example is provided in Section 2 to explain why we choose such weighted factor. We emphasize that when \eqref{a1.8} is satisfied, the weighted $L^2$-space is just the usual $L^2$-space. Our result unifies and significantly strengthens some existing ones, see corollaries together with remarks and examples in Section 3 for more details. In \cref{pro:1.1} of Section 2, with the help of the weighted factor, It\^{o}'s formula and the Burkholder-Davis-Gundy inequality we establish an a priori estimate on the weighted solution of a random time horizon multidimensional BSDE in the weighted $L^2$-space. Based on it, the uniqueness part of \cref{thm:3.1}, a continuous dependence property and a stability theorem for the weighted solutions of BSDEs (see \cref{thm:3.2,thm:3.3} in Section 3) are naturally verified. As for the proof of the existence part of \cref{thm:3.1}, the approaching strategy applied to study BSDE \eqref{BSDE1.1*} is not available any longer, and the additional integrability condition on $g(t,\xi_t,\eta_t)$ is not required. Instead, we apply systematically the techniques used in \cite{Pardoux1999}, \cite{Briand2003}, \cite{FanJiang2013} and \cite{XiaoandFan2017}, and develop some innovative ideas tackling new difficulties. These difficulties arise naturally due to the combination of the weaker growth of the generator $g$ in $y$, the more general integrability of stochastic coefficients $\mu_\cdot$ and $\nu_\cdot$, and the more general weighted $L^2$-space, see \cref{pro:2.1} and \cref{rmk:4.2,rmk:4.3,rmk:4.4} in Section 4 for more details. It finally turns out that if a proper weighted factor is taken into consideration, then the existence and uniqueness result on solutions of BSDEs over finite time intervals presented for example in \cite{Briand2003} and \cite{Pardoux1999} holds still for random time horizon BSDEs with stochastic coefficients.

The rest of this study is organized as follows. In Section 2, a motivational example is given and the a prior estimate-\cref{pro:1.1} is established. In Section 3, we state our main result-\cref{thm:3.1} and prove its uniqueness part. Two corollaries of \cref{thm:3.1} as well as several remarks and examples are provided to illustrate the novelty of our main result, and \cref{thm:3.2,thm:3.3} are also proposed and proved in this section. As by-product, Theorems \ref{thm:3.4.1} and \ref{thm:3.4.2} present the applications for related PDEs. Finally, section 4 is devoted to the proof of the existence part of \cref{thm:3.1}.\vspace{0.2cm}

\section{Preliminaries}
\setcounter{equation}{0}

In this section, we firstly introduce some notations, and then give a motivational example about the weighted factor. Finally, with the help of the weighted factor, a general a priori estimate is established.


\subsection{Some notations and a definition}

The following notations will be used throughout this study. Denote the usual Euclidean norm by $|\cdot|$ and the usual Euclidean inner product of $x,y\in \R^k$ by $\langle x,y\rangle$. If $A$ is a matrix, we have $|A|^2=tr(AA^*)$, where $A^*$ denotes the transpose of $A$. Let $\R_+:=[0,+\infty)$ and $\hat{y}:=\frac{y}{|y|}{\bf 1}_{|y|\neq0}$ for each $y\in \R^k$, where ${\bf 1}_B$ represents the indicator of set $B$. Every equality and inequality between random elements should be understood as holding $\mathbb{P}$-almost surely.

Throughout this study, let $\beta$ and $\rho$ be two given constants satisfying $\beta\geq1$ and $\rho>1$, $\mu_\cdot$ and $\nu_\cdot$ be two given $(\F_t)$-progressively measurable nonnegative processes with $$a_\cdot:=\beta\mu_\cdot+\frac{\rho}{2} \nu_\cdot^2.$$
We also always assume that
$$\int_{0}^{\tau}{a}_t{\rm d}t<+\infty.\vspace{0.2cm}$$
Note that the ranges of $\beta$ and $\rho$ are inspired by a motivational example postponed in subsection 2.2. Let us further introduce the following weighted $L^2$-spaces with a weighted factor $e^{\int_0^t a_s{\rm d}s}$.\vspace{0.2cm}

$\bullet$  $L_\tau^2(a_\cdot;\R^k)$ denotes the set of $\F_\tau$-measurable $\R^k$-valued random vectors $\xi$ such that
$$\|\xi\|^2_{a_\cdot}:=\E\left[e^{2 \int_0^\tau a_s{\rm d}s}|\xi|^2\right]<+\infty.$$

$\bullet$  $S_\tau^2(a_\cdot;\R^k)$ denotes the set of $(\F_t)$-adapted, $\R^k$-valued and continuous processes $(Y_t)_{t\in[0,\tau]}$ such that
$$\|Y_\cdot\|^2_{a_\cdot,c}:=\E\left[\sup_{t\in[0,\tau]}\left(e^{2 \int_0^t a_r{\rm d}r}|Y_t|^2\right)\right]<+\infty.$$

$\bullet$  $M_\tau^2(a_\cdot;\R^{k\times d})$ denotes the set of $(\F_t)$-progressively measurable $\R^{k\times d}$-valued processes $(Z_t)_{t\in[0,\tau]}$ such that
$$\|Z_\cdot\|^2_{a_\cdot}:=\E\left[\int_0^\tau e^{2 \int_0^s a_r{\rm d}r}|Z_s|^2{\rm d}s\right]<+\infty.$$
When the process $a_\cdot$ is replaced with another $(\F_t)$-progressively measurable $\tilde{a}_\cdot$ with $\int_{0}^{\tau}\tilde{a}_t{\rm d}t<+\infty$, the above three spaces can be identically defined. In particular, in case of $a_\cdot\equiv 0$, or more generally, $\int_0^\tau a_t{\rm d}t\leq M$ for a constant $M>0$,
these spaces are just the usual spaces used in \cite{Briand2003}, \cite{Xiao2015} and so on. Furthermore, define
$$H_\tau^2(a_\cdot;\R^{k}\times\R^{k\times d}):=S_\tau^2(a_\cdot;\R^k)\times M_\tau^2(a_\cdot;\R^{k\times d}).$$
It is clear that $H_\tau^2(a_\cdot;\R^{k}\times\R^{k\times d})$ is a Banach space with the norm
$$\|(Y_\cdot,Z_\cdot)\|^2_{a_\cdot}:=\|Y_\cdot\|^2_{a_\cdot,c}
+\|Z_\cdot\|^2_{a_\cdot}.$$

We introduce the following definition on the weighted $L^2$-solution of BSDEs.

\begin{dfn}
Assume that $(y_t,z_t)_{t\in[0,\tau]}$ is an adapted solution to BSDE \eqref{BSDE1.1}. More specifically, $(y_t,z_t)_{t\in[0,\tau]}$ is a pair of $(\F_t)$-progressively measurable processes taking values in $\R^k\times\R^{k\times d}$ such that $\mathbb{P}-a.s.$, $y_t$ is continuous, $\int_0^\tau (|g(t,y_t,z_t)|+|z_t|^2){\rm d}t<+\infty$, and \eqref{BSDE1.1} holds. Moreover, if $(y_t,z_t)_{t\in[0,\tau]}\in H_\tau^2(a_\cdot;\R^{k}\times\R^{k\times d})$, then it will be called a weighted $L^2$-solution of BSDE \eqref{BSDE1.1} in $H_\tau^2(a_\cdot;\R^{k}\times\R^{k\times d})$.
\end{dfn}

\subsection{A motivational example}

In this section, we give a motivational example to explain why we choose the weighted factor $e^{\int_0^t (\beta\mu_s+\frac{\rho}{2}\nu_s^2){\rm d}s}$ with $\beta\geq1$ and $\rho>1$.

Assume that $k=d=1$ and $T>0$ is a finite real number. Considering the following linear BSDE:
\begin{align}\label{linearBSDE}
  y_t=\xi+\int_t^T (\mu_sy_s+\nu_sz_s){\rm d}s-\int_t^T z_s{\rm d}B_s, \ \ t\in[0,T].
\end{align}
It is well known that when $\mu_\cdot$, $\nu_\cdot$ and $\xi$ are bounded, BSDE \eqref{linearBSDE} admits a unique bounded solution
\begin{align}\label{yt}
y_t=\E\left[\xi e^{\int_t^T\mu_r{\rm d}r+\int_t^T \nu_r{\rm d}B_r-\frac{1}{2}\int_t^T\nu_r^2{\rm d}r}\bigg|\F_t\right], \ t\in[0,T].
\end{align}
In the financial market, $\mu_\cdot$ and $\nu_\cdot$ usually denote the risk-free interest rate and the risk volatility, respectively, and $y_t$ represents the price of the contingent claim $\xi$ at time $t$. Generally speaking, all of $\xi$, $\mu_\cdot$, $\nu_\cdot$ and $y_\cdot$ can be unbounded. Therefore, it is natural to study the unbounded solution of BSDE \eqref{linearBSDE} with unbounded stochastic coefficients.

In order to ensure that the process $y_\cdot$ in \eqref{yt} is well defined, it is necessary and sufficient that
$$
\E\left[\xi e^{\int_0^T\mu_r{\rm d}r+\int_0^T \nu_r{\rm d}B_r-\frac{1}{2}\int_0^T\nu_r^2{\rm d}r}\right]<+\infty.
$$
We do not hope this condition to be related to $\int_0^T \nu_r{\rm d}B_r$. Note by H\"{o}lder's inequality that
\begin{align*}
\begin{split}
&\E\left[\xi e^{\int_0^T\mu_r{\rm d}r+\frac{1}{2}\int_0^T\nu_r^2{\rm d}r+\int_0^T \nu_r{\rm d}B_r-\int_0^T\nu_r^2{\rm d}r}\right]\\
&\ \ \leq\left(\E\left[|\xi|^2e^{2\int_0^T\mu_r{\rm d}r+\int_0^T\nu_r^2{\rm d}r}\right]\right)^{\frac{1}{2}}
\left(\E\left[e^{2\int_0^T \nu_r{\rm d}B_r-2\int_0^T\nu_r^2{\rm d}r}\right]\right)^{\frac{1}{2}}\\
&\ \ \leq\left(\E\left[|\xi|^2e^{2\int_0^T\mu_r{\rm d}r+\int_0^T\nu_r^2{\rm d}r}\right]\right)^{\frac{1}{2}}.
\end{split}
\end{align*}
The process $y_\cdot$ in \eqref{yt} can be well defined provided that for some constants $\tilde{\beta}\geq 1$ and $\tilde{\rho}\geq 1$,
\begin{align}\label{betarou}
\begin{split}
\E\left[|\xi|^2 e^{2\tilde{\beta}\int_0^T\mu_r{\rm d}r+\tilde{\rho}\int_0^T\nu_r^2{\rm d}r}\right]<+\infty.
\end{split}
\end{align}
Furthermore, we can prove that if \eqref{betarou} holds, then the process $y_\cdot$ in \eqref{yt} satisfies that for each $t\in \T$,
\begin{align}\label{yt222}
\begin{split}
\E\left[|y_t|^2e^{2\tilde{\beta}\int_0^t\mu_r{\rm d}r+\tilde{\rho}\int_0^t\nu_r^2{\rm d}r}\right]<+\infty.
\end{split}
\end{align}
Indeed, considering $\tilde{\beta}\geq1$, $\tilde{\rho}\geq1$ and H\"{o}lder's inequality we deduce that for each $t\in \T$,
\begin{align}\label{yt*}
&y_te^{\tilde{\beta}\int_0^t\mu_r{\rm d}r+\frac{\tilde{\rho}}{2}\int_0^t\nu_r^2{\rm d}r} =\E\left[\xi e^{\int_t^T\mu_r{\rm d}r+\int_t^T\nu_r{\rm d}B_r-\frac{1}{2}\int_t^T\nu_r^2{\rm d}r+\tilde{\beta}\int_0^t\mu_r{\rm d}r+\frac{\tilde{\rho}}{2}\int_0^t\nu_r^2{\rm d}r}\bigg|\F_t\right]\nonumber\\
&\ \ =\E\left[\xi e^{\tilde{\beta}\int_0^T\mu_r{\rm d}r+\frac{\tilde{\rho}}{2}\int_0^T\nu_r^2{\rm d}r}e^{(1-\tilde{\beta})\int_t^T\mu_r{\rm d}r+\int_t^T\nu_r{\rm d}B_r-\frac{\tilde{\rho}+1}{2}\int_t^T\nu_r^2{\rm d}r}\bigg|\F_t\right]\\
&\ \ \leq \left(\E\left[|\xi|^2 e^{2\tilde{\beta}\int_0^T\mu_r{\rm d}r+\tilde{\rho}\int_0^T\nu_r^2{\rm d}r}\bigg|\F_t\right]\right)^{\frac{1}{2}} \left(\E\left[e^{2(1-\tilde{\beta})\int_t^T\mu_r{\rm d}r+2\int_t^T\nu_r{\rm d}B_r-(\tilde{\rho}+1)\int_t^T\nu_r^2{\rm d}r}\bigg|\F_t\right]\right)^{\frac{1}{2}}\nonumber\\
&\ \ \leq \left(\E\left[|\xi|^2 e^{2\tilde{\beta}\int_0^T\mu_r{\rm d}r+\tilde{\rho}\int_0^T\nu_r^2{\rm d}r}\bigg|\F_t\right]\right)^{\frac{1}{2}}
\left(\E\left[e^{2\int_t^T\nu_r{\rm d}B_r-2\int_t^T\nu_r^2{\rm d}r}\bigg|\F_t\right]\right)^{\frac{1}{2}}.\nonumber
\end{align}
Since a non-negative local martingale is a supermartingale, we know that for each $c\in\R_+$ and $t\in[0,T]$,
\begin{align}\label{23}
\begin{split}
\E\left[e^{c\int_t^T\nu_r{\rm d}B_r-\frac{c^2}{2}\int_t^T\nu_r^2{\rm d}r}\bigg|\F_t\right] =\E\left[e^{c\int_0^T\tilde{\nu}_r{\rm d}B_r-\frac{c^2}{2}\int_0^T\tilde{\nu}_r^2{\rm d}r}\bigg|\F_t\right]\leq e^{c\int_0^t\tilde{\nu}_r{\rm d}B_r-\frac{c^2}{2}\int_0^t\tilde{\nu}_r^2{\rm d}r}=1,
\end{split}
\end{align}
where
\begin{align*}
\tilde{\nu}_r:=
\left\{\begin{aligned}
&\nu_r, \quad t\leq r\leq T;\\
&0, \quad \quad 0\leq r< t.
\end{aligned}\right.
\end{align*}
Thus, the desired assertion \eqref{yt222} follows immediately by combining \eqref{yt*}, \eqref{23} and \eqref{betarou}.

It should be especially mentioned that for the case of $\tilde{\rho}=1$, when \eqref{betarou} is satisfied, although
\eqref{yt222} holds for each $t\in\T$, it is uncertain that the process $y_\cdot$ in \eqref{yt} satisfies
\begin{align}\label{29sup}
\begin{split}
\E\left[\sup\limits_{t\in[0,T]}\left(|y_t|^2e^{2\tilde{\beta}\int_0^t\mu_r{\rm d}r+\tilde{\rho}\int_0^t\nu_r^2{\rm d}r}\right)\right]<+\infty.
\end{split}
\end{align}
For example, let $\mu_\cdot=0$, $\nu_\cdot=b_\cdot$ and $\xi:=e^{\int_0^Tb_r{\rm d}B_r-\frac{3}{2}\int_0^Tb_r^2{\rm d}r}$, where $b_\cdot$ is a given $(\F_t)$-progressively measurable nonnegative processes such that $\left\{e^{2\int_0^tb_r{\rm d}B_r-2\int_0^tb_r^2{\rm d}r}\right\}_{t\in[0,T]}$ is uniformly integrable, but
\begin{align*}
\E\left[\sup\limits_{t\in[0,T]}e^{2\int_0^tb_r{\rm d}B_r-2\int_0^tb_r^2{\rm d}r}\right]=+\infty.
\end{align*}
Then, \eqref{betarou} with $\tilde\rho=1$ holds since
\begin{align*}
\E\left[|\xi|^2 e^{\int_0^T b_r^2{\rm d}r}\right]=\E\left[e^{2\int_0^Tb_r{\rm d}B_r-2\int_0^Tb_r^2{\rm d}r}\right]<+\infty,
\end{align*}
but the process $y_\cdot$ in \eqref{yt} does not satisfies \eqref{29sup} with $\tilde\rho=1$  since
$$
y_t=\E\left[\left.e^{2\int_0^Tb_r{\rm d}B_r-2\int_0^Tb_r^2{\rm d}r}
e^{-\int_0^t b_r{\rm d}B_r+\frac{1}{2}\int_0^t b_r^2{\rm d}r}\right|\F_t \right]=e^{\int_0^t b_r{\rm d}B_r-\frac{3}{2}\int_0^t b_r^2{\rm d}r},\ \ t\in \T
$$
and then
$$
\E\left[\sup\limits_{t\in[0,T]}\left(|y_t|^2 e^{\int_0^tb_r^2{\rm d}r}\right)\right]=\E\left[\sup\limits_{t\in[0,T]}e^{2\int_0^tb_r{\rm d}B_r-2\int_0^tb_r^2{\rm d}r}\right]=+\infty.\vspace{0.1cm}
$$

However, we can prove that for the case of $\tilde{\rho}>1$, when \eqref{betarou} holds, the process $y_\cdot$ in \eqref{yt} must satisfy \eqref{29sup}. Indeed, set $p:=\frac{\tilde{\rho}+1}{\tilde{\rho}}\in (1,2)$ and $q:=\tilde{\rho}+1>2$ satisfying $\frac{1}{p}+\frac{1}{q}=1$. By \eqref{yt*}, H\"{o}lder's inequality and \eqref{23}, we deduce that for each $t\in\T$,
\begin{align*}
&y_te^{\tilde{\beta}\int_0^t\mu_r{\rm d}r+\frac{\tilde{\rho}}{2}\int_0^t\nu_r^2{\rm d}r}=\E\left[\xi e^{\tilde{\beta}\int_0^T\mu_r{\rm d}r+\frac{\tilde{\rho}}{2}\int_0^T\nu_r^2{\rm d}r}e^{(1-\tilde{\beta})\int_t^T\mu_r{\rm d}r+\int_t^T\nu_r{\rm d}B_r-\frac{\tilde{\rho}+1}{2}\int_t^T\nu_r^2{\rm d}r}\bigg|\F_t\right]\\
&\ \ \leq \left(\E\left[|\xi|^p e^{p\tilde{\beta}\int_0^T\mu_r{\rm d}r+\frac{p\tilde{\rho}}{2}\int_0^T\nu_r^2{\rm d}r}\bigg|\F_t\right]\right)^{\frac{1}{p}} \left(\E\left[e^{q(1-\tilde{\beta})\int_t^T\mu_r{\rm d}r+q\int_t^T\nu_r{\rm d}B_r-\frac{q(\tilde{\rho}+1)}{2}\int_t^T\nu_r^2{\rm d}r}\bigg|\F_t\right]\right)^{\frac{1}{q}}\\
&\ \ \leq \left(\E\left[|\xi|^p e^{p\tilde{\beta}\int_0^T\mu_r{\rm d}r+\frac{p\tilde{\rho}}{2}\int_0^T\nu_r^2{\rm d}r}\bigg|\F_t\right]\right)^{\frac{1}{p}}
\left(\E\left[e^{q\int_t^T\nu_r{\rm d}B_r-\frac{q^2}{2}\int_t^T\nu_r^2{\rm d}r}\bigg|\F_t\right]\right)^{\frac{1}{q}}\\
&\ \ \leq \left(\E\left[|\xi|^p e^{p\tilde{\beta}\int_0^T\mu_r{\rm d}r+\frac{p\tilde{\rho}}{2}\int_0^T\nu_r^2{\rm d}r}\bigg|\F_t\right]\right)^{\frac{1}{p}}.
\end{align*}
Then, in view of $\frac{2}{p}>1$, combining the last inequality and Doob's inequality, we obtain
\begin{align*}
\begin{split}
\E\left[\sup\limits_{t\in[0,T]}\left(|y_t|^2e^{2\tilde{\beta}\int_0^t\mu_r{\rm d}r+\tilde{\rho}\int_0^t\nu_r^2{\rm d}r}\right)\right] &\leq \E\left[\sup\limits_{t\in[0,T]}\left(\E\left[|\xi|^p e^{p\tilde{\beta}\int_0^T\mu_r{\rm d}r+\frac{p\tilde{\rho}}{2}\int_0^T\nu_r^2{\rm d}r}\bigg|\F_t\right]\right)^\frac{2}{p}\right]\\
&\leq \E\left[|\xi|^2 e^{2\tilde{\beta}\int_0^T\mu_r{\rm d}r+\tilde{\rho}\int_0^T\nu_r^2{\rm d}r}\right].
\end{split}
\end{align*}
Consequently, for the case of $\tilde\beta\geq 1$ and $\tilde\rho>1$, \eqref{29sup} holds when \eqref{betarou} is satisfied. More specifically, we have verified that if $\xi\in L_T^2(\tilde{\beta}\mu_\cdot+\frac{\tilde{\rho}}{2}\nu_\cdot^2;\R)$ for some $\tilde{\beta}\geq 1$ and $\tilde{\rho}>1$, then the linear BSDE \eqref{linearBSDE} admits an adapted solution $(y_\cdot,z_\cdot)$ such that the process $y_\cdot$ belongs to the weighted space of $ S_T^2(\tilde{\beta}\mu_\cdot+\frac{\tilde{\rho}}{2}\nu_\cdot^2;\R)$. This inspires the research of this paper on the unbounded solution of a general nonlinear BSDE \eqref{BSDE1.1} with unbounded stochastic coefficients $\mu_\cdot$ and $\nu_\cdot$ in the weighted space with a weighted factor $e^{{\beta}\int_0^t\mu_r{\rm d}r+\frac{\rho}{2}\int_0^t\nu_r^2{\rm d}r}$ for $\beta\geq1$ and $\rho>1$.

\subsection{An a priori estimate}

In this subsection, we establish the following a priori estimate for weighted $L^2$-solutions of BSDEs with random terminal time. This result improves Proposition 3.2 in \cite{Briand2003} and Proposition 2.1 in \cite{Xiao2015}, and will play a vital role in the proof of our main result in this study.

\begin{pro}\label{pro:1.1}
Let $\bar\mu_\cdot$ and $\bar\nu_\cdot$ be two $(\F_t)$-progressively measurable nonnegative processes such that
$$\int_{0}^{\tau}\bar{a}_t{\rm d}t<+\infty$$
with $\overline{a}_t:=\beta\bar\mu_t+\frac{\rho}{2} \bar\nu_t^2$, and $f_\cdot$ be an $(\F_t)$-progressively measurable nonnegative process such that
$$
\E\left[\left(\int_{0}^{\tau}e^{ \int_{0}^{t}\overline{a}_r{\rm d}r}f_t{\rm d}t\right)^2\right]<+\infty.
$$
Assume that $\xi\in L_\tau^2(\overline{a}_\cdot;\R^k)$, the generator $g$ satisfies the following assumption
\begin{enumerate}
\renewcommand{\theenumi}{(A)}
\renewcommand{\labelenumi}{\theenumi}
\item\label{A:A} $\forall(y,z)\in\R^k\times\R^{k\times{d}}, \ \left<\hat{y},g(\omega,t,y,z)\right>\leq f_{t}(\omega)+\bar\mu_t(\omega)|y|+\bar\nu_t(\omega)|z|, \ t\in[0,\tau],$
\end{enumerate}
and $(Y_t,Z_t)_{t\in[0,\tau]}$ is a solution of BSDE \eqref{BSDE1.1}. If $Y_\cdot\in S_\tau^{2}(\overline{a}_\cdot;\R^k)$, then $Z_\cdot\in M_\tau^2(\overline{a}_\cdot;\R^{k\times d})$ and for each $1<\overline{\rho}\leq\rho$ and $0\leq r\leq t<+\infty$, we have
\begin{align}\label{2.02}
\begin{split}
&\E\left[\int_{t\wedge\tau}^{\tau}e^{2\int_{0}^{s}\overline{a}_{r}{\rm d}r}|Z_s|^2{\rm d}s\bigg|\F_{r\wedge\tau}\right]\\
&\ \ \leq
\frac{2\overline{\rho}}{\overline{\rho}-1}\left(\E\left[\sup_{s\in[t\wedge\tau,\tau]}\left(e^{2 \int_{0}^{s}\overline{a}_r{\rm d}r}|Y_s|^2\right)\bigg|\F_{r\wedge\tau}\right] +\E\left[\left(\int_{t\wedge\tau}^{\tau}e^{ \int_{0}^{s}\overline{a}_r{\rm d}r}f_s{\rm d}s\right)^2\bigg|\F_{r\wedge\tau}\right]\right)
\end{split}
\end{align}
and
\begin{align}\label{2.03*}
\begin{split}
&\E\left[\sup_{s\in[t\wedge\tau,\tau]}\left(e^{2\int_{0}^{s}\overline{a}_{r}{\rm d}r}|Y_s|^2\right)\bigg|\F_{r\wedge\tau}\right]+ \E\left[\int_{t\wedge\tau}^{\tau}e^{2 \int_{0}^{s}\overline{a}_r{\rm d}r}|Z_s|^2{\rm d}s\bigg|\F_{r\wedge\tau}\right]\\
&\hspace*{0.6cm}+\E\left[\int_{t\wedge\tau}^{\tau}e^{2 \int_{0}^{s}\overline{a}_r{\rm d}r}\left((2\beta-2)\bar\mu_s+(\rho-\overline{\rho})\bar\nu_s^2\right)|Y_s|^2{\rm d}s\bigg|\F_{r\wedge\tau}\right]\\
&\ \  \leq 4\left(2+\frac{33\overline{\rho}}{\overline{\rho}-1}\right)^2\left(\E\left[e^{2\int_{0}^{\tau}\overline{a}_{r}{\rm d}r}|\xi|^2\bigg|\F_{r\wedge\tau}\right]+\E\left[\left(
\int_{t\wedge\tau}^{\tau}e^{ \int_{0}^{s}\overline{a}_r{\rm d}r}f_s{\rm d}s\right)^2\bigg|\F_{r\wedge\tau}\right]\right).
\end{split}
\end{align}
In particular, there exists a uniform constant $C>0$ such that
\begin{align}\label{2.03}
\begin{split}
&\E\left[\sup_{s\in[t\wedge\tau,\tau]}\left(e^{2\int_{0}^{s}\overline{a}_{r}{\rm d}r}|Y_s|^2\right)\bigg|\F_{r\wedge\tau}\right]+ \E\left[\int_{t\wedge\tau}^{\tau}e^{2 \int_{0}^{s}\overline{a}_r{\rm d}r}|Z_s|^2{\rm d}s\bigg|\F_{r\wedge\tau}\right]\\
&\ \ \leq C\left(\E\left[e^{2\int_{0}^{\tau}\overline{a}_{r}{\rm d}r}|\xi|^2\bigg|\F_{r\wedge\tau}\right]+\E\left[\left(
\int_{t\wedge\tau}^{\tau}e^{ \int_{0}^{s}\overline{a}_r{\rm d}r}f_s{\rm d}s\right)^2\bigg|\F_{r\wedge\tau}\right]\right)
\end{split}
\end{align}
and
\begin{align}\label{2.04}
\begin{split}
&\E\left[\sup_{s\in[t\wedge\tau,\tau]}\left(e^{2 \int_{0}^{s}\overline{a}_r{\rm d}r}|Y_s|^2\right)\bigg|\F_{r\wedge\tau}\right]+
\E\left[\int_{t\wedge\tau}^{\tau}e^{2 \int_{0}^{s}\overline{a}_r{\rm d}r}|Z_s|^2{\rm d}s\bigg|\F_{r\wedge\tau}\right]\\
&\ \ \leq C\left(\E\left[e^{2 \int_{0}^{\tau}\overline{a}_r{\rm d}r}|\xi|^2\bigg|\F_{r\wedge\tau}\right]+\E\left[
\int_{t\wedge\tau}^{\tau}e^{2 \int_{0}^{s}\overline{a}_r{\rm d}r}|Y_s|f_s{\rm d}s\bigg|\F_{r\wedge\tau}\right]\right).
\end{split}
\end{align}
\end{pro}

\begin{proof}
For each integer $n\geq1$, define the following $(\F_t)$-stopping time
$$\tau_{n}:=\inf \left\{t\geq0: \int_{0}^{t} e^{2 \int_{0}^{s}\overline{a}_r{\rm d}r}\overline{a}_s|Y_s|^{2}{\rm d}s+\int_{0}^{t}e^{2 \int_{0}^{s}\overline{a}_r{\rm d}r}|Z_{s}|^{2} {\rm d}s \geq n\right\} \wedge \tau,$$
with convention that $\inf \emptyset=+\infty$.

Applying It\^{o}'s formula to $|Y_t|^2e^{2\int_{0}^{t}\overline{a}_r{\rm d}r}$ yields that for each $t\geq0$ and $n\geq1$,
\begin{align}\label{22.4}
\begin{split}
&|Y_{t\wedge\tau_n}|^2e^{2 \int_{0}^{t\wedge\tau_n}\overline{a}_r{\rm d}r}+\int_{t\wedge\tau_n}^{\tau_n} e^{2 \int_{0}^{s}\overline{a}_r{\rm d}r}|Z_s|^2{\rm d}s+2 \int_{t\wedge\tau_n}^{\tau_n} e^{2 \int_{0}^{s}\overline{a}_r{\rm d}r}\overline{a}_s|Y_s|^2{\rm d}s\\
&\ \ =|Y_{\tau_n}|^2e^{2 \int_{0}^{\tau_n}\overline{a}_r{\rm d}r}+2\int_{t\wedge\tau_n}^{\tau_n} e^{2 \int_{0}^{s}\overline{a}_r{\rm d}r}\langle Y_s,g(s,Y_s,Z_s)\rangle{\rm d}s-2\int_{t\wedge\tau_n}^{\tau_n} e^{2 \int_{0}^{s}\overline{a}_r{\rm d}r}\langle Y_s,Z_s{\rm d}B_s\rangle.
\end{split}
\end{align}
In light of assumption \ref{A:A} and inequality $2ab\leq \overline{\rho} a^2+\frac{1}{\overline{\rho}}b^2$, we have
\begin{align}\label{2.2}
\begin{split}
2\left<Y_t,g(t,Y_t,Z_t)\right>
\leq&\  2\bar\mu_{t}|Y_t|^2+2\bar\nu_t|Y_t||Z_t|+2|Y_t|f_{t}\\
\leq&\  2\bar\mu_{t}|Y_t|^2+\overline{\rho} \bar\nu_{t}^2|Y_t|^2+\frac{1}{\overline{\rho}}|Z_t|^2+2|Y_t|f_{t},  \ t\in[0,\tau_{n}].
\end{split}
\end{align}
It follows from the Burkholder-Davis-Gundy (BDG in short) inequality that for each $n\geq1$,
$$\left(\int_{0}^{t\wedge\tau_n} e^{2 \int_{0}^{s}\overline{a}_r{\rm d}r}\langle Y_s,Z_s{\rm d}B_s\rangle\right)_{t\geq0}$$
is a uniformly integrable martingale. Indeed, by Theorem 1 in \cite{Ren2008BDG} we have that for each $n\geq1$,
\begin{align}\label{2.66}
\begin{split}
&2\E\left[\sup_{t\geq0}\left|\int_0^{t\wedge\tau_n} e^{2 \int_{0}^{s}\overline{a}_r{\rm d}r}\langle Y_s,Z_s{\rm d}B_s\rangle\right|\right]
\leq 4\sqrt{2}\E\left[\left(\int_0^{\tau_n} e^{4 \int_{0}^{s}\overline{a}_r{\rm d}r}|Y_s|^2|Z_s|^2{\rm d}s\right)^{\frac{1}{2}}\right]\\
&\ \ \leq  \frac{1}{2}\E\left[\sup_{s\in[0,\tau_n]}\left(e^{2 \int_{0}^{s}\overline{a}_r{\rm d}r}|Y_s|^2\right)\right]+16\E\left[\int_0^{\tau_n} e^{2 \int_{0}^{s}\overline{a}_r{\rm d}r}|Z_s|^2{\rm d}s\right]<+\infty.
\end{split}
\end{align}
Then, in view of \eqref{2.2} and the fact of $1<\overline{\rho}\leq\rho$, by taking the  conditional mathematical expectation with respect to $\F_{r\wedge\tau_m}$ in both sides of \eqref{22.4} and using the inequality $2ab\leq a^2+b^2$ we deduce that for each $0\leq r\leq t<+\infty$
and $n\geq m\geq1$,
\begin{align}\label{0032.2}
\begin{split}
&\left(1-\frac{1}{\overline{\rho}}\right)\E\left[\int_{t\wedge\tau_n}^{\tau_n}e^{2 \int_{0}^{s}\overline{a}_r{\rm d}r}|Z_s|^2{\rm d}s\bigg|\F_{r\wedge\tau_m}\right]\\
&\hspace*{0.6cm}+\E\left[\int_{t\wedge\tau_n}^{\tau_n} e^{2 \int_{0}^{s}\overline{a}_r{\rm d}r}\left((2\beta-2)\bar\mu_s+(\rho-\overline{\rho})\bar\nu_s^2\right)|Y_s|^2{\rm d}s\bigg|\F_{r\wedge\tau_m}\right]\\
&\ \ \leq
\E\left[\sup_{s\in[t\wedge\tau_n,\tau_n]}\left(e^{2 \int_{0}^{s}\overline{a}_r{\rm d}r}|Y_s|^2\right)\bigg|\F_{r\wedge\tau_m}\right]+2\E\left[\int_{t\wedge\tau_n}^{\tau_n}e^{2 \int_{0}^{s}\overline{a}_r{\rm d}r}|Y_s|f_s{\rm d}s\bigg|\F_{r\wedge\tau_m}\right]\\
&\ \ \leq
2\E\left[\sup_{s\in[t\wedge\tau_n,\tau_n]}\left(e^{2 \int_{0}^{s}\overline{a}_r{\rm d}r}|Y_s|^2\right)\bigg|\F_{r\wedge\tau_m}\right]+\E\left[\left(\int_{t\wedge\tau_n}^{\tau_n}e^{ \int_{0}^{s}\overline{a}_r{\rm d}r}f_s{\rm d}s\right)^2\bigg|\F_{r\wedge\tau_m}\right].
\end{split}
\end{align}
Letting $n\rightarrow \infty$ and using Fatou's lemma in both sides of the last inequality yields that for each $0\leq r\leq t<+\infty$ and $m\geq1$,
\begin{align}\label{0032.2*}
\begin{split}
&\left(1-\frac{1}{\overline{\rho}}\right)\E\left[\int_{t\wedge\tau}^{\tau}e^{2 \int_{0}^{s}\overline{a}_r{\rm d}r}|Z_s|^2{\rm d}s\bigg|\F_{r\wedge\tau_m}\right]\\
&\hspace*{0.6cm}+\E\left[\int_{t\wedge\tau}^{\tau} e^{2 \int_{0}^{s}\overline{a}_r{\rm d}r}\left((2\beta-2)\bar\mu_s+(\rho-\overline{\rho})\bar\nu_s^2\right)|Y_s|^2{\rm d}s\bigg|\F_{r\wedge\tau_m}\right]\\
&\ \ \leq
\E\left[\sup_{s\in[t\wedge\tau,\tau]}\left(e^{2 \int_{0}^{s}\overline{a}_r{\rm d}r}|Y_s|^2\right)\bigg|\F_{r\wedge\tau_m}\right]+2\E\left[\int_{t\wedge\tau}^{\tau}e^{2 \int_{0}^{s}\overline{a}_r{\rm d}r}|Y_s|f_s{\rm d}s\bigg|\F_{r\wedge\tau_m}\right]\\
&\ \ \leq
2\E\left[\sup_{s\in[t\wedge\tau,\tau]}\left(e^{2 \int_{0}^{s}\overline{a}_r{\rm d}r}|Y_s|^2\right)\bigg|\F_{r\wedge\tau_m}\right]+\E\left[\left(\int_{t\wedge\tau}^{\tau}e^{ \int_{0}^{s}\overline{a}_r{\rm d}r}f_s{\rm d}s\right)^2\bigg|\F_{r\wedge\tau_m}\right].
\end{split}
\end{align}
Thus, since $Y_\cdot\in S_\tau^{2}(\overline{a}_\cdot;\R^k)$, the desired assertion \eqref{2.02} follows by sending $m\rightarrow\infty$ and using the martingale convergence theorem (see Corollary A.9 in Appendix C of \cite{Oksendal2005}) in both sides of \eqref{0032.2*}, and then $Z_\cdot\in M_\tau^{2}(\overline{a}_\cdot;\R^{k\times d})$.

Furthermore, since both \eqref{22.4} and \eqref{2.2} are also true on $[t\wedge\tau,\tau], 0\leq t<+\infty$, we have
\begin{align}\label{08}
\begin{split}
&|Y_{t\wedge\tau}|^2e^{2 \int_{0}^{t}\overline{a}_r{\rm d}r}+\left(1-\frac{1}{\overline{\rho}}\right) \int_{t\wedge\tau}^\tau e^{2 \int_{0}^{s}\overline{a}_r{\rm d}r}|Z_s|^2{\rm d}s\\
&\hspace*{0.6cm}+\int_{t\wedge\tau}^\tau e^{2 \int_{0}^{s}\overline{a}_r{\rm d}r}\left((2\beta-2)\bar\mu_s+(\rho-\overline{\rho})\bar\nu_s^2\right)|Y_s|^2{\rm d}s\\
&\ \ \leq |\xi|^2e^{2 \int_{0}^{\tau}\overline{a}_r{\rm d}r}+2\int_{t\wedge\tau}^\tau e^{2 \int_{0}^{s}\overline{a}_r{\rm d}r}|Y_s|f_s{\rm d}s-2\int_{t\wedge\tau}^\tau e^{2 \int_{0}^{s}\overline{a}_r{\rm d}r}\langle Y_s,Z_s{\rm d}B_s\rangle.
\end{split}
\end{align}
It follows from the BDG inequality that
$$\left(\int_{0}^{t\wedge\tau} e^{2 \int_{0}^{s}\overline{a}_r{\rm d}r}\langle Y_s,Z_s{\rm d}B_s\rangle\right)_{t\geq0}$$ is a uniformly integrable martingale. And, since $(Y_\cdot,Z_\cdot)\in H_\tau^{2}(\overline{a}_\cdot;\R^k\times\R^{k\times d})$, by virtue of Theorem 1 in \cite{Ren2008BDG}, we know that for each $0\leq r\leq t<+\infty$,
\begin{align}\label{2.66}
\begin{split}
&2\E\left[\sup_{s\in[t\wedge\tau,\tau]}\left|\int_{s\wedge\tau}^\tau e^{2 \int_{0}^{s}\overline{a}_r{\rm d}r}\langle Y_s,Z_s{\rm d}B_s\rangle\right|\bigg|\F_{r\wedge\tau}\right]\\
&\ \ \leq  4\sqrt{2}\E\left[\left(\int_{t\wedge\tau}^\tau e^{4 \int_{0}^{s}\overline{a}_r{\rm d}r}|Y_s|^2|Z_s|^2{\rm d}s\right)^{\frac{1}{2}}\bigg|\F_{r\wedge\tau}\right]\\
&\ \ \leq  \frac{1}{2}\E\left[\sup_{s\in[t\wedge\tau,\tau]} \left(e^{2 \int_{0}^{s}\overline{a}_r{\rm d}r}|Y_s|^2\right)\bigg|\F_{r\wedge\tau}\right]+16\E\left[\int_{t\wedge\tau}^\tau e^{2 \int_{0}^{s}\overline{a}_r{\rm d}r}|Z_s|^2{\rm d}s\bigg|\F_{r\wedge\tau}\right]<+\infty.
\end{split}
\end{align}
Then, in view of the last inequality, by taking supremum with respect to $s$ and the conditional mathematical expectation in both sides of inequality \eqref{08} we obtain that for each $0\leq r\leq t<+\infty$,
\begin{align}\label{008}
\begin{split}
\E\left[\int_{t\wedge\tau}^\tau e^{2 \int_{0}^{s}\overline{a}_r{\rm d}r}|Z_s|^2{\rm d}s\bigg|\F_{r\wedge\tau}\right]\leq \frac{\overline{\rho}}{\overline{\rho}-1}\E\left[|\xi|^2e^{2 \int_{0}^{\tau}\overline{a}_r{\rm d}r}+2\int_{t\wedge\tau}^\tau e^{2 \int_{0}^{s}\overline{a}_r{\rm d}r}|Y_s|f_s{\rm d}s\bigg|\F_{r\wedge\tau}\right]
\end{split}
\end{align}
and
\begin{align}\label{0008}
\begin{split}
&\frac{1}{2}\E\left[\sup_{s\in[t\wedge\tau,\tau]}\left(e^{2 \int_{0}^{s}\overline{a}_r{\rm d}r}|Y_s|^2\right)\bigg|\F_{r\wedge\tau}\right]\\
&\hspace*{0.6cm}+\E\left[\int_{t\wedge\tau}^{\tau} e^{2 \int_{0}^{s}\overline{a}_r{\rm d}r}\left((2\beta-2)\bar\mu_s+(\rho-\overline{\rho})\bar\nu_s^2\right)|Y_s|^{2}{\rm d}s\bigg|\F_{r\wedge\tau}\right]\\
&\ \ \leq  \E\left[|\xi|^2e^{2 \int_{0}^{\tau}\overline{a}_r{\rm d}r}+2\int_{t\wedge\tau}^\tau e^{2 \int_{0}^{s}\overline{a}_r{\rm d}r}|Y_s|f_s{\rm d}s\bigg|\F_{r\wedge\tau}\right]+16\E\left[\int_{t\wedge\tau}^\tau e^{2 \int_{0}^{s}\overline{a}_r{\rm d}r}|Z_s|^2{\rm d}s\bigg|\F_{r\wedge\tau}\right].
\end{split}
\end{align}
Combining \eqref{008} and \eqref{0008} we deduce that for each $0\leq r\leq t<+\infty$,
\begin{align}\label{2.013}
&\E\left[\sup_{s\in[t\wedge\tau,\tau]}\left(e^{2 \int_{0}^{s}\overline{a}_r{\rm d}r}|Y_s|^2\right)\bigg|\F_{r\wedge\tau}\right] +\E\left[\int_{t\wedge\tau}^\tau e^{2 \int_{0}^{s}\overline{a}_r{\rm d}r}|Z_s|^2{\rm d}s\bigg|\F_{r\wedge\tau}\right]\nonumber\\ &\hspace{0.6cm}+\E\left[\int_{t\wedge\tau}^{\tau} e^{2 \int_{0}^{s}\overline{a}_r{\rm d}r}\left((2\beta-2)\bar\mu_s+(\rho-\overline{\rho})\bar\nu_s^2\right)|Y_s|^{2}{\rm d}s\bigg|\F_{r\wedge\tau}\right]\nonumber\\
&\ \ \leq \left(2+\frac{33\overline{\rho}}{\overline{\rho}-1}\right)\left(\E\left[|\xi|^2e^{2 \int_{0}^{\tau}\overline{a}_r{\rm d}r}\bigg|\F_{r\wedge\tau}\right] +2\E\left[\int_{t\wedge\tau}^\tau e^{2 \int_{0}^{s}\overline{a}_r{\rm d}r}|Y_s|f_s{\rm d}s\bigg|\F_{r\wedge\tau}\right]\right).
\end{align}
Hence, \eqref{2.04} comes true.

Finally, by virtue of inequality $2ab\leq \frac{1}{2}a^2+2b^2$, we have that for each $0\leq r\leq t<+\infty$,
\begin{align*}
&2\left(2+\frac{33\overline{\rho}}{\overline{\rho}-1}\right)\E\left[\int_{t\wedge\tau}^\tau e^{2 \int_{0}^{s}\overline{a}_r{\rm d}r}|Y_s|f_s{\rm d}s\bigg|\F_{r\wedge\tau}\right]\nonumber\\
&\ \ \leq \frac{1}{2}\E\left[\sup_{s\in[t\wedge\tau,\tau]}\left(e^{2 \int_{0}^{s}\overline{a}_r{\rm d}r}|Y_s|^2\right)\bigg|\F_{r\wedge\tau}\right]
+2\left(2+\frac{33\overline{\rho}}{\overline{\rho}-1}\right)^2\E\left[\left(\int_{t\wedge\tau}^\tau e^{\int_{0}^{s}\overline{a}_r{\rm d}r}f_s{\rm d}s\right)^2\bigg|\F_{r\wedge\tau}\right],
\end{align*}
from which together with \eqref{2.013} the desired results \eqref{2.03*} and \eqref{2.03} follow immediately.
\end{proof}

\begin{rmk}\label{rmk:222}
From the above proof, it can be observed that the restrictive condition of $\beta\geq 1$ and $\rho>1$ is necessary in order to obtain the desired estimate in \cref{pro:1.1}. Furthermore, by letting $\beta>1$ and $\overline{\rho}=\frac{1+\rho}{2}$ in \cref{pro:1.1}, from \eqref{2.03*} we can deduce that
\begin{align}\label{1}
\E\left[\int_{0}^{\tau} e^{2 \int_{0}^{s}\overline{a}_r{\rm d}r}\overline{a}_s|Y_s|^{2}{\rm d}s\right]<+\infty,
\end{align}
which is required in \cite{BenderKohlmann2000} and \cite{Li2023}. However, generally speaking, the last assertion \eqref{1} does not hold for $\beta=1$. Therefore, some relevant arguments in \cite{BenderKohlmann2000} and \cite{Li2023} are inavailable for the case of $\beta=1$.
\end{rmk}

%

\section{Main result}
\setcounter{equation}{0}

In this section, we will establish an existence and uniqueness theorem, a continuous dependence property and a stability theorem for the weighted $L^2$-solution of BSDE \eqref{BSDE1.1}, where the generator $g$ satisfies the stochastic monotonicity condition with general growth in the first unknown variable $y$ and the stochastic Lipschitz continuity condition in the second unknown variable $z$. Also, we will give some examples and remarks to illustrate that these results strengthen some existing works. Finally, we derive the nonlinear Feynman-Kac formulas in our framework about parabolic PDEs and elliptic PDEs.

\subsection{Existence and uniqueness\vspace{0.2cm}}

Recalling $\mu_\cdot$ and $\nu_\cdot$ are two given $(\F_t)$-progressively measurable nonnegative processes with $a_t:=\beta\mu_t+\frac{\rho}{2}\nu_t^2$ satisfying $\int_0^\tau a_t{\rm d}t<+\infty$. Let us first introduce the following assumptions on the generator $g:\Omega\times[0,\tau]\times \R^k\times \R^{k\times d}\mapsto \R^k$.

\begin{enumerate}
\renewcommand{\theenumi}{(H1)}
\renewcommand{\labelenumi}{\theenumi}
\item\label{A:H1} $\Dis \E\left[\left(\int_0^\tau e^{\int_{0}^{s}a_r{\rm d}r}|g(s,0,0)|{\rm d}s\right)^2\right]<+\infty.$
\renewcommand{\theenumi}{(H2)}
\renewcommand{\labelenumi}{\theenumi}
 \item\label{A:H2} ${\rm d}\mathbb{P}\times{\rm d} t-a.e.$, $g(\omega,t,\cdot,z)$ is continuous for each $z\in\R^{k\times d}$.
\renewcommand{\theenumi}{(H3)}
\renewcommand{\labelenumi}{\theenumi}
\item\label{A:H3} $g$ has a general growth in $y$, i.e., there exists an $(\F_t)$-progressively measurable non-increasing process $(\alpha_t)_{t\in[0,\tau]}$ taking values in $(0,1]$ such that for each $r\in \R_+$, it holds that
\begin{align}\label{3.27*}
\E\left[\int_0^\tau e^{\beta \int_{0}^{t}\mu_s{\rm d}s}\psi_{r}^{\alpha_\cdot}(t)dt\right]<+\infty
\end{align}
with
$$\psi_{r}^{\alpha_\cdot}(t):=\sup_{|y|\leq r\alpha_t}\left|g(t,y,0)-g(t,0,0)\right|, \ t\in[0,\tau].$$
\renewcommand{\theenumi}{(H4)}
\renewcommand{\labelenumi}{\theenumi}
\item\label{A:H4} $g$ satisfies a stochastic monotonicity condition in $y$, i.e., for each $y_1, \ y_2\in\R^k$ and $z\in\R^{k\times d}$,
$$\left\langle y_1-y_2,g(\omega,t,y_1,z)-g(\omega,t,y_2,z)\right\rangle\leq \mu_t(\omega)|y_1-y_2|^2, \ t\in[0,\tau].$$
\renewcommand{\theenumi}{(H5)}
\renewcommand{\labelenumi}{\theenumi}
\item\label{A:H5} $g$ satisfies a stochastic Lipschitz continuity condition in $z$, i.e., for each $y\in\R^k$ and $z_1,z_2\in\R^{k\times d}$,
$$\left|g(\omega,t,y,z_1)-g(\omega,t,y,z_2)\right|\leq \nu_t(\omega)|z_1-z_2|, \ t\in[0,\tau].$$
\end{enumerate}

The following theorem is the main result of this subsection establishing a general existence and uniqueness result of the weighted $L^2$-solutions for BSDEs under the above assumptions \ref{A:H1}-\ref{A:H5}.

\begin{thm}[\bf Existence and uniqueness]\label{thm:3.1}
Let $\xi\in L_\tau^2(a_\cdot;\R^k)$ and the generator $g$ satisfy \ref{A:H1}-\ref{A:H5}. Then, BSDE \eqref{BSDE1.1} admits a unique weighted $L^2$-solution $(y_t,z_t)_{t\in[0,\tau]}$ in $H_\tau^2(a_\cdot;\R^{k}\times\R^{k\times d})$.

\begin{proof}[\bf Proof]
The proof of existence part is postponed in Section 4 due to its complexity. Here we only prove the uniqueness part. Assume that $(y_\cdot,z_\cdot)$ and $(y_\cdot',z_\cdot')$ are two  weighted $L^2$-solutions in $H_\tau^2(a_\cdot;\R^{k}\times\R^{k\times d})$ of BSDE \eqref{BSDE1.1}. It is obvious that $(\tilde{y}_\cdot,\tilde{z}_\cdot):=(y_\cdot-y_\cdot',z_\cdot-z_\cdot')$ is a weighted $L^2$-solution in $H_\tau^2(a_\cdot;\R^{k}\times\R^{k\times d})$ to the following BSDE:
\begin{align}\label{2.5}
\tilde{y}_t=\int_t^\tau\tilde{g}(s,\tilde{y}_s,\tilde{z}_s){\rm d}s-\int_t^\tau\tilde{z}_s{\rm d}B_s, \ t\in[0,\tau],
\end{align}
where for each $(y,z)\in\R^k\times\R^{k\times{d}}$, $$\tilde{g}(t,y,z)=g(t,y+y_t',z+z_t')-g(t,y_t',z_t'), \ t\in[0,\tau].$$
It follows from \ref{A:H4} and \ref{A:H5} that for each $(y,z)\in\R^k\times\R^{k\times{d}}$,
\begin{align*}
\begin{split}
\left<\hat{y},\tilde{g}(t,y,z)\right>&=\left<\hat{y},g(t,y+y_t',z+z_t')-g(t,y_t',z+z_t')+g(t,y_t',z+z_t')-g(t,y_t',z_t')\right>\\
&\leq \mu_t|y|+\nu_t|z|, \ t\in[0,\tau].
\end{split}
\end{align*}
This implies that assumption (A) is satisfied by the generator $\tilde{g}$ with $\bar\mu_\cdot=\mu_\cdot$, $\bar\nu_\cdot=\nu_\cdot$ and $f_\cdot\equiv0$. Thus, it follows from \eqref{2.03} in \cref{pro:1.1} with $r=t=0$ that
\begin{align*}
\begin{split}
&\E\left[\sup_{s\in[0,\tau]}\left(e^{2 \int_{0}^{s}(\beta\mu_r+\frac{\rho}{2} \nu_r^2){\rm d}r}|\tilde{y}_s|^2\right)\right]+\E\left[\int_{0}^{\tau}e^{2 \int_{0}^{s}(\beta\mu_r+\frac{\rho}{2}\nu_r^2){\rm d}r}|\tilde{z}_s|^2{\rm d}s\right]=0.
\end{split}
\end{align*}
Therefore, $(\tilde{y}_t,\tilde{z}_t)_{t\in[0,\tau]}=(0,0)$. The proof of the uniqueness part is then completed.
\end{proof}
\end{thm}

The following \cref{cor:3.2} is a direct consequence of \cref{thm:3.1}. It indicates that the assertion of \cref{thm:3.1} can include BSDEs with stochastic Lipshitz generators as its particular case. For this, let us introduce the following stochastic Lipschitz assumption of the generator $g$:

\begin{enumerate}
\renewcommand{\theenumi}{(SL)}
\renewcommand{\labelenumi}{\theenumi}
\item\label{A:SL} $g$ is stochastic Lipschitz continuous in $(y,z)$, i.e., for each $y_1, y_2\in\R^k$ and $z_1, z_2\in\R^{k\times d}$, we have
    $$|g(\omega,t,y_1,z_1)-g(\omega,t,y_2,z_2)|\leq \mu_t|y_1-y_2|+\nu_t|z_1-z_2|, \ t\in[0,\tau].\vspace{-0.1cm}$$
\end{enumerate}

\begin{cor}\label{cor:3.2}
Let $\xi\in L_\tau^2(a_\cdot;\R^k)$ and the generator $g$ satisfy \ref{A:H1} and \ref{A:SL}. Then, BSDE \eqref{BSDE1.1} admits a unique weighted $L^2$-solution $(y_t,z_t)_{t\in[0,\tau]}$ in the space of $H_\tau^2(a_\cdot;\R^{k}\times\R^{k\times d})$.
\end{cor}

\begin{proof}
It is clear that \ref{A:H2}, \ref{A:H4} and \ref{A:H5} are all true if assumption \ref{A:SL} is in force. Then, according to \cref{thm:3.1}, it suffices to verify that \ref{A:H3} is also true. In fact, it follows \ref{A:SL} that for each $y\in\R^k$,
\begin{align}\label{3.21*}
|g(t,y,0)-g(t,0,0)|\leq \mu_t|y|, \ t\in[0,\tau].
\end{align}
Then, by taking
$$\alpha_t:=\frac{e^{-t}}{1+\sup\limits_{0\leq s\leq t}\left(e^{\beta\int_0^s\mu_r{\rm d}r}\mu_s\right)}\in(0,1), \ t\in[0,\tau],$$
which is an $(\F_t)$-progressively measurable non-increasing process taking values in $(0,1]$, we obtain that for each $r\in \R_+$, it holds that
\begin{align*}
\begin{split}
\E\left[\int_0^\tau e^{\beta\int_0^t\mu_s{\rm d}s}\psi_{r}^{\alpha_\cdot}(t)dt\right]\leq r\E\left[\int_0^\tau e^{\beta \int_{0}^{t}\mu_s{\rm d}s}\mu_t\alpha_t dt\right]
\leq r\E\left[\int_0^\tau e^{-t}dt\right]\leq r<+\infty.
\end{split}
\end{align*}
This implies that \ref{A:H3} comes true. The proof is then complete.
\end{proof}

With respect to \cref{cor:3.2} and \cref{thm:3.1}, we would like to make the following important remark. The following three assumptions related to \ref{A:H1} and \ref{A:H3} on the generator $g$ will be used in it.\vspace{0.2cm}

\begin{enumerate}
\renewcommand{\theenumi}{(H1a)}
\renewcommand{\labelenumi}{\theenumi}
\item\label{A:H1a}
$\Dis \E\left[\int_0^\tau e^{2\theta\int_0^s (\mu_r+\frac{1}{2}\nu^2_r) {\rm d}r}\left|\frac{g(s,0,0)}{\sqrt{\mu_r+\nu^2_r}}\right|^2{\rm d}s\right]<+\infty$, where $\theta>1$ and $\mu_\cdot+\nu^2_\cdot\geq\varepsilon>0$.\vspace{0.1cm}
\end{enumerate}

\begin{enumerate}
\renewcommand{\theenumi}{(H3*)}
\renewcommand{\labelenumi}{\theenumi}
\item\label{A:H3*}
For each $r\in \R_+$, it holds that
$$
\E\left[\left(\int_0^\tau e^{\int_{0}^{t}\mu_s{\rm d}s}\psi_{r}(t)dt\right)^2\right]<+\infty
$$
with
$$\psi_{r}(t):=\sup_{|y|\leq r}\left|g(t,y,0)-g(t,0,0)\right|, \ t\in[0,\tau].\vspace{-0.1cm}$$
\end{enumerate}

\begin{enumerate}
\renewcommand{\theenumi}{(H3a)}
\renewcommand{\labelenumi}{\theenumi}
\item\label{A:H3a} There exists an $(\F_t)$-progressively measurable nonnegative process $(\tilde{\mu}_t)_{t\in[0,\tau]}$ without any integrability requirement such that for each $y\in \R^k$,
    $$|g(t,y,0)-g(t,0,0)|\leq \tilde{\mu}_t\varphi(|y|), \ t\in[0,\tau],$$
    where $\varphi(\cdot):\R_+\rightarrow \R_+$ is a nondecreasing convex function with $\varphi(0)=0$.
\end{enumerate}

\begin{rmk}\label{rmk:3.3}
(i) It is clear that \cref{cor:3.2} extends Theorem 2.1 in \cite{Li2023}, where the constants $\beta$ and $\rho$ appearing in the process $a_\cdot$ are required to satisfy that $\beta>\frac{1+\sqrt{5}}{2}$ and $\rho=2\beta$.\vspace{0.2cm}

(ii) We emphasize that \cref{cor:3.2} also extends Theorem 4.1 in \cite{O2020}, where the assumption \ref{A:H1} is replaced with the stronger assumption \ref{A:H1a}. Indeed, if assumption \ref{A:H1a} holds for some $\theta>1$, then by H\"{o}lder's inequality we have, taking $\beta:=\frac{3\theta+1}{4}>1$ and $\rho:=\frac{\theta+1}{2}>1$,
\begin{align}
&\E\left[\left(\int_0^\tau e^{\int_0^s(\beta\mu_r+\frac{\rho}{2}\nu^2_r){\rm d}r}|g(s,0,0)|{\rm d}s\right)^2\right]\nonumber\\
&\ \ = \E\left[\left(\int_0^\tau e^{\theta\int_0^s (\mu_r+\frac{1}{2}\nu^2_r) {\rm d}r}\frac{|g(s,0,0)|}{\sqrt{\mu_s+\nu^2_s}}\cdot\sqrt{\mu_s+\nu^2_s} e^{-\frac{\theta-1}{4}\int_0^s(\mu_r+\nu^2_r){\rm d}r}{\rm d}s\right)^2\right]\nonumber\\
&\ \ \leq \E\left[\int_0^\tau e^{2\theta\int_0^s (\mu_r+\frac{1}{2}\nu^2_r) {\rm d}r}\left|\frac{g(s,0,0)}{\sqrt{\mu_s+\nu^2_s}}\right|^2{\rm d}s\cdot\int_0^ \tau (\mu_s+\nu^2_s) e^{-\frac{\theta-1}{2}\int_0^s(\mu_r+\nu^2_r){\rm d}r}{\rm d}s\right]\nonumber\\
&\ \ \leq \frac{2}{\theta-1}\E\left[\int_0^\tau e^{2\theta\int_0^s (\mu_r+\frac{1}{2}\nu^2_r) {\rm d}r}\left|\frac{g(s,0,0)}{\sqrt{\mu_r+\nu^2_r}}\right|^2{\rm d}s\right]<+\infty.\nonumber
\end{align}
On the other hand, it is clear that \ref{A:H1} can not imply \ref{A:H1a} since $\mu_\cdot+\nu^2_\cdot\geq\varepsilon>0$ is not always true. Consequently, our assumption \ref{A:H1} is strictly weaker than \ref{A:H1a}. Furthermore, we mention that a stronger assumption than \ref{A:H1a} was used in Theorem 3 in \cite{BenderKohlmann2000} and some subsequent related works on BSDEs with stochastic Lipschitz generators (see for example \cite{KarouiHuang1997}, \cite{WangRanChen2007}, \cite{Wen(2011)}, \cite{HuandRen2011}, \cite{HuLanYing(2012)}, \cite{Owo(2015),Owo(2017)} and
\cite{MarzougueandEl Otmaniv2017}), where the constant $1/2$ appearing in \ref{A:H1a} is required to be replaced with $1$, and the $\theta$ is required to be sufficient large. Hence, \cref{cor:3.2} improves Theorem 3 in \cite{BenderKohlmann2000}, and other related works can be accordingly  improved. \vspace{0.2cm}

(iii) Assumption \ref{A:H3*} was usually used in the existing literature on the study of BSDEs with stochastic monotonicity generators, see for example Theorem 5.30, Theorem 5.57, and Corollary 5.59 in \cite{PardouxandRascanu(2014)}. It is clear that \ref{A:H3*} can imply \ref{A:H3} with $\beta=1$ and $\alpha_\cdot\equiv 1$, while \ref{A:H3} can not imply \ref{A:H3*}. In fact, if the generator $g$ satisfies \eqref{3.21*}, then it follows from the proof of \cref{cor:3.2} that this $g$ satisfies \ref{A:H3}, while it does not satisfy \ref{A:H3*} under the situation that
$$
\E\left[\left(\int_0^\tau e^{\int_{0}^{t}\mu_s{\rm d}s}\mu_t dt\right)^2\right]=+\infty.
$$
Consequently, our assumption \ref{A:H3} is strictly weaker than \ref{A:H3*}, and then our \cref{thm:3.1} unifies and strengthens those above-mentioned results in \cite{PardouxandRascanu(2014)}, Theorems 3.6 and 4.2 in \cite{Bahlali2004} for the cases of both constant terminal time and finite terminal time, and Theorem 5.2 in \cite{O2020} for the case of $p=2$ and finite terminal time.\vspace{0.2cm}

(iv) It can be easily checked that assumption \ref{A:H3a} extends \eqref{3.21*}, and can be regarded, in some sense, as a generalization of the corresponding assumptions c(ii) in \cite{Pardoux1999} (see \eqref{a1.5} in the introduction), (H5') in \cite{Briand2003}, (H2) in \cite{Royer2004} and (H3') in \cite{Xiao2015}. Now, we prove that assumption \ref{A:H3a} can imply \ref{A:H3}. In fact, let \ref{A:H3a} hold with $\varphi(\cdot)$ and $\tilde{\mu}_\cdot$. Since $\varphi(\cdot)$ is a convex function with $\varphi(0)=0$, we have for each $\lambda\in(0,1)$ and $r\in \R_+$,
$$\varphi(\lambda r)\leq \lambda \varphi(r).$$
Then, considering \ref{A:H3a} and the last inequality, we know that for each $r\in \R_+$,
$$\psi_{r}^{\alpha_\cdot}(t):=\sup_{|y|\leq r\alpha_t}\left\{ \left|g(t,y,0)-g(t,0,0)\right|\right\}\leq \tilde{\mu}_t\varphi(r \alpha_t)\leq\tilde{\mu}_t\alpha_t\varphi(r), \ \ t\in[0,\tau].$$
Thus, by taking
$$\alpha_t:=\frac{e^{-t}}{1+\sup\limits_{0\leq s\leq t}\left(e^{\beta\int_0^s\mu_r{\rm d}r}\tilde{\mu}_s\right)}\in(0,1), \ t\in[0,\tau],$$
which is an $(\F_t)$-progressively measurable non-increasing process taking values in $(0,1]$,
we obtain that for each $r\in \R_+$, it holds that
\begin{align*}
\begin{split}
\E\left[\int_0^\tau e^{\beta\int_0^t\mu_s{\rm d}s}\psi_{r}^{\alpha_\cdot}(t)dt\right]\leq \E\left[\int_0^\tau e^{\beta \int_{0}^{t}\mu_s{\rm d}s}\tilde{\mu}_t\alpha_t\varphi(r)dt\right]
\leq\varphi(r)\E\left[\int_0^\tau e^{-t}dt\right]\leq \varphi(r)<+\infty.
\end{split}
\end{align*}
This implies that \ref{A:H3} comes true. Consequently, our assumption \ref{A:H3} is also strictly weaker than \ref{A:H3a}, while \ref{A:H3a} is more easily to be verified than \ref{A:H3}, and will be used several times in subsequent subsection. Notably, \cref{thm:3.1} strengthens and improves Theorem 3.4 in \cite{DarlingandPardoux1997} and Theorem 4.1 in \cite{Pardoux1999} for the case of finite terminal time, where the processes $\mu_\cdot$ and $\nu_\cdot$ are two constants independent of $(t,\omega)$.
\end{rmk}

In order to compare perfectly \cref{thm:3.1} with some related existing results, let us further introduce the following assumption of the generator $g$ and verify another corollary of \cref{thm:3.1}.\vspace{0.2cm}

\begin{enumerate}
\renewcommand{\theenumi}{(H1b)}
\renewcommand{\labelenumi}{\theenumi}
\item\label{A:H1b} $\Dis \E\left[\left(\int_0^\tau |g(s,0,0)|{\rm d}s\right)^2\right]<+\infty.$
\renewcommand{\theenumi}{(H3b)}
\renewcommand{\labelenumi}{\theenumi}
\item\label{A:H3b} There exists an $(\F_t)$-progressively measurable non-increasing process $(\alpha_t)_{t\in[0,\tau]}$ taking values in $(0,1]$ such that for each $r\in \R_+$, it holds that
$$\E\left[\int_0^\tau \psi_{r}^{\alpha_\cdot}(t)dt\right]<+\infty$$
with
$$\psi_{r}^{\alpha_\cdot}(t):=\sup_{|y|\leq r\alpha_t} \left|g(t,y,0)-g(t,0,0)\right|, \ t\in[0,\tau].$$
\renewcommand{\theenumi}{(H3b*)}
\renewcommand{\labelenumi}{\theenumi}
\item\label{A:H3b*}For each $r\in \R_+$, we have
$$\E\left[\int_0^\tau \psi_{r}(t)dt\right]<+\infty$$
with
$$\psi_{r}(t):=\sup_{|y|\leq r} \left|g(t,y,0)-g(t,0,0)\right|, \ t\in[0,\tau].\vspace{0.2cm}$$
\end{enumerate}

\begin{cor}\label{cor:3.4}
Let $\xi\in L_\tau^2(0;\R^k)$ and the generator $g$ satisfy \ref{A:H1b}, \ref{A:H2}, \ref{A:H3b}, \ref{A:H4} and \ref{A:H5} with
\begin{align}\label{BoundM}
\int_0^\tau\left(\mu_t+\nu_t^2\right){\rm d}t\leq M
\end{align}
for some constant $M>0$. Then, BSDE \eqref{BSDE1.1} admits a unique solution $(y_t,z_t)_{t\in[0,\tau]}$ in $H_\tau^2(0;\R^{k}\times\R^{k\times d})$.
\end{cor}

\begin{proof}
Note that under the condition \eqref{BoundM}, assumptions \ref{A:H1b} and \ref{A:H3b} are respectively equivalent to assumptions \ref{A:H1} and \ref{A:H3}, and that $L_\tau^2(0;\R^k)$ and $H_\tau^2(0;\R^{k}\times\R^{k\times d})$ are respectively identical with $L_\tau^2(a_\cdot;\R^k)$ and $H_\tau^2(a_\cdot;\R^{k}\times\R^{k\times d})$.
The desired assertion follows immediately from \cref{thm:3.1}.
\end{proof}

At the end of this subsection, we make the following remark concerning \cref{cor:3.4} and \cref{thm:3.1}.

\begin{rmk}\label{rmk:3.5}
(i) Assumption \ref{A:H3b*} was usually used on the study of BSDEs with (weak) monotonicity generators, see for example \cite{Briand2003}, \cite{FanJiang2013}, \cite{Fan2015JMAA} and \cite{Xiao2015}. It is clear that due to the presence of process $\alpha_\cdot$, \ref{A:H3b*} is strictly stronger than \ref{A:H3b}. See \cref{ex:3.8} in Section 3.2 for more details. Consequently, for the case of BSDEs with infinite terminal time, \cref{cor:3.4} strengthens Theorem 3.1 in \cite{Liu2020} as well as Theorem 1.2 in \cite{ZChenBWang2000JAMS} and
Theorem 3.1 in \cite{Xiao2015}, where the processes $\mu_\cdot$ and $\nu_\cdot$ are two deterministic functions independent of $\omega$; and for the case of BSDEs with constant terminal time, \cref{cor:3.4} improves Theorem 2.2 in \cite{DarlingandPardoux1997}, Theorem 2.2 in \cite{Pardoux1999} and Theorem 4.2 in \cite{Briand2003} with $p=2$, where the processes $\mu_\cdot$ and $\nu_\cdot$ are two constants independent of $(t,\omega)$.\vspace{0.2cm}

(ii) We would mention the following several works closely related to \cref{thm:3.1}. First, \cite{Yong2006} considered the solvability of linear BSDEs with unbounded stochastic coefficients $\mu_\cdot$ and $\nu_\cdot$, where $\int_0^\tau\mu_t {\rm d}t$ and $\int_0^\tau\nu_t^2 {\rm d}t$ are assumed to have a certain exponential moment, but $\mu_\cdot$ can take values on $\R$. Second, \cite{Bahlali2015} dealt with the solvability of BSDEs with super-linear growth generators under a local condition covering some BSDEs with stochastic coefficients $\mu_\cdot$ and $\nu_\cdot$, see Example 3 in \cite{Bahlali2015} for more details. It should be mentioned that in this example the $\int_0^\tau(\mu_t+\nu_t^2){\rm d}t$ is also supposed to have a certain exponential moment and the generator $g$ is required to grow at most sub-quadratically in the unknown variable $y$, see (H.3) and (H.4) therein. Third, \cite{BriandandConfortola2008} studied BSDEs with stochastic coefficients $\mu_\cdot$ and $\nu_\cdot$ satisfying certain integrability condition associated with the bound mean oscillation martingale, see also \cite{Perninge2023} for further study on reflected BSDEs. In comparison with these works, the generator $g$ of BSDEs in our \cref{thm:3.1} may have a very general growth in $y$ (see assumption \ref{A:H3}), and the stochastic coefficients $\mu_\cdot$ and $\nu_\cdot$ only need to satisfy $\int_0^\tau(\mu_t+\nu_t^2){\rm d}t<+\infty$, without any restriction of finite moment.\vspace{0.2cm}

(iii) It is supposed in assumption \ref{A:H3} that $\alpha_\cdot$ is a non-increasing process taking values in $(0,1]$. This will play a crucial role in the proof of existence part of \cref{thm:3.1}. However, if there exists an $(\F_t)$-progressively measurable, continuous and positive process $\alpha_\cdot$ such that for each $r\in \R_+$, \eqref{3.27*} holds, then for each $r\in \R_+$, it also holds with $\alpha_\cdot$ being replaced with
$$
\overline{\alpha}_t:=\left(\inf\limits_{0\leq s\leq t}\alpha_s\right)\wedge1, \ t\in [0,\tau],
$$
which is an $(\F_t)$-progressively measurable non-increasing process taking values in $(0,1]$. Consequently, the $\alpha_\cdot$ in \ref{A:H3} can be supposed to be an $(\F_t)$-progressively measurable, continuous and positive process.
\end{rmk}

\subsection{Examples\vspace{0.2cm}}

In this subsection, we will provide several examples, to which \cref{thm:3.1} can be applied, but none of existing results including those in \cite{Li2023}, \cite{O2020}, \cite{Bahlali2015}, \cite{Xiao2015}, \cite{PardouxandRascanu(2014)}, \cite{Briand2003}, \cite{Pardoux1999}, \cite{DarlingandPardoux1997} and \cite{PardouxPeng1990SCL} could. In order to facilitate the understanding of these examples, we first introduce the following lemma.

\begin{lem}\label{lem:3.8}
For each $\lambda\in[0,1]$, we have
$$|e^{\lambda x}-1|\leq\lambda(e^{|x|}+|x|-1), \ \forall x\in \R.$$
\end{lem}
\begin{proof}
Given $\lambda\in[0,1]$. We need to consider two cases: $x\leq0$ and $x>0$. For $x\leq0$, it is clear that
$$|e^{\lambda x}-1|=1-e^{\lambda x}\leq-\lambda x=\lambda|x|.$$
On the other hand, since $f(x):=e^x-1, x\in \R$ is a convex function with $f(0)=0$, we can deduce that for each $x>0$,
$$|e^{\lambda x}-1|=e^{\lambda x}-1=f(\lambda x)\leq\lambda f(x)=\lambda(e^x-1)=\lambda(e^{|x|}-1).$$
Then the desired assertion follows immediately from the last two inequalities.
\end{proof}

Next, we give some examples to show that our results can cover the known results, but not vice versa. Firstly, the generator of the following example satisfies \ref{A:H3b}, but does not satisfy \ref{A:H3b*}. Hence, \ref{A:H3b} is strictly weaker than \ref{A:H3b*} due to the presence of $(\alpha_t)_{t\in[0,\tau]}$, as mentioned in (i) of \cref{rmk:3.5}.

\begin{ex}\label{ex:3.8}
Let $\tau$ be a bounded stopping time, i.e., $\tau\leq T$ for a real $T>0$, and for each $y=(y_1,\cdots,y_k)\in \R^k$ and $z\in \R^{k\times d}$, let
$$g(t,y,z):=(g^1(t,y,z),\cdots,g^k(t,y,z)), \ t\in[0,\tau],$$
where for each $i=1,\cdots,k,$
$$g^i(t,y,z):=e^{-|B_t|^3y_i}+|z|.$$
It is straightforward to verify that assumptions \ref{A:H1b}, \ref{A:H2}, \ref{A:H4} and \ref{A:H5} with $|g(\cdot,0,0)|\equiv \sqrt{k}$, $\mu_\cdot\equiv 0$ and $\nu_\cdot\equiv 1$ are satisfied by this generator $g$. Let
$$\alpha_t:=\frac{1}{1+\sup\limits_{0\leq s\leq t}|B_s|^3}, \ t\in[0,\tau].$$
Then $(\alpha_t)_{t\in[0,\tau]}$ is an $(\F_t)$-progressively measurable non-increasing process taking values in $(0,1]$, and
\begin{align}\label{3.93966}
\begin{split}
\psi_{r}^{\alpha_\cdot}(t):=&\sup_{|y|\leq r\alpha_t}\left|g(t,y,0)-g(t,0,0)\right|
\leq\sup_{|y_i|\leq r\alpha_t}\left(\mathop{\sum}\limits_{i=1}^{k}\bigg|e^{-|B_t|^3y_i}-1\bigg|\right)\\
\leq&k\sup_{|x|\leq r\alpha_t|B_t|^3} \bigg|e^{-x}-1\bigg|
\leq k\sup_{|x|\leq r}\bigg|e^{-x}-1\bigg|
\leq k(e^r+r-1), \ t\in[0,\tau],
\end{split}
\end{align}
where the definition of $\alpha_\cdot$ and \cref{lem:3.8} are respectively used in the last two steps of \eqref{3.93966}. Therefore, \ref{A:H3b} also holds for this $g$. It then follows from \cref{cor:3.4} that for each $\xi\in L_\tau^2(0;\R^k)$, BSDE \eqref{BSDE1.1} admits a solution in $H_\tau^2(0;\R^{k}\times\R^{k\times d})$. However, $g$ does not satisfy assumption \ref{A:H3b*}. In fact, for each $r>0$, by letting $y=(-r,0,\cdots,0)$, we have
$$\psi_{r}(t):=\sup_{|y|\leq r}\left|g(t,y,0)-g(t,0,0)\right|\geq e^{r|B_t|^3}-1$$
and then
$$
\E\left[\int_0^\tau \psi_{r}(t)dt\right]=+\infty.\vspace{0.2cm}
$$
Hence, the above conclusion cannot be obtained by  Theorem 4.2 in \cite{Briand2003} and Theorem 3.1 in \cite{Xiao2015}.
\end{ex}

\begin{ex}\label{ex:3.12}
Let $k=2$, and $\tau$ be a finite stopping time, i.e., $P(\tau<+\infty)=1$. Consider the following generator: for $(y,z)\in \R^2\times \R^{2\times d}$ with $y=(y_1,y_2)$,
$$g(t,y,z):=|B_t|^3\begin{bmatrix}
-y_1^5+y_2\\
-y_2^3-y_1\\
\end{bmatrix}+|B_t|\begin{bmatrix}
\sin|z|\\
|z|\\
\end{bmatrix}, \ t\in[0,\tau].$$
Obviously, this generator $g$ satisfies assumptions \ref{A:H1}, \ref{A:H2}, \ref{A:H4} and \ref{A:H5} with
$$g(t,0,0)\equiv0, \ \mu_t=|B_t|^3, \  \nu_t=|B_t|,$$
and $a_t:=\beta|B_t|^3+\frac{\rho}{2}|B_t|^2$ satisfying $\int_0^\tau a_t{\rm d}t<+\infty$. Moreover, $g$ also satisfies \ref{A:H3a} with $$\tilde{\mu}_t=|B_t|^3 \ and \ \varphi(x)=\sqrt{2}(x^5+x^3+x), \ x\geq0.$$
It follows from \cref{thm:3.1} together with (iv) of \cref{rmk:3.3} that for each $\xi\in L_\tau^2(a_\cdot;\R^k)$, for example, $\xi=e^{-\int_0^\tau a_t{\rm d}t}B_\tau$. BSDE \eqref{BSDE1.1} admits a unique weighted $L^2$-solution in the space of $H_\tau^2(a_\cdot;\R^{2}\times\R^{2\times d})$. However, note that $g$ has a polynomial growth in $y$ and $\E[e^{\varepsilon\int_0^\tau a_t{\rm d}t}]=+\infty$ for each $\varepsilon>0$, the proceeding assertion can not be derived by any existing results including those in \cite{PardouxPeng1990SCL}, \cite{Briand2003}, \cite{Xiao2015}, \cite{Li2023} and \cite{Bahlali2015}.
\end{ex}

For simplicity, we assume that $k=1$ in the following examples.

\begin{ex}\label{ex:3.9}
Let $k=1$, $\tau$ be a finite stopping time, i.e., $P(\tau<+\infty)=1$, and for each $(y,z)\in\R\times\R^{d}$,
$$g(t,y,z):=e^{-|B_t|^4y}+|B_t|(|y|+|z|)-1, \ t\in[0,\tau].$$
It is not hard to verify that $g$ satisfies assumptions \ref{A:H1}-\ref{A:H5} with
$$g(t,0,0)\equiv0, \ \mu_t=|B_t|, \  \nu_t=|B_t|,$$
and $a_t:=\beta|B_t|+\frac{\rho}{2}|B_t|^2$ with $\int_0^\tau a_t{\rm d}t<+\infty$ due to the fact that $\tau$ is a finite stopping time. In fact, it is clear that \ref{A:H1}, \ref{A:H2}, \ref{A:H4} and \ref{A:H5} hold for $g$. Next we prove that $g$ also satisfies \ref{A:H3}. Let $$\alpha_t:=\frac{\lambda_t}{\sup\limits_{0\leq s\leq t}(1+|B_s|)^4}, \ t\in[0,\tau],$$
where $\lambda_t=e^{-\beta\int_0^t\mu_s{\rm d}s-t}$. Then $(\alpha_t)_{t\in[0,\tau]}$ is an $(\F_t)$-progressively measurable non-increasing process taking values in $(0,1]$, and for each $r\in \R_+$, it holds that
\begin{align}\label{3.939}
\begin{split}
\psi_{r}^{\alpha_\cdot}(t):=&\sup_{|y|\leq r\alpha_t}\left\{ \left|g(t,y,0)-g(t,0,0)\right|\right\}\\
=&\sup_{|y|\leq r\alpha_t}\bigg|e^{-|B_t|^4y}+|B_t||y|-1\bigg| \leq\sup_{|y|\leq r\alpha_t}\bigg|e^{-|B_t|^4y}-1\bigg|+r|B_t|\alpha_t\\
=&\sup_{|x|\leq \frac{r\alpha_t|B_t|^4}{\lambda_t}} \bigg|e^{-\lambda_tx}-1\bigg|+r|B_t|\alpha_t
\leq\sup_{|x|\leq r}\bigg|e^{\lambda_tx}-1\bigg|+r\lambda_t, \ t\in[0,\tau].
\end{split}
\end{align}
Then, by \cref{lem:3.8} we know that
$$\E\left[\int_0^\tau e^{\beta\int_0^t\mu_s{\rm d}s} \psi_r^{\alpha_\cdot}(t)dt\right]\leq\E\left[\int_0^\tau e^{\beta\int_0^t\mu_s{\rm d}s}\lambda_t(e^r+2r-1)dt\right]\leq(e^r+2r-1)\int_0^{+\infty}e^{-t}{\rm d}t<+\infty.$$
It then follows from Theorem \ref{thm:3.1} that for each $\xi\in L_\tau^2(a_\cdot;\R^k)$, BSDE \eqref{BSDE1.1} admits a unique weighted $L^2$-solution in the space of $H_\tau^2(a_\cdot;\R\times\R^{1\times d})$. However, since $g$ does not satisfy assumptions \ref{A:SL} and \ref{A:H3*} by a similar analysis to that in \cref{ex:3.8}, to the best of our knowledge, this conclusion cannot be obtained by any existing results.
\end{ex}

\begin{ex}\label{ex:3.10}
Let $k=1$, $\tau\equiv+\infty$, and for each $(y,z)\in\R\times\R^{d}$,
$$g(t,y,z):=e^{-\frac{\rho}{2}\int_0^{t\wedge1}|B_s|{\rm d}s-t}e^{y^-}+\sqrt{|B_t|{\bf 1}_{0\leq t\leq 1}+\frac{1}{1+t^2}}\sin|z|, \ , \ t\in[0,\tau].$$
It can be directly verified that $g$ satisfies assumptions \ref{A:H1}, \ref{A:H2}, \ref{A:H4} and \ref{A:H5} with
$$g(t,0,0)=e^{-\frac{\rho}{2}\int_0^{t\wedge1}|B_s|{\rm d}s-t}, \ \mu_t\equiv0, \  \nu_t=\sqrt{|B_t|{\bf 1}_{0\leq t\leq 1}+\frac{1}{1+t^2}},$$
and $a_t:=\frac{\rho}{2}\left(|B_t|{\bf 1}_{0\leq t\leq 1}+\frac{1}{1+t^2}\right)$ satisfying that $\int_0^\tau a_t{\rm d}t<+\infty$. Furthermore, $g$ also satisfies \ref{A:H3a} with
$$\tilde{\mu}_t=e^{-\frac{\rho}{2}\int_0^{t\wedge1}{|B_s|\rm d}s-t} \ and  \ \varphi(x)=e^x-1, \ x\geq0.$$
Then, by \cref{thm:3.1} together with (iv) of \cref{rmk:3.3} we know that for each $\xi\in L_\tau^2(a_\cdot;\R^k)$, BSDE \eqref{BSDE1.1} admits a unique weighted $L^2$-solution in $H_\tau^2(a_\cdot;\R\times\R^{1\times d})$. Clearly, this assertion can not be obtained from Theorem 2.1 of \cite{Li2023}, Theorem 3.1 of \cite{Xiao2015} and any known results.
\end{ex}

\begin{ex}\label{ex:3.11}
Let $k=1$, $\tau$ be a stopping time taking values in $[0,+\infty]$ and $\sigma$ be a finite stopping time. For each $(y,z)\in\R\times\R^{d}$, let
$$g(t,y,z):=|B_t|^6(1-e^{y^{+}})+|B_t|{\bf 1}_{0\leq t\leq \sigma}\sin y+\sqrt{|B_t|{\bf 1}_{0\leq t\leq \sigma}}|z|, \ t\in[0,\tau].$$
It is straightforward to verify that this generator $g$ satisfies assumptions \ref{A:H1}, \ref{A:H2}, \ref{A:H4} and \ref{A:H5} with
$$g(t,0,0)\equiv0, \ \mu_t=|B_t|{\bf 1}_{0\leq t\leq \sigma}, \ \nu_t=\sqrt{|B_t|{\bf 1}_{0\leq t\leq \sigma}},$$
and $a_t:=(\beta+\frac{\rho}{2})|B_t|{\bf 1}_{0\leq t\leq \sigma}$ satisfying $\int_0^\tau a_t{\rm d}t<+\infty$ due to the fact that $\sigma$ is a finite stopping time. Furthermore, it can also be verified that $g$ satisfies assumption \ref{A:H3a} with $\tilde{\mu}_t=(|B_t|+1)^6$ and $\varphi(x)=e^x+x-1, \ x\geq0$. It follows from \cref{thm:3.1} that for each $\xi\in L_\tau^2(a_\cdot;\R^k)$, BSDE \eqref{BSDE1.1} has a unique weighted $L^2$-solution in $H_\tau^2(a_\cdot;\R\times\R^{1\times d})$. This conclusion cannot be obtained by any existing results.
\end{ex}

\begin{rmk}\label{rmk:3.11}
The above examples illustrate that our assumptions are strictly weaker than the corresponding ones in some existing works, and then \cref{thm:3.1} can completely cover some known results, but the converse assertions are not true.
\end{rmk}

\subsection{Continuous dependence and stability\vspace{0.2cm}}

The following theorem gives a general continuous dependence property for the weighted $L^2$-solutions with respect to parameters of BSDEs. This result extends Theorem 3.1 in \cite{Li2023}.

\begin{thm}[\bf Continuous dependence]\label{thm:3.2}
Let $\xi, \xi'\in L_\tau^2(a_\cdot;\R^k)$, and $g$ and $g'$ be two generators of BSDEs satisfying assumptions \ref{A:H4} and \ref{A:H5}. Assume that $(Y_\cdot,Z_\cdot)$ and $(Y_\cdot',Z_\cdot')$ are a weighted $L^2$-solutions of BSDE$(\xi,\tau,g)$ and BSDE$(\xi',\tau,g')$, respectively, in the space of $H_\tau^2(a_\cdot;\R^{k}\times\R^{k\times d})$.
Then there exists a uniform constant $C>0$ such that
\begin{align*}
\begin{split}
&\E\left[\sup_{s\in[0,\tau]}\left(e^{2 \int_{0}^{s}a_r{\rm d}r}|Y_s-Y_s'|^2\right)\right]+\E\left[\int_0^{\tau} e^{2\int_0^s a_r{\rm d}r}|Z_s-Z_s'|^2{\rm d}s\right]\\
&\ \ \leq  C\left(\E\left[e^{2 \int_{0}^{\tau}a_r{\rm d}r}|\xi-\xi'|^2\right]+\E\left[\left(\int_0^\tau e^{\int_0^sa_r{\rm d}r}\left|g(t,Y'_t,Z'_t)-g'(t,Y'_t,Z'_t)\right|{\rm d}s\right)^2\right]\right).
\end{split}
\end{align*}
\end{thm}

\begin{proof}
Without loss of generality, we can assume that
$$\E\left[\left(\int_0^\tau e^{\int_0^sa_r{\rm d}r}\left|g(t,Y'_t,Z'_t)-g'(t,Y'_t,Z'_t)\right|{\rm d}s\right)^2\right]<+\infty.$$
Let\vspace{-0.2cm}
\begin{align*}
\begin{split}
\tilde{\xi}:=\xi-\xi', \ \tilde{Y}_\cdot:=Y_\cdot-Y_\cdot', \ \tilde{Z}_\cdot:=Z_\cdot-Z_\cdot'.
\end{split}
\end{align*}
Then $(\tilde{Y}_\cdot,\tilde{Z}_\cdot)$ is a weighted $L^2$-solution of the following BSDE:
\begin{align*}
\tilde{Y}_t=\tilde{\xi}+\int_t^\tau \tilde{g}(s,\tilde{Y}_s,\tilde{Z}_s){\rm d}s-\int_t^\tau\tilde{Z}_s{\rm d}B_s, \ t\in[0,\tau],
\end{align*}
where for each $(y,z)\in\R^k\times\R^{k\times{d}}$,
$$\tilde{g}(t,y,z):=g(t,y+Y_t',z+Z_t')-g'(t,Y_t',Z_t'), \ t\in[0,\tau].$$
It follows from \ref{A:H4} and \ref{A:H5} that for each $(y,z)\in\R^k\times\R^{k\times{d}}$,
\begin{align*}
\begin{split}
\left<\hat{y},\tilde{g}(t,y,z)\right>&=\left<\hat{y},g(t,y+Y_t',z+Z_t')-g'(t,Y_t',Z_t')\right>\\
&\leq \mu_t|y|+\nu_t|z|+|g(t,Y_t',Z_t')-g'(t,Y_t',Z_t')|, \ t\in[0,\tau].
\end{split}
\end{align*}
This implies that assumption (A) is satisfied by the generator $\tilde{g}$ with $\bar\mu_t=\mu_t$, $\bar\nu_t=\nu_t$ and $$f_t=|g(t,Y_t',Z_t')-g'(t,Y_t',Z_t')|.$$
Thus, by using \eqref{2.03} of \cref{pro:1.1} we can get the desired result.
\end{proof}

Next, we present and prove the following stability theorem for the weighted $L^2$-solutions of BSDEs. It improves Theorem 2.1 in \cite{HuPeng1997}.

\begin{thm}[\bf Stability]\label{thm:3.3}
For each $n\geq1$, let $\xi^n, \xi\in L_\tau^2(a_\cdot;\R^k)$, and both $g^n$ and $g$ satisfy assumptions \ref{A:H4} and \ref{A:H5}, and assume that $(Y_\cdot^n,Z_\cdot^n)$ and $(Y_\cdot,Z_\cdot)$ are weighted $L^2$-solutions of BSDE $(\xi^n,\tau,g^n)$ and BSDE $(\xi,\tau,g)$ in the space of $H_\tau^2(a_\cdot;\R^{k}\times\R^{k\times d})$. If
$$\lim_{n\rightarrow\infty}\E\left[e^{2 \int_{0}^{\tau}a_r{\rm d}r}|\xi^n-\xi|^2+\left(\int_0^\tau e^{\int_0^sa_r{\rm d}r}\left|g^n(s,Y_s,Z_s)-g(s,Y_s,Z_s)\right|{\rm d}s\right)^2\right]=0,$$
then we have
$$
\lim_{n\rightarrow\infty}\E\left[\sup_{s\in[0,\tau]}\left(e^{2 \int_{0}^{s}a_r{\rm d}r}|Y_s^n-Y_s|^2\right)+\int_0^{\tau} e^{2\int_0^s a_r{\rm d}r}|Z_s^n-Z_s|^2{\rm d}s\right]=0.\vspace{0.2cm}
$$
\end{thm}

\begin{proof}
For each $n\geq1$, let
$$\hat{Y}_\cdot^n:=Y_\cdot^n-Y_\cdot, \ \hat{Z}_\cdot^n:=Z_\cdot^n-Z_\cdot, \ \hat{\xi}^n:=\xi^n-\xi.$$
Then
\begin{align*}
\hat{Y}_t^n=\hat{\xi}^n+\int_t^\tau \hat{g}^n(s,\hat{Y}_s^n,\hat{Z}_s^n){\rm d}s-\int_t^\tau\hat{Z}_s^n{\rm d}B_s, \ t\in[0,\tau],
\end{align*}
where for each $(y,z)\in\R^k\times\R^{k\times{d}}$,
$$\hat{g}^n(t,y,z):=g^n(t,y+Y_t,z+Z_t)-g(t,Y_t,Z_t), \ t\in[0,\tau].$$
By virtue of \ref{A:H4} and \ref{A:H5} we can check that for each $n\geq1$, the generator $\hat g^n$ satisfies assumption \ref{A:A} with $\bar\mu_t=\mu_t$, $\bar\nu_t=\nu_t$ and $f_t=|g^n(t,Y_t,Z_t)-g(t,Y_t,Z_t)|.$ Therefore, it follows from \eqref{2.03} in \cref{pro:1.1} with $r=t=0$ that there exists a uniform constant $C>0$ such that
\begin{align*}
\begin{split}
&\E\left[\sup_{s\in[0,\tau]}\left(e^{2 \int_{0}^{s}a_r{\rm d}r}|Y_s^n-Y_s|^2\right)+\int_0^{\tau} e^{2\int_0^s a_r{\rm d}r}|Z_s^n-Z_s|^2{\rm d}s\right]\\
&\ \ \leq C\E\left[e^{2 \int_{0}^{\tau}a_r{\rm d}r}|\xi^n-\xi|^2+\left(\int_0^\tau e^{\int_0^sa_r{\rm d}r}\left|g^n(s,Y_s,Z_s)-g(s,Y_s,Z_s)\right|{\rm d}s\right)^2\right].
\end{split}
\end{align*}
Thus, the desired assertion follows by sending $n$ to infinity in the last inequality.
\end{proof}

\subsection{Application to PDEs\vspace{0.2cm}}

In this subsection, we present applications of our findings regarding BSDEs to the realm of parabolic PDEs and elliptic PDEs. More precisely, we obtain two nonlinear Feynman-Kac formulas under some general assumptions, see \cref{thm:3.4.1} and \cref{thm:3.4.2}. For convenience, we always assume that $k=1$ in this subsection. Let $K,p>0$ and $q\in[1,2)$ be three constants, and $l$ be a positive integer.

\subsubsection{Probabilistic interpretation for parabolic PDEs\vspace{0.2cm}}

Assume that the terminal time $\tau\equiv T$ for a finite constant $T>0$. Let both $b(t,x):[0,T]\times\R^l\rightarrow\R^l$ and $\sigma(t,x):[0,T]\times \R^l\rightarrow\R^{l\times d}$ be jointly continuous and globally Lipschitz continuous in $x$, uniformly with respect to $t$, and satisfy that for each $(t,x)\in \T\times\R^l$,
\begin{align}\label{conditionb}
|b(t,x)|\leq K(1+|x|)\ \ {\rm and}\ \ |\sigma(t,x)|\leq K.
\end{align}
Furthermore, let both $h(x):\R^l \rightarrow \R$ and $g(t,x,y,z):[0,T]\times \R^l\times\R\times \R^{1\times d} \rightarrow \R$ be (jointly) continuous such that for each $(t,x,y,y_1,y_2,z,z_1,z_2)\in [0,T]\times \R^l\times\R^3\times \R^{3(1\times d)}$,
\begin{align}\label{condition-1}
\begin{array}{c}
|h(x)|\leq K\exp(p|x|^q),\\
|g(t,x,y,0)|\leq K\exp(p|x|^q)\varphi(|y|),\\
(y_1-y_2)(g(t,x,y_1,z)-g(t,x,y_2,z))\leq \bar b(t,x)|y_1-y_2|^2,\\
|g(t,x,y,z_1)-g(t,x,y,z_2)|\leq \bar\sigma (t,x)|z_1-z_2|,
\end{array}
\end{align}
where $\varphi(\cdot):\R_+\rightarrow \R_+$ is a nondecreasing convex function, and $\bar b(t,x), \bar\sigma(t,x):[0,T]\times\R^l\rightarrow\R_+$ are two jointly continuous functions satisfying the following growth condition:
$$\bar b(t,x)+\bar\sigma^2(t,x)\leq K(1+|x|^q).$$
For example, it is easy to verify that the following function
$$g(t,x,y,z):=|x|^q|y|-\exp(2|x|^q)e^y+|x|\sin |z|,\ \ (t,x,y,z)\in \T\times\R^l\times\R\times\R^{1\times d}$$
satisfies the above-mentioned condition on the function $g$.

Consider the following parabolic PDE:
\begin{align}\label{PDE}
\left\{\begin{aligned}
&\frac{\partial u(t,x)}{\partial t}+\mathcal{L}u(t,x)+g(t,x,u(t,x),(\nabla_x u\sigma)(t,x))=0,\ \ (t,x)\in \T\times\R^l,\\
&u(T,x)=h(x),\ x\in\R^l,
\end{aligned}\right.
\end{align}
where
$$
\mathcal{L}:=\frac{1}{2}\sum\limits_{i,j}(\sigma\sigma^*)_{i,j}(t,x)
\frac{\partial^2}{\partial x_i\partial x_j}+\sum\limits_{i}b_i(t,x)\frac{\partial}{\partial x_i},\ \ \nabla_x:=(\frac{\partial}{\partial x_1},\cdots,\frac{\partial}{\partial x_l}),\vspace{0.2cm}
$$
and recall the definition of a continuous viscosity solution to PDE \eqref{PDE} in our framework, see for example \cite{CrandallIshiiLions1992BAMS}.

\begin{dfn}
A continuous function $u(t,x):\T\times \R^l\rightarrow \R$ with $u(T,\cdot)=h(\cdot)$ is a viscosity sub-solution (super-solution) of PDE \eqref{PDE} if for any function $\phi\in C^{1,2}(\T\times \R^l;\R)$, it holds that
$$\frac{\partial \phi(t_0,x_0)}{\partial t}+\mathcal{L}\phi(t_0,x_0)+g(t_0,x_0,u(t_0,x_0),(\nabla_x\phi\sigma)(t_0,x_0))\geq 0,\ \ (\leq 0), $$
provided that $u-\phi$ attains a local maximum (minimum) at point $(t_0,x_0)\in [0,T)\times \R^l$. A viscosity solution of PDE \eqref{PDE} is both a viscosity sub-solution and viscosity super-solution.
\end{dfn}

In the sequel, we aim to give a viscosity solution of the previous PDE \eqref{PDE} via the adapted solution of a BSDE coupled with a SDE. First of all, given $(t,x)\in \T\times \R^l$. From the classical theory of SDEs, it is known that under the above assumptions on the functions $b$ and $\sigma$, the following SDE
\begin{align}\label{SDE-1}
X_s^{t,x}=x+\int_t^s b(r,X_r^{t,x}){\rm d}r+\int_t^s \sigma(r,X_r^{t,x}){\rm d}B_r,\ \ s\in [t,T],
\end{align}
admits a unique solution $(X_s^{t,x})_{s\in [t,T]}$, and for each $\gamma>0$, there exists a constant $C>0$ depending only on $(\gamma,q,T)$ such that
\begin{align}\label{exp}
\E\left[\exp\left(\gamma\sup\limits_{t\leq s\leq T}|X_s^{t,x}|^q\right)\right]\leq C\exp\left(C|x|^q\right).\vspace{0.1cm}
\end{align}
Furthermore, by the above assumptions on functions $h$ and $g$ along with the last inequality, we can verify that $h(X_T^{t,x})\in L_T^2(a_\cdot;\R)$ and the generator $g(r,X_r^{t,x},y,z){\bf 1}_{t\leq r\leq T}$ satisfies assumptions \ref{A:H1}, \ref{A:H2}, \ref{A:H3b}, \ref{A:H4} and \ref{A:H5} with $\beta=1$, $\rho=2$,
$$\mu_r:=\bar b(r,X_r^{t,x}){\bf 1}_{t\leq r\leq T},\ \ \nu_r:=\bar\sigma(r,X_r^{t,x}){\bf 1}_{t\leq r\leq T}\ \ {\rm and} \ \  {\tilde \mu}_r:=K(1+|X_r^{t,x}|){\bf 1}_{t\leq r\leq T}.$$
Consequently, it follows from \cref{cor:3.4} that the following BSDE
\begin{align}\label{BSDE-1}
Y_s^{t,x}=h(X_T^{t,x})+\int_s^T g(r,X_r^{t,x},Y_r^{t,x},Z_r^{t,x}){\rm d}r-\int_s^T Z_r^{t,x}{\rm d}B_r,\ \ s\in[t,T],
\end{align}
admits a unique weighted $L^2$-solution $(Y_s^{t,x},Z_s^{t,x})_{s\in [t,T]}$. In addition, by the Markov property of SDE \eqref{SDE-1} and BSDE \eqref{BSDE-1} we can conclude that $Y_t^{t,x}$ is a deterministic real number.


\begin{thm}\label{thm:3.4.1}
Under the above assumptions, $u(t,x):=Y_t^{t,x},\ (t,x)\in \T\times\R^l$ is a continuous function with respect to $(t,x)$, and it is a viscosity solution of PDE \eqref{PDE}. Moreover, there exists a constant $C>0$ depending only on $(K,p,q,T)$ such that
\begin{align}\label{eq:3.35}
|u(t,x)|\leq C\exp(C|x|^q).
\end{align}
\end{thm}

\begin{proof}
From \cref{thm:3.2} and the continuity of $X_\cdot^{t,x}$ with respect to $(t,x)$, it follows that $u(t,x):=Y_t^{t,x}$ is a continuous function with respect to $(t,x)$. The desired assertion \eqref{eq:3.35} can be easily derived from \eqref{exp} and \cref{pro:1.1} together with the assumptions on functions $h$ and $g$. We omit its proof here.

In the sequel, we prove that $u(t,x)$ is a viscosity sub-solution. By an identical way, it can be shown that $u(t,x)$ is also a viscosity super-solution. Now, take any function $\phi\in C^{1,2}(\T\times \R^l;\R)$ such that $u-\phi$ attains a local maximum at $(t_0,x_0)\in [0,T)\times\R^l$. We assume without loss of generality that
$$\phi(t_0,x_0)=u(t_0,x_0).$$
Therefore, in order to show that $u(t,x)$ is a viscosity sub-solution, it only need to prove that
$$\frac{\partial \phi(t_0,x_0)}{\partial t}+\mathcal{L}\phi(t_0,x_0)+g(t_0,x_0,u(t_0,x_0),(\nabla_x\phi\sigma)(t_0,x_0))\geq0.$$
If not, it is known from continuity that there exists $\gamma>0$ and $0<\delta\leq T-t_0$, when $t_0\leq t\leq t_0+\delta$ and $|y-x_0|\leq\delta$, such that
\begin{align}\label{333,33}
\left\{\begin{aligned}
&u(t,y)\leq \phi(t,y);\\
&\frac{\partial \phi(t,y)}{\partial t}+\mathcal{L}\phi(t,y)+g(t,y,u(t,y),(\nabla_x\phi\sigma)(t,y))\leq -\gamma.
\end{aligned}\right.
\end{align}
Define the following stopping time:
$$\vartheta:=\inf\{u\geq t_0:|X_u^{t_0,x_0}-x_0|\geq \delta\}\wedge(t_0+\delta)>t_0.$$
By the Markov property and BSDE \eqref{BSDE-1}, we know that $(\bar {Y}_t,\bar {Z}_t)
:=(Y_{t}^{t_0,x_0},Z_t^{t_0,x_0})$ is an adapted solution of the following BSDE:
\begin{align}\label{BSDE1}
\bar{Y}_t=u(\vartheta,X_{\vartheta}^{t_0,x_0})+\int_t^{\vartheta} g(s,X_s^{t_0,x_0},u(s,X_s^{t_0,x_0}),\bar {Z}_s){\rm d}s-\int_t^{\vartheta} \bar{Z}_s{\rm d}B_s, \ t\in[t_0,\vartheta].
\end{align}
On the other hand, it follows from It\^{o}'s formula that $(\widetilde{Y}_t,\widetilde{Z}_t):=(\phi(t,X_t^{t_0,x_0}),
(\nabla_x\phi\sigma)(t,X_t^{t_0,x_0}))$ is an adapted solution for the following BSDE:
\begin{align}\label{BSDE2}
\widetilde{Y}_t=\phi(\vartheta,X_{\vartheta}^{t_0,x_0})-\int_t^{\vartheta} \left\{\frac{\partial\phi}{\partial t}+\mathcal{L}\phi\right\}(s,X_s^{t_0,x_0})){\rm d}s-\int_t^{\vartheta} \widetilde{Z}_s{\rm d}B_s, \ t\in[t_0,\vartheta].
\end{align}
Combining \eqref{333,33}, \eqref{BSDE1} and \eqref{BSDE2} along with the Lipschitz continuity for $g$ in $z$ and the definition of the stopping time $\vartheta$, we obtain
\begin{align}\label{eq:3.39}
\begin{split}
\widetilde{Y}_{t_0}-\bar {Y}_{t_0}=&\ \phi(\vartheta,X_{\vartheta}^{t_0,x_0})
-u(\vartheta,X_{\vartheta}^{t_0,x_0})-\int_{t_0}^{\vartheta} (\widetilde{Z}_s-\bar {Z}_s){\rm d}B_s\\
&\ -\int_{t_0}^{\vartheta}\left(\left\{\frac{\partial\phi}{\partial t}+\mathcal{L}\phi\right\}(s,X_s^{t_0,x_0})+g(s,X_s^{t_0,x_0},
u(s,X_s^{t_0,x_0}),\widetilde{Z}_s)\right){\rm d}s\\
&\ +\int_{t_0}^{\vartheta} \left( g(s,X_s^{t_0,x_0},u(s,X_s^{t_0,x_0}),\widetilde{Z}_s)
-g(s,X_s^{t_0,x_0},u(s,X_s^{t_0,x_0}),\bar {Z}_s)\right){\rm d}s\\
\geq &\ \gamma(\vartheta-t_0) -\int_{t_0}^{\vartheta} \bar\sigma(s,X_s^{t_0,x_0})|\widetilde{Z}_s-\bar {Z}_s|{\rm d}s-\int_{t_0}^{\vartheta}(\widetilde{Z}_s-\bar {Z}_s){\rm d}B_s.
\end{split}
\end{align}
Denote a new probability measure $\mathbb{Q}$ equivalent to $\mathbb{P}$ by
$$
\frac{{\rm d}\mathbb{Q}}{{\rm d}\mathbb{P}}:=\exp\left(-\int_{t_0}^\vartheta \bar\sigma(r,X_r^{t_0,x_0})
\frac{\widetilde{Z}_r-\bar {Z}_r}
{|\widetilde{Z}_r-\bar {Z}_r|}{\bf 1}_{|\widetilde{Z}_r-\bar {Z}_r|\neq 0} {\rm d}B_r-\frac{1}{2}\int_{t_0}^\vartheta \bar\sigma^2(r,X_r^{t_0,x_0}){\bf 1}_{|\widetilde{Z}_r-\bar {Z}_r|\neq 0}{\rm d}r\right).
$$
According to Girsanov's theorem, we know that the process
$$
B^{\mathbb{Q}}_s:=B_s+\int_{t_0}^{s\wedge\vartheta\vee t_0} \bar\sigma(r,X_r^{t_0,x_0})\frac{(\widetilde{Z}_r-\bar Z_r)^*}
{|\widetilde{Z}_r-\bar Z_r|}{\bf 1}_{|\widetilde{Z}_r-\bar Z_r|\neq 0} {\rm d}r,\ \ s\in [0,T]
$$
is a standard Brownian motion under the probability measure $\mathbb{Q}$, and we have, in view of \eqref{eq:3.39},
$$
\widetilde{Y}_{t_0}-\bar Y_{t_0}\geq \gamma(\vartheta-t_0) -\int_{t_0}^{\vartheta}(\widetilde{Z}_s-\bar Z_s){\rm d}B^{\mathbb{Q}}_s.
$$
By taking the conditional mathematical expectation under $\mathbb{Q}$ with respect to $\F_{t_0}$ in the last inequality, we have $\widetilde{Y}_{t_0}>\bar Y_{t_0}$, i.e., $\phi(t_0,x_0)>u(t_0,x_0)$ contradicting our standing assumption.
\end{proof}

\subsubsection{Probabilistic interpretation for elliptic PDEs\vspace{0.2cm}}

Let both $b(x):\R^l\rightarrow\R^l$ and $\sigma(x): \R^l\rightarrow\R^{l\times d}$ be globally Lipschitz continuous functions. For $x\in \R^l$, let $(X_t^{x})_{t\geq 0}$ denote the solution of the following SDE:
\begin{align}\label{SDE-2}
X_t^{x}=x+\int_0^t b(X_s^{x}){\rm d}s+\int_0^t \sigma(X_s^{x}){\rm d}B_s, \ \ t\geq 0.
\end{align}
Let $D$ be an open bounded subset of $\R^l$, whose boundary $\partial D$ is of class $C^1$. For each $x\in \bar {D}:=D\cup\partial D$, define the following stopping time:
$$\tau_x\equiv\inf\{t\geq0: X_t^{x}\notin \bar {D}\}.$$
We assume that $\mathbb{P}(\tau_x<\infty)=1$ for all $x\in\bar {D}$, and that the set
\begin{align}\label{closed}
\Gamma\equiv\{x\in\partial D:\mathbb{P}(\tau_x>0)=0\}\ \ {\rm is \ closed}.
\end{align}
Let both $h(x):\R^l\rightarrow\R$ and $g(x,y,z): \R^l\times\R\times \R^{1\times d}\rightarrow\R$ be (jointly) continuous such that for each $(x,y,y_1,y_2,z,z_1,z_2)\in \R^l\times \R^3\times\R^{3(1+d)}$,
\begin{align}\label{condition-2}
\begin{array}{c}
|g(x,y,0)|\leq \varphi(|x|)\varphi(|y|),\\
(y_1-y_2)g(x,y_1,z)-g(x,y_2,z))\leq \bar b(x)|y_1-y_2|^2,\\
|g(x,y,z_1)-g(x,y,z_2)|\leq \bar\sigma(x)|z_1-z_2|,
\end{array}
\end{align}
where $\varphi(\cdot):\R_+\rightarrow \R_+$ is a nondecreasing convex function, and $\bar b(x),\bar\sigma(x):\R^l\rightarrow \R_+$ are two continuous functions. We further assume that for each $x\in \bar {D}$,
$$
\E\left[e^{2\int_0^{\tau_x} a^x_s {\rm d}s}h^2(X^x_{\tau_x})+
\left(\int_0^{\tau_x}e^{\int_0^t a^x_s {\rm d}s}|g(X^x_t,0,0)| {\rm d}t\right)^2\right]<+\infty
$$
with
$$a^x_s:=\beta \bar b(X^x_s)+\frac{\rho}{2} \bar\sigma^2(X^x_s).$$
Then, according to \cref{thm:3.1} and \cref{cor:3.4}, we know that
the following BSDE
\begin{align}\label{BSDE-2}
Y_t^{x}=h(X_{\tau_x}^{x})+\int_t^{\tau_x} g(X_s^{x},Y_s^{x},Z_s^{x}){\rm d}s-\int_t^{\tau_x} Z_s^{x}{\rm d}B_s,\ \ t\in [0, \tau_x]
\end{align}
admits a unique weighted $L^2$-solution $(Y_t^{x},Z_t^{x})_{t\in [0,\tau_x]}$.

In the sequel, we consider the following elliptic PDE with Dirichlet boundary condition:
\begin{align}\label{PDE*}
\left\{\begin{aligned}
&\mathcal{\bar L}u(x)+g(x,u(x),(\nabla u\sigma)(x))=0, \ \ x\in D;\\
&u(x)=h(x),\ \ x\in\partial D,
\end{aligned}\right.
\end{align}
where
$$
\mathcal{\bar L}:=\frac{1}{2}\sum\limits_{i,j}(\sigma\sigma^*)_{i,j}(x)
\frac{\partial^2}{\partial x_i\partial x_j}+\sum\limits_{i}b_i(x)\frac{\partial}{\partial x_i},\ \ \nabla:=(\frac{\partial}{\partial x_1},\cdots,\frac{\partial}{\partial x_l}),\vspace{0.2cm}
$$
and recall the definition of a continuous viscosity solution to \eqref{PDE*} in our framework, see for example \cite{PardouxandRascanu(2014)}.

\begin{dfn}
A continuous function $u:\bar {D}\rightarrow\R$ is called a viscosity sub-solution (super-solution) of PDE \eqref{PDE*} if for any function $\phi\in C^2(\bar {D};\R)$, it holds that
$$
\begin{array}{c}
\mathcal{\bar L}\phi(x_0)+g(x_0,u(x_0),(\nabla\phi\sigma)(x_0))\geq 0\ (\leq 0), \ x_0\in D;\\
\max\{\mathcal{\bar L}\phi(x_0)+g(x_0,u(x_0),(\nabla\phi\sigma)(x_0)),
h(x_0)-u(x_0)\}\geq 0, \ x_0\in \partial D\\
(\min\{\mathcal{\bar L}\phi(x_0)+g(x_0,u(x_0),(\nabla\phi\sigma)(x_0)),
h(x_0)-u(x_0)\}\leq 0, \ x_0\in \partial D),
\end{array}
$$
provided that $u-\phi$ attains a local maximum (minimum) at point $x_0\in \R^l$. A continuous function $u$ is said to be a viscosity solution of PDE \eqref{PDE*} if it is both a viscosity sub-solution and a viscosity super-solution.
\end{dfn}

Now we shall give the probabilistic interpretation for elliptic PDE \eqref{PDE*}.

\begin{thm}\label{thm:3.4.2}
Under the above assumptions, $u(x):=Y_0^{x},\ x\in \R^l$ is a continuous function on $\bar {D}$ and it is a viscosity solution of PDE \eqref{PDE*}.
\end{thm}

\begin{proof}
First of all, it follows from Proposition 5.76 in \cite{PardouxandRascanu(2014)} that under the condition \eqref{closed}, the mapping $x\rightarrow \tau_x$ is almost surely continuous on $\bar {D}$. Then, from \cref{thm:3.2} and the continuity of $X_\cdot^{x}$ with respect to $x$, we can conclude that $u(x):=Y_0^x$ is a continuous function on $\bar {D}$.

We prove only that $u(x)$ is a viscosity sub-solution of  PDE \eqref{PDE*}. The proof of the other statement follows by a similar argument. Take any function $\phi\in C^2(\bar {D};\R)$ such that $u-\phi$ attains a local maximum at $x_0\in \bar {D}$. If $x_0\in \Gamma$, then $\tau_{x_0}=0$, and hence $u(x_0)=h(x_0)$. If $x_0\in D\cup (\partial D\cap \Gamma^c)$, then $\tau_{x_0}>0$. For the latter case, we assume without loss of generality that
$$
u(x_0)=\phi(x_0).\vspace{-0.1cm}
$$
We now suppose that
$$\mathcal{\bar L}\phi(x_0)+g(x_0,u(x_0),(\nabla\phi\sigma)(x_0))<0,$$
and we will find a contradiction. In fact, by continuity let $\delta,\gamma>0$ such that whenever $|y-x_0|\leq \delta,$
\begin{align}\label{eq:3.43}
\left\{\begin{aligned}
&u(y)\leq \phi(y);\\
&\mathcal{\bar L}\phi(y)+g(y,u(y),(\nabla\phi\sigma)(y))\leq -\gamma.
\end{aligned}\right.
\end{align}
Define the following stopping time:
$$\vartheta:=\inf\{t\geq 0:|X_t^{x_0}-x_0|+\int_0^t \bar\sigma^2(X_s^{x_0}){\rm d}s\geq \delta\}\wedge \tau_{x_0}\wedge \delta>0.$$
By the Markov property and BSDE \eqref{BSDE-2}, we know that $(\bar {Y}_t,\bar {Z}_t):=(Y^{x_0}_t, Z^{x_0}_t)$ is an adapted solution of the following BSDE:
\begin{align}\label{BSDE3.38}
\bar Y_t=u(X_{\vartheta}^{x_0})+\int_t^{\vartheta} g(X_s^{x_0},u(X_s^{x_0}),\bar Z_s){\rm d}s-\int_t^{\vartheta} \bar Z_s{\rm d}B_s,\ \ t\in [0,\vartheta].
\end{align}
On the other hand, it follows from It\^{o}'s formula that
$(\widetilde{Y}_t,\widetilde{Z}_t):=(\phi(X_t^{x_0}),
(\nabla\phi\sigma)(X_t^{x_0}))$ is an adapted solution of the following BSDE:
\begin{align}\label{BSDE3.39}
\widetilde{Y}_t=\phi(X_{\vartheta}^{x_0})-
\int_t^{\vartheta} \mathcal{\bar L}\phi(X_{s}^{x_0}){\rm d}s-\int_t^{\vartheta} \widetilde{Z}_s{\rm d}B_s,\ \ t\in [0,\vartheta].
\end{align}
In light of \eqref{eq:3.43}, \eqref{BSDE3.38} and \eqref{BSDE3.39} along with the definition of the stopping time $\vartheta$ and by using the Lipschitz continuity for $g$ in $z$ and Girsanov's theorem, a similar argument as in the proof of \cref{thm:3.4.1} yields that $\widetilde{Y}_0>\bar Y_0$, i.e., $\phi(x_0)>u(x_0) $ that is a contradiction. The proof is complete.
\end{proof}

\begin{rmk}
We would like to emphasize that both generators of \eqref{BSDE-1} and \eqref{BSDE-2} satisfy the stochastic monotonicity condition in the unknown variable $y$ and the stochastic Lipschitz continuity condition in the unknown variable $z$ due to the presence of functions $\bar b$ and $\bar\sigma$ in assumptions \eqref{condition-1} and \eqref{condition-2}. Under the above situation, Feynman-Kac formula for PDEs have been investigated for example in \cite{Bahlali2004} and \cite{PardouxandRascanu(2014)}. However, our assumptions \eqref{condition-1} and \eqref{condition-2} on the function $g$ in PDEs are obviously more general than theirs. In addition, we mention that Theorems \ref{thm:3.4.1} and \ref{thm:3.4.2} can be easily extended to the multidimensional case in the spirit of \cite{PardouxandRascanu(2014)}.
\end{rmk}

\section{The proof of the existence part of \cref{thm:3.1}}
\setcounter{equation}{0}

In this section, we will give the proof of the existence part of \cref{thm:3.1}. We first consider the case that $g$ is independent of $z$, and then the general case. 

\subsection{The case of $g$ independent of $z$}

In this subsection, inspired by \cite{Li2023}, \cite{Xiao2015} and \cite{XiaoandFan2017} and combined with some new ideas, we propose and prove the following \cref{pro:2.1}. This proposition is a special case of \cref{thm:3.1} when the generator $g$ does not depend on $z$.

\begin{pro}\label{pro:2.1}
Assume that $g$ is independent of $z$ and satisfies the following assumptions \ref{A:H1'}-\ref{A:H4'}:
\begin{enumerate}
\renewcommand{\theenumi}{(H1')}
\renewcommand{\labelenumi}{\theenumi}
\item\label{A:H1'} $\Dis \E\left[\left(\int_0^\tau e^{\beta \int_0^s\mu_r{\rm d}r}|g(s,0)|{\rm d}s\right)^2\right]<+\infty;$
\renewcommand{\theenumi}{(H2')}
\renewcommand{\labelenumi}{\theenumi}
\item\label{A:H2'} $\as$, $g(\omega,t,\cdot)$ is continuous;
\renewcommand{\theenumi}{(H3')}
\renewcommand{\labelenumi}{\theenumi}
\item\label{A:H3'} $g$ has a general growth in $y$, i.e., for each $r\in \R_+$, it holds that
$$\E\left[\int_0^\tau e^{\beta \int_0^t\mu_s{\rm d}s}\overline{\psi}_{r}^{\alpha_\cdot}(t)dt\right]<+\infty$$
with
$$\overline{\psi}_{r}^{\alpha_\cdot}(t):=\sup_{|y|\leq r\alpha_t} \left|g(t,y)-g(t,0)\right|, \ t\in[0,\tau],\vspace{0.1cm}$$
where $(\alpha_t)_{t\in[0,\tau]}$ is defined in assumption
    \ref{A:H3};
\renewcommand{\theenumi}{(H4')}
\renewcommand{\labelenumi}{\theenumi}
\item\label{A:H4'} $g$ satisfies a stochastic monotonicity condition in $y$, i.e., for each $y_1, y_2\in\R^k$,
$$\left\langle y_1-y_2,g(\omega,t,y_1)-g(\omega,t,y_2)\right\rangle\leq \mu_t(\omega)|y_1-y_2|^2, \ t\in[0,\tau].$$
\end{enumerate}
Then, for each $\xi\in L_\tau^2(\beta\mu_\cdot;\R^k)$, the following BSDE
\begin{align}\label{BSDE2.1}
  y_t=\xi+\int_t^\tau g(s,y_s){\rm d}s-\int_t^\tau z_s{\rm d}B_s, \ \ t\in[0,\tau]
\end{align}
admits a weighted $L^2$-solution in the space of $H_\tau^2(\beta\mu_\cdot;\R^{k}\times\R^{k\times d})$.
\end{pro}

\begin{rmk}\label{rmk:4.2}
Assume that $g$ is independent of $z$ and satisfies \ref{A:H1'}-\ref{A:H4'}. It is not hard to verify that $(y_t,z_t)_{t\in[0,\tau]}$ is a weighted $L^2$-solution of BSDE \eqref{BSDE2.1} in the space of $H_\tau^2(\beta\mu_\cdot;\R^{k}\times\R^{k\times d})$ if and only if
$$
(\bar y_t,\bar z_t)_{t\in[0,\tau]}:=(e^{\beta \int_0^t\mu_r{\rm d}r}y_t,e^{\beta \int_0^t\mu_r{\rm d}r}z_t)_{t\in[0,\tau]}
$$
is a weighted-$L^2$ solution of BSDE$(e^{\beta \int_0^t\mu_r{\rm d}r}\xi,\tau,\bar g)$ in the space of $H_\tau^2(0;\R^{k}\times\R^{k\times d})$, where
$$
\bar g(t,y):=e^{\beta \int_0^t\mu_r{\rm d}r}g(t,e^{-\beta \int_0^t\mu_r{\rm d}r}y)-\beta \mu_t y,\ \ (\omega,t,y)\in \Omega\times [0,\tau]\times \R^k,
$$
and that, in view of $\beta\geq 1$, the generator $\bar g$ satisfies \ref{A:H1'}, \ref{A:H2'} and \ref{A:H4'} with $0$ instead of $\mu_\cdot$ Note that for each $r\in \R_+$,
$$
\begin{array}{l}
\Dis \sup_{|y|\leq r}|\bar g(t,y)-\bar g(t,0)|= \sup_{|y|\leq r}
\left|e^{\beta \int_0^t\mu_r{\rm d}r}g(t,e^{-\beta \int_0^t\mu_r{\rm d}r}y)-e^{\beta \int_0^t\mu_r{\rm d}r}g(t,0)-\beta \mu_t y \right|\vspace{0.1cm}\\
\hspace{3.3cm} \Dis \leq e^{\beta \int_0^t\mu_r{\rm d}r} \sup_{|y|\leq r}
\left|g(t,e^{-\beta \int_0^t\mu_r{\rm d}r}y)-g(t,0)\right|+\beta \mu_t  \sup_{|y|\leq r}|y|\vspace{0.1cm}\\
\hspace{3.3cm} \Dis =e^{\beta \int_0^t\mu_r{\rm d}r} \sup_{|y|\leq r e^{-\beta \int_0^t\mu_r{\rm d}r}}
\left|g(t,y)-g(t,0)\right|+\beta \mu_t r.
\end{array}
$$
In view of \ref{A:H3'} of $g$, in order to guarantee the following condition holds:
$$
\RE r\in \R_+,\ \ \E\left[\int_0^\tau \sup_{|y|\leq r}|\bar g(t,y)-\bar g(t,0)| {\rm d}t\right]<+\infty,
$$
which is necessary in \cite{Briand2003} and \cite{Xiao2015}, we need to further assume that
$$\E\left[\int_0^\tau \mu_t{\rm d}t\right]<+\infty,$$
and the process $\alpha_\cdot$ in assumption \ref{A:H3'} for $g$ needs to satisfy that $$\alpha_t\geq e^{-\beta \int_0^t\mu_s{\rm d}s},\ \ t\in [0,\tau].$$
Therefore, \cref{pro:2.1} can not be proved by using the method of exponential translation transformation via the corresponding results obtained in \cite{Briand2003} and \cite{Xiao2015}.
\end{rmk}

\begin{proof}[\bf Proof of \cref{pro:2.1}.]
Note that if assumption \ref{A:H3'} holds for the process $\alpha_\cdot$, then it is also true for $\overline{\alpha}_t:=\alpha_t\wedge (e^{-\beta \int_0^t\mu_s{\rm d}s}),\ t\in [0,\tau]$. Therefore, without loss of generality, we can assume that the process $\alpha_\cdot$ in assumption \ref{A:H3'} satisfies that
\begin{align}\label{*}
\alpha_t\leq e^{-\beta \int_0^t\mu_s{\rm d}s}\leq 1, \ t\in[0,\tau].
\end{align}
The following proof will be divided into three steps.\vspace{0.2cm}

{\bf First Step:} We first prove that BSDE \eqref{BSDE2.1} has a weighted $L^2$-solution in $H_\tau^2(\beta\mu_\cdot;\R^{k}\times\R^{k\times d})$ under assumptions \ref{A:H2'} and \ref{A:H4'}, provided that there exists a nonnegative constant $K$ such that
\begin{align}\label{2.6}
\begin{split}
|\xi|\leq K\alpha_\tau, \ |g(t,y)|\leq Ke^{-t}\alpha_t, \ t\in[0,\tau].\vspace{0.2cm}
\end{split}
\end{align}

For each $n\geq1$, let $\phi_n(x):=n^k\phi(nx)$, where $\phi(\cdot):\R^k\mapsto\R_+$ is a nonnegative $\mathcal{C}^\infty$ function with the unit ball as compact support and satisfying $\int_{\R^k}\phi(x){\rm d}x=1$. We define for each $n\geq1$ and $y\in \R^k$,
\begin{align}\label{2.100}
\begin{split}
g_n(t,y):&=\left(\phi_n(\cdot)\ast g(t,\cdot)\right)(y)\\
&=\int_{\R^k}\phi_n(x)g(t,y-x){\rm d}x=\int_{\R^k}\phi_n(y-x)g(t,x){\rm d}x\\
&=\int_{\R^k}\phi(x)g\left(t,y-\frac{x}{n}\right){\rm d}x=\int_{x:|x|\leq1}\phi(x)g\left(t,y-\frac{x}{n}\right){\rm d}x, \ t\in[0,\tau].
\end{split}
\end{align}
Then, $g_n(t,y)$ is a $(\F_t)$-progressively measurable process for each $y\in \R^k$ and it is straightforward to prove that $g_n$ satisfies assumptions \ref{A:H2'} and \ref{A:H4'} and tends locally uniformly to the generator $g$ as $n\rightarrow\infty$. Additionally, it follows from
\eqref{2.6} that for each $n\geq1$ and $y\in\R^k$,
$$|\nabla g_n(t,y)|\leq Ke^{-t}\alpha_t\int_{\R^k}|\nabla \phi_n(y-x)|{\rm d}x=Ke^{-t}\alpha_t\int_{z:|z|\leq{\frac{1}{n}}}|\nabla\phi_n(z)|{\rm d}z\leq K_n'e^{-t}, \ t\in[0,\tau],$$
where $K_n'>0$ is a constant depending only on $n$ and $K$. Thus, we have for each $n\geq1$ and $y_1, y_2\in\R^k$,
$$|g_n(t,y_1)-g_n(t,y_2)|\leq K_n'e^{-t}|y_1-y_2|, \ t\in[0,\tau].$$
Furthermore, according to \eqref{2.6} and \eqref{2.100}, for each $n\geq1$ and $y\in\R^k$ we have
\begin{align}\label{2.10}
|g_n(t,y)|\leq Ke^{-t}\alpha_t\leq Ke^{-t}, \ t\in[0,\tau],
\end{align}
and then
$$
\E\left[\left(\int_0^\tau|g_n(s,y)|{\rm d}s\right)^2\right]<+\infty.\vspace{0.1cm}
$$
Thus, for each $n\geq1$, the following BSDE
\begin{align}\label{BSDE2.13}
  y_t^n=\xi+\int_t^\tau g_n(s,y_s^n){\rm d}s-\int_t^\tau z_s^n{\rm d}B_s, \ \ t\in[0,\tau],
\end{align}
admits a unique solution $(y_t^n,z_t^n)_{t\in[0,\tau]}$ in the space of  $S_\tau^2(0;\R^k)\times M_\tau^2(0;\R^{k\times d})$. In fact, define $\overline{g}_{n,\tau}(t,y):=g_n(t,y){\bf 1}_{t\leq\tau}$ for $(\omega,t,y,z)\in \Omega\times[0,+\infty)\times \R^k\times \R^{k\times d}$. According to Theorem 1.2 in \cite{ZChenBWang2000JAMS}, we know that for each $n\geq1$, the following BSDE
\begin{align}\label{BSDE*2.13}
  \overline{y}_t^n=\xi+\int_t^{+\infty} \overline{g}_{n,\tau}(s,\overline{y}_s^n){\rm d}s-\int_t^{+\infty} \overline{z}_s^n{\rm d}B_s, \ \ t\geq 0,
\end{align}
admits a unique solution $(\overline{y}_t^n,\overline{z}_t^n)_{t\geq 0}$ in
$S_{+\infty}^2(0;\R^k)\times M_{+\infty}^2(0;\R^{k\times d})$. Then,
\begin{align*}
\overline{y}_\tau^n&=\xi+\int_\tau^{+\infty} g_n(s,\overline{y}_s^n){\bf 1}_{s\leq\tau}{\rm d}s-\int_\tau^{+\infty} \overline{z}_s^n{\rm d}B_s=\xi-\int_\tau^{+\infty} \overline{z}_s^n{\rm d}B_s.
\end{align*}
Taking conditional expectation with respect to $\F_\tau$ in the last identity yields that $\overline{y}_\tau^n=\E[\xi|\F_\tau]=\xi$ and
$$\int_\tau^{+\infty} \overline{z}_s^n{\rm d}B_s=\int_0^{+\infty} \overline{z}_s^n{\bf 1}_{s>\tau}{\rm d}B_s=0,$$
which means that $\E[\int_0^{+\infty} |\overline{z}_s^n{\bf 1}_{s>\tau}|^2{\rm d}s]=0$. Thus, $\overline{z}_t^n{\bf 1}_{t>\tau}=0$ and then $\overline{y}_t^n{\bf 1}_{t>\tau}=\xi$ for each $n\geq1$. Furthermore, by \eqref{BSDE*2.13} we have
\begin{align}
  \overline{y}_t^n=\xi+\int_t^{\tau} g_n(s,\overline{y}_s^n){\rm d}s-\int_t^{\tau} \overline{z}_s^n{\rm d}B_s, \ \ t\in[0,\tau].
\end{align}
Consequently, for each $n\geq1$, $(y_t^n,z_t^n)_{t\in[0,\tau]}:=(\overline{y}_t^n,\overline{z}_t^n)_{t\in[0,\tau]}$
is the desired unique solution of \eqref{BSDE2.13}.

In the sequel, according to \eqref{*}, \eqref{2.6}, \eqref{2.10} and \eqref{BSDE2.13}, we deduce that for each $n\geq1$ and $t\geq0$,
\begin{align}\label{4.16**}
\begin{split}
|y_{t\wedge\tau}^n|e^{\beta \int_0^{t\wedge\tau}\mu_s{\rm d}s}&=|\E\left[y_{t\wedge\tau}^n|\F_{t\wedge\tau}\right]|e^{\beta \int_0^{t\wedge\tau}\mu_s{\rm d}s}\leq\E\left[|\xi|+\int_{t\wedge\tau}^\tau |g_n(s,y_s^n)|{\rm d}s\bigg|\F_{t\wedge\tau}\right]e^{\beta \int_0^{t\wedge\tau}\mu_s{\rm d}s}\\
&\leq \E\left[e^{\beta \int_0^\tau \mu_s{\rm d}s}|\xi|\bigg|\F_{t\wedge\tau}\right]+\E\left[\int_{t\wedge\tau}^\tau e^{\beta \int_0^s\mu_r{\rm d}r}|g_n(s,y_s^n)|{\rm d}s\bigg|\F_{t\wedge\tau}\right]\\
&\leq K+\E\left[\int_0^\tau Ke^{-s}{\rm d}s\bigg|\F_{t\wedge\tau}\right]\leq2K.
\end{split}
\end{align}
Hence, $y_\cdot^n\in S_\tau^2(\beta\mu_\cdot;\R^k)$. Furthermore, it follows from \eqref{2.10} that for each $n\geq1$ and $y\in\R^k$, $$\left<\hat{y},g_n(t,y)\right>\leq Ke^{-t}\alpha_t\leq \mu_t|y|+Ke^{-t}\alpha_t, \ t\in[0,\tau].$$
This implies that assumption (A) is satisfied by $g_n(t,y)$ with $\bar\mu_t=\mu_t$, $\bar\nu_t=0$, $f_t=Ke^{-t}\alpha_t$. Then, in view of \eqref{4.16**}, by \cref{pro:1.1} we can conclude that for each $n\geq1$, $z_\cdot^n\in M_\tau^2(\beta\mu_\cdot;\R^k)$, and then $(y_\cdot^n,z_\cdot^n)\in H_\tau^2(\beta\mu_\cdot;\R^{k}\times\R^{k\times d})$.\vspace{0.2cm}

Now, we show that $\{(y_\cdot^n,z_\cdot^n)\}^{+\infty}_{n=1}$ is a Cauchy sequence in $H_\tau^2(\beta\mu_\cdot;\R^{k}\times\R^{k\times d})$. For each $n,m\geq1$, let $\hat{y}_\cdot^{n,m}:=y_\cdot^n-y_\cdot^m$, $\hat{z}_\cdot^{n,m}:=z_\cdot^n-z_\cdot^m$. Then,
\begin{align}\label{BSDE:3.5}
\hat y_t^{n,m}=\int_t^\tau \hat{g}^{n,m}(s,\hat y_s^{n,m}){\rm d}s-\int_t^\tau\hat z_s^{n,m}{\rm d}B_s, \ t\in[0,\tau],
\end{align}
where for each $y\in \R^k$, $$\hat{g}^{n,m}(s,y):=g_{n}(s,y+y_s^m)-g_m(s,y_s^m), \ s\in[0,\tau].$$
It follows from the assumption \ref{A:H4'} of $g_n$ that
\begin{align*}
&\left<\hat{y},\hat{g}^{n,m}(t,y)\right>=\left<\hat{y},g_{n}(t,y+y_t^m)
-g_m(t,y_t^m)\right>\\
&\ \ = \left<\hat{y},g_{n}(t,y+y_t^m)-g_n(t,y_t^m)+g_n(t,y_t^m)-g_m(t,y_t^m)\right>\\
&\ \ \leq  \mu_t|y|+|g_n(t,y_t^m)-g_m(t,y_t^m)|, \ t\in[0,\tau].
\end{align*}
Hence, in view of \eqref{2.10}, the generator $\hat{g}^{n,m}$ satisfies assumption (A) with $\bar\mu_t=\mu_t$, $\bar\nu_t=0$ and $f_t=|g_n(t,y_t^m)-g_m(t,y_t^m)|$. Then, by \eqref{2.03} of \cref{pro:1.1} with $r=t=0$ we deduce that there exists a uniform constant $C>0$ such that
\begin{align}\label{2.101}
\begin{split}
&\E\left[\sup_{s\in[0,\tau]}\left(e^{2\beta \int_0^s\mu_r{\rm d}r}|\hat{y}_s^{n,m}|^2\right)\right]+\E\left[\int_{0}^{\tau}e^{2\beta \int_0^s\mu_r{\rm d}r}|\hat{z}_s^{n,m}|^2{\rm d}s\right]\\
&\ \ \leq C\E\left[\left(
\int_{0}^{\tau}e^{\beta \int_0^s\mu_r{\rm d}r}|g_n(s,y_s^m)-g_m(s,y_s^m)|{\rm d}s\right)^2\right].
\end{split}
\end{align}
Furthermore, from \eqref{2.100} we know that for each $n, m\geq1$ and $s\in[0,\tau]$,
$$|g_n(s,y_s^m)-g_m(s,y_s^m)|\leq\int_{x:|x|\leq1}\phi(x)\bigg|g\left(s,y_s^m-\frac{x}{n}\right)-g\left(s,y_s^m-\frac{x}{m}\right)\bigg|{\rm d}x.$$
In view of \ref{A:H2'}, \eqref{4.16**} and the fact that a function which is continuous on a compact set is uniformly continuous on the set, we deduce that for each $x\in \R^k$ as $n, m\rightarrow+\infty$, $$\bigg|g\left(s,y_s^m-\frac{x}{n}\right)-g\left(s,y_s^m-\frac{x}{m}\right)\bigg|\rightarrow0, \ s\in[0,\tau].$$
Additionally, \eqref{2.6} indicates that for each $n, m\geq1$ $$\bigg|g\left(s,y_s^m-\frac{x}{n}\right)-g\left(s,y_s^m-\frac{x}{m}\right)\bigg|\leq 2Ke^{-s}\alpha_s, \ s\in[0,\tau].$$
Consequently, by using Lebesgue's dominated convergence theorem twice, it can be shown that the right-hand side of inequality \eqref{2.101} tends to 0 as $n, m\rightarrow+\infty$, and then
\begin{align*}
\begin{split}
\lim\limits_{n,m\rightarrow\infty}\left\{\E\left[\sup_{s\in[0,\tau]}\left(e^{2\beta \int_0^s\mu_r{\rm d}r}|\hat{y}_s^{n,m}|^2\right)\right]+\E\left[\int_{0}^{\tau}e^{2\beta \int_0^s\mu_r{\rm d}r}|\hat{z}_s^{n,m}|^2{\rm d}s\right]\right\}=0.
\end{split}
\end{align*}
This means that $\{(y_\cdot^n,z_\cdot^n)\}^{+\infty}_{n=1}$ is a Cauchy sequence in $H_\tau^2(\beta\mu_\cdot;\R^{k}\times\R^{k\times d})$.\vspace{0.2cm}

Finally, we denote by $(y_t,z_t)_{t\in[0,\tau]}$ the limit of the Cauchy sequence $\{(y_t^n,z_t^n)_{t\in[0,\tau]}\}_{n=1}^\infty$ in the space of $H_\tau^2(\beta\mu_\cdot;\R^{k}\times\R^{k\times d})$, and pass to the limit under the uniform convergence in probability (ucp for short) for BSDE \eqref{BSDE2.13}, by combining \eqref{2.100}, \eqref{2.10}, \eqref{4.16**}, \ref{A:H2'} and Lebesgue's dominated convergence theorem, to see that $(y_t,z_t)_{t\in[0,\tau]}\in H_\tau^2(\beta\mu_\cdot;\R^{k}\times\R^{k\times d})$ is a weighted $L^2$-solution of BSDE \eqref{BSDE2.1}.\vspace{0.2cm}

{\bf Second Step:} Under assumptions \ref{A:H2'}-\ref{A:H4'}, we prove that BSDE \eqref{BSDE2.1} admits a weighted $L^2$-solution in $H_\tau^2(\beta\mu_\cdot;\R^{k}\times\R^{k\times d})$, provided that there exists a nonnegative constant $K$ such that
\begin{align}\label{22.6}
\begin{split}
|\xi|\leq K \alpha_\tau^2, \ |g(t,0)|\leq Ke^{-t}\alpha_t^2, \ t\in[0,\tau],
\end{split}
\end{align}
where $\alpha_t$ is defined in \ref{A:H3'} and satisfies \eqref{*}.\vspace{0.2cm}

Assume now that \ref{A:H2'}-\ref{A:H4'} and \eqref{22.6} hold. For some fixed positive real $r>0$ assigned later, we define the following function: for each $u\geq0$ and $t\in[0,\tau]$,
\begin{align}\label{Matrix}
\theta_{r}^{\alpha_\cdot}(t,u):=\left\{\begin{aligned}
&\alpha_t, \quad \quad \quad \quad \quad \quad \quad 0\leq u\leq r\alpha_t;\\
&-u+(r+1)\alpha_t, \quad r\alpha_t< u\leq(r+1)\alpha_t;\\
&0, \quad \quad \quad \quad \quad \quad \quad \quad u>(r+1)\alpha_t.
\end{aligned}\right.
\end{align}
It is clear that for each $r>0$, $u,u_1,u_2\geq0$ and $t\in[0,\tau]$,
\begin{align}\label{662.25}
0\leq\theta_{r}^{\alpha_\cdot}(t,u)\leq \alpha_t,
\end{align}
and
\begin{align}\label{2.25}
|\theta_{r}^{\alpha_\cdot}(t,u_1)-\theta_{r}^{\alpha_\cdot}(t,u_2)| \leq|u_1-u_2|.
\end{align}
Thus we define for each $n\geq1$ and $y\in \R^k$,
$$g^n(t,y):=\theta_{r}^{\alpha_\cdot}(t,|y|)\left(g(t,y)-g(t,0)\right) \frac{n{e^{-t}}}{{\overline{\psi}_{r+1}^{\alpha_\cdot}(t)}\vee \left(ne^{-t}\alpha_t\right)}+g(t,0), \ t\in[0,\tau].$$
Then by the definitions of $\overline{\psi}_{r+1}^{\alpha_\cdot}(\cdot)$ and $\theta_{r}^{\alpha_\cdot}$ togher with \eqref{662.25} and \eqref{22.6}, we can conclude that for each $n\geq1$ and $y\in \R^k$,
$$|g^n(t,y)|\leq ne^{-t}\alpha_t+Ke^{-t}\alpha_t^2\leq(n+K)e^{-t}\alpha_t, \ t\in[0,\tau].$$
Clearly, $g^n$ satisfies \ref{A:H2'} for each $n\geq1$. Moreover, we can prove that $g^n(t,y)$ satisfies \ref{A:H4'} with a process $\tilde{\mu}_\cdot$ instead of $\mu_\cdot$ as follows. In fact, for each pair of $y_1, y_2\in \R^k$ and fixed $t\in[0,\tau]$, if $|y_1|>(r+1)\alpha_t$ and $|y_2|>(r+1)\alpha_t$, \ref{A:H4'} is trivially satisfied and thus we reduce to the case where $|y_2|\leq (r+1)\alpha_t$. Then we can deduce that for each $n\geq1$,
\begin{align}\label{4.35*}
\begin{split}
&\left\langle y_1-y_2,g^n(t,y_1)-g^n(t,y_2)\right\rangle\\
&\ \ = \frac{ne^{-t}}{{\overline{\psi}_{r+1}^{\alpha_\cdot}(t)}\vee \left(ne^{-t}\alpha_t\right)}\theta_{r}^{\alpha_\cdot}(t,|y_1|)\left\langle y_1-y_2,g(t,y_1)-g(t,y_2)\right\rangle\\
& \hspace{0.6cm}+\frac{ne^{-t}}{{\overline{\psi}_{r+1}^{\alpha_\cdot}(t)}\vee \left(ne^{-t}\alpha_t\right)}\left(\theta_{r}^{\alpha_\cdot}(t,|y_1|) -\theta_{r}^{\alpha_\cdot}(t,|y_2|)\right)
\cdot \left\langle y_1-y_2,g(t,y_2)-g(t,0)\right\rangle, \ t\in[0,\tau].
\end{split}
\end{align}
By virtue of \eqref{662.25} and \ref{A:H4'} for $g$ we know that the first term on the right-hand side of the last inequality is equal or lesser than $\mu_t|y_1-y_2|^2$. For the second term on the right-hand side, it follows from \eqref{2.25}, triangle inequality and the definition of $\overline{\psi}_{r+1}^{\alpha_\cdot}(\cdot)$ that
$$|\theta_{r}^{\alpha_\cdot}(t,|y_1|) -\theta_{r}^{\alpha_\cdot}(t,|y_2|)|\leq||y_1|-|y_2||\leq|y_1-y_2|$$
and
$$\left\langle y_1-y_2,g(t,y_2)-g(t,0)\right\rangle\leq|y_1-y_2||g(t,y_2)-g(t,0)| \leq|y_1-y_2|\overline{\psi}_{r+1}^{\alpha_\cdot}(t).$$
This implies that the second term on the right-hand side of \eqref{4.35*} is equal or lesser than $ne^{-t}|y_1-y_2|^2$. Thus, we have that for each $n\geq1$ and $y_1,y_2\in \R^k$,
$$\left\langle y_1-y_2,g^n(t,y_1)-g^n(t,y_2)\right\rangle\leq \tilde{\mu}_t|y_1-y_2|^2, \ t\in[0,\tau],$$
where $\tilde{\mu}_t:=\mu_t+ne^{-t}$. Hence, $g^n$ also satisfies \ref{A:H4'} with $\tilde{\mu}_\cdot$ instead of $\mu_\cdot$ for each $n\geq1$. Since the difference between $\tilde{\mu}_t$ and $\mu_t$ is a deterministic function $ne^{-t}$, the space $H_\tau^2(\beta\mu_\cdot;\R^{k}\times\R^{k\times d})$ does not change when the process $\mu_t$ is replaced with $\tilde{\mu}_t$, We can conclude from the conclusion of the first step that for each $n\geq1$, BSDE$(\xi,\tau,g^n)$ admits a weighted $L^2$-solution $(y_t^n,z_t^n)_{t\in[0,\tau]}$ in $H_\tau^2(\beta\mu_\cdot;\R^{k}\times\R^{k\times d})$.\vspace{0.2cm}

Furthermore, \ref{A:H4'} for $g$ together with \eqref{662.25} and \eqref{22.6} indicates that for $n\geq1$ and $y\in \R^k$,
\begin{align*}
\begin{split}
\left\langle \hat{y},g^n(t,y)\right\rangle=&\left\langle \hat{y},g(t,y)-g(t,0)\right\rangle\theta_{r}^{\alpha_\cdot}(t,|y|) \frac{ne^{-t}}{{\overline{\psi}_{r+1}^{\alpha_\cdot}(t)}\vee \left(ne^{-t}\alpha_t\right)}+\left\langle \hat{y},g(t,0)\right\rangle\\
\leq&\mu_t|y|+Ke^{-t}\alpha_t^2, \ t\in[0,\tau].
\end{split}
\end{align*}
In view of the definition of $\alpha_\cdot$, \eqref{*} and \eqref{22.6}, by \eqref{2.03} of \cref{pro:1.1} with $\bar\mu_t=\mu_t$, $\bar\nu_t=0$  and $f_t=Ke^{-t}\alpha_t^2$ we get that there exists a constant $C>0$ such that for each $n\geq1$ and $t\geq0$,
\begin{align}\label{22.101}
\begin{split}
&|y_{t\wedge\tau}^{n}|^2\leq e^{2\beta \int_0^{t\wedge\tau}\mu_s{\rm d}s}|y_{t\wedge\tau}^{n}|^2\\
&\ \ \leq  C\left(\E\left[e^{2\beta \int_0^\tau\mu_s{\rm d}s}|\xi|^2\bigg|\F_{t\wedge\tau}\right]+\E\left[\left(
\int_{t\wedge\tau}^{\tau}e^{\beta \int_0^s\mu_r{\rm d}r}Ke^{-s}\alpha_s^2{\rm d}s\right)^2\bigg|\F_{t\wedge\tau}\right]\right)\\
&\ \ \leq C\left(\E\left[e^{2\beta \int_0^\tau\mu_s{\rm d}s}K^2 \alpha_{\tau}^2\alpha_{t\wedge\tau}^2\bigg|\F_{t\wedge\tau}\right] +\E\left[\left(\int_{t\wedge\tau}^{\tau}e^{\beta \int_0^s\mu_r{\rm d}r}Ke^{-s}\alpha_s\alpha_{t\wedge\tau}{\rm d}s\right)^2\bigg|\F_{t\wedge\tau}\right]\right)\\
&\ \ \leq C\left(\E\left[e^{2\beta \int_0^\tau\mu_s{\rm d}s}K^2 \alpha_{\tau}^2\bigg|\F_{t\wedge\tau}\right] +\E\left[\left(\int_{t\wedge\tau}^{\tau}e^{\beta \int_0^s\mu_r{\rm d}r}Ke^{-s}\alpha_s{\rm d}s\right)^2\bigg|\F_{t\wedge\tau}\right]\right)\alpha_{t\wedge\tau}^2\\
&\ \ \leq r^2\alpha_{t\wedge\tau}^2,
\end{split}
\end{align}
where $r^2:=2K^2C$.
Then for each $n\geq1$ and $t\in[0,\tau]$, we have
\begin{align}\label{2.32}
|y_t^n|\leq r\alpha_t.
\end{align}
Thus, in light of the definition of $\theta_{r}^{\alpha_\cdot}$, $g^n$ can be replaced with
$$\tilde{g}^n(t,y):=\left(g(t,y)-g(t,0)\right) \frac{ne^{-t}\alpha_t}{{\overline{\psi}_{r+1}^{\alpha_\cdot}(t)}\vee \left(ne^{-t}\alpha_t\right)}+g(t,0), \ t\in[0,\tau],$$
and $(y_t^n,z_t^n)_{t\in[0,\tau]}\in H_\tau^2(\beta\mu_\cdot;\R^{k}\times\R^{k\times d})$ is a weighted $L^2$-solution of BSDE$(\xi,\tau,\tilde{g}^n)$.

Next, we prove that  $\{(y_\cdot^n,z_\cdot^n)\}^{+\infty}_{n=1}$ is a Cauchy sequence in $H_\tau^2(\beta\mu_\cdot;\R^{k}\times\R^{k\times d})$. For each $n,i\geq1$, let $\hat{y}_\cdot^{n,i}:=y_\cdot^{n+i}-y_\cdot^n$, $\hat{z}_\cdot^{n,i}:=z_\cdot^{n+i}-z_\cdot^n$. Then we have
\begin{align}\label{BSDE:3.5}
\hat y_t^{n,i}=\int_t^\tau \tilde{g}^{n,i}(s,\hat y_s^{n,i}){\rm d}s-\int_t^\tau\hat z_s^{n,i}{\rm d}B_s, \ t\in[0,\tau],
\end{align}
where for each $y\in \R^k$, $$\tilde{g}^{n,i}(s,y):=\tilde{g}^{n+i}(s,y+y_s^n)-\tilde{g}^{n}(s,y_s^n).$$
From the definition of $\tilde{g}^n$, we can deduce that for each $n, i\geq1$, $y\in \R^k$ and $t\in[0,\tau]$,
\begin{align}\label{21010}
\begin{split}
\left\langle \hat{y},\tilde{g}^{n,i}(t,y)\right\rangle=&\left\langle \hat{y},\tilde{g}^{n+i}(t,y+y_t^n)-\tilde{g}^{n+i}(t,y_t^n)+\tilde{g}^{n+i}(t,y_t^n)-\tilde{g}^{n}(t,y_t^n)\right\rangle\\
=&\left\langle \hat{y},g(t,y+y_t^n)-g(t,y_t^n)\right\rangle\frac{(n+i)e^{-t}\alpha_t}{{\overline{\psi}_{r+1}^{\alpha_\cdot}(t)}\vee \left((n+i)e^{-t}\alpha_t\right)}\\
&+\left\langle \hat{y},g(t,y_t^n)-g(t,0)\right\rangle\Psi_{r,n,i}^{\alpha_\cdot}(t),
\end{split}
\end{align}
where
$$
\Psi_{r,n,i}^{\alpha_\cdot}(t)=\frac{(n+i)e^{-t}\alpha_t}{{\overline{\psi}_{r+1}^{\alpha_\cdot}(t)}\vee \left((n+i)e^{-t}\alpha_t\right)}-\frac{ne^{-t}\alpha_t}{{\overline{\psi}_{r+1}^{\alpha_\cdot}(t)}\vee \left(ne^{-t}\alpha_t\right)}.\vspace{0.1cm}
$$
It follows from assumption \ref{A:H4'} that the first term on the right-hand side of \eqref{21010} is equal or lesser than $\mu_t|y|$. It is clear that if $\overline{\psi}_{r+1}^{\alpha_\cdot}(t)\leq ne^{-t}\alpha_t$, then  $\Psi_{r,n,i}^{\alpha_\cdot}(t)=0$; if ${n}e^{-t}\alpha_t<\overline{\psi}_{r+1}^{\alpha_\cdot}(t)\leq (n+i)e^{-t}\alpha_t$, then $0<\Psi_{r,n,i}^{\alpha_\cdot}(t)<1$; if $\overline{\psi}_{r+1}^{\alpha_\cdot}(t)> (n+i)e^{-t}\alpha_t$, then $0<\Psi_{r,n,i}^{\alpha_\cdot}(t)<1$. In addition, since \eqref{2.32} holds, we have that $|g(t,y_t^n)-g(t,0)|\leq\overline{\psi}_{r}^{\alpha_\cdot}(t)\leq\overline{\psi}_{r+1}^{\alpha_\cdot}(t)$. From the above discussions, we can conclude that for each $n, i\geq1$ and $y\in \R^k$,
$$\left\langle \hat{y},\tilde{g}^{n,i}(t,y)\right\rangle\leq \mu_t|y|+\overline{\psi}_{r+1}^{\alpha_\cdot}(t){\bf 1}_{\overline{\psi}_{r+1}^{\alpha_\cdot}(t)>{{n}e^{-t}\alpha_t}}, \ t\in[0,\tau].$$
So $\tilde{g}^{n,i}(t,y)$ satisfies assumption (A) with $\bar\mu_t=\mu_t$, $\bar\nu_t=0$ and  $f_t=\overline{\psi}_{r+1}^{\alpha_\cdot}(t){\bf 1}_{\overline{\psi}_{r+1}^{\alpha_\cdot}(t)>{{n}e^{-t}\alpha_t}}$.
On the other hand, in view of \eqref{22.101}, we have that for each $t\in[0,\tau]$,
\begin{align}\label{662101}
\begin{split}
e^{2\beta \int_0^t\mu_s{\rm d}s}|\hat{y}_t^{n,i}|^2\leq 2r^2,
\end{split}
\end{align}
thus $\hat{y}_\cdot^{n,i}\in H_\tau^2(\beta\mu_\cdot;\R^{k}\times\R^{k\times d})$.
On the basis of \eqref{662101} and \eqref{2.04} in \cref{pro:1.1} with $r=t=0$, we deduce that for each $n, i\geq1$, there exists a uniform constant $C>0$ such that
\begin{align}\label{2101.1}
\begin{split}
&\E\left[\sup_{s\in[0,\tau]}\left(e^{2\beta \int_0^s\mu_r{\rm d}r}|\hat{y}_s^{n,i}|^2\right)\right]+\E\left[\int_{0}^{\tau}e^{2\beta \int_0^s\mu_r{\rm d}r}|\hat{z}_s^{n,i}|^2{\rm d}s\right]\\
&\ \ \leq C\E\left[
\int_{0}^{\tau}e^{2\beta \int_0^s\mu_r{\rm d}r}|\hat{y}_s^{n,i}|\overline{\psi}_{r+1}^{\alpha_\cdot}(s){\bf 1}_{\overline{\psi}_{r+1}^{\alpha_\cdot}(s)>{{n}e^{-s}\alpha_s}}{\rm d}s\right]\\
&\ \ \leq \sqrt{2}rC\E\left[
\int_{0}^{\tau}e^{\beta \int_0^s\mu_r{\rm d}r}\overline{\psi}_{r+1}^{\alpha_\cdot}(s){\bf 1}_{\overline{\psi}_{r+1}^{\alpha_\cdot}(s)>{{n}e^{-s}\alpha_s}}{\rm d}s\right].
\end{split}
\end{align}
According to \ref{A:H3'} and Lebesgue's dominated convergence theorem we deduce that the right-hand side of the last inequality tends to 0 as $n\rightarrow\infty$. Consequently, by taking first the supremum with respect to $i$ and then taking the limit with respect to $n$ on both sides of \eqref{2101.1}, we can conclude that $\{(y_\cdot^n,z_\cdot^n)\}^{+\infty}_{n=1}$ is a Cauchy sequence in the space of $H_\tau^2(\beta\mu_\cdot;\R^{k}\times\R^{k\times d})$.\vspace{0.2cm}

Finally, by taking the limit of BSDE$(\xi,\tau,\tilde{g}^n)$ under ucp, it can be concluded that under the assumptions of this step, BSDE \eqref{BSDE2.1} admits a weighted $L^2$-solution in the space of $H_\tau^2(\beta\mu_\cdot;\R^{k}\times\R^{k\times d})$.\vspace{0.2cm}

{\bf Third Step:} We remove the additional condition \eqref{22.6} and prove the existence of a weighted $L^2$-solution in $H_\tau^2(\beta\mu_\cdot;\R^{k}\times\R^{k\times d})$ to BSDE \eqref{BSDE2.1} under assumptions \ref{A:H1'}-\ref{A:H4'}.\vspace{0.2cm}

For each $x\in\R^k$, $r>0$ and $n\geq1$, let $q_r(x):=\frac{xr}{|x|\vee r}$, $\xi_n:=q_{n\alpha_{\tau}^2} (\xi)$ and
\begin{align}\label{3.1}
\overline{g}_n(t,y):=g(t,y)-g(t,0)+q_{ne^{-t} \alpha_{t}^2}(g(t,0)), \ t\in[0,\tau],
\end{align}
where $\alpha_\cdot$ is defined in assumption \ref{A:H3'}.
Then, for each $n\geq1$ we have
\begin{align}\label{1.13}
|\xi_n|\leq n\alpha_{\tau}^2,   \ |\overline{g}_n(t,0)|\leq ne^{-t}\alpha_{t}^2, \ t\in[0,\tau].
\end{align}
The definitions of $\xi_n$ and $g_n$ indicate that
\begin{align}\label{1.14}
\begin{split}
|\xi_{n+i}-\xi_{n}|\leq |\xi|{\bf 1}_{|\xi|>n\alpha_{\tau}^2}, \ \left|q_{(n+i)e^{-s}\alpha_{s}^2}(g(s,0))-q_{ne^{-s}\alpha_{s}^2}(g(s,0))\right|\leq |g(s,0)|{\bf 1}_{|g(s,0)|>ne^{-s}\alpha_{s}^2}.
\end{split}
\end{align}
Furthermore, by combining \eqref{1.13} and \ref{A:H1'}-\ref{A:H4'} of $g$, we know that $\overline{g}_n$ satisfies \ref{A:H2'}-\ref{A:H4'} for each $n\geq1$, and then  $\xi_n$ and $\overline{g}_n$ satisfy all assumptions in the second step. Then, BSDE$(\xi_n,\tau,\overline{g}_n)$ admits a weighted $L^2$-solution in $H_\tau^2(\beta\mu_\cdot;\R^{k}\times\R^{k\times d})$ for each $n\geq1$, denoted by $(y_t^n,z_t^n)_{t\in[0,\tau]}$.\vspace{0.2cm}

In the sequel, for each pair of integers $n, i\geq1$, let
$$\hat{\xi}^{n,i}:=\xi_{n+i}-\xi_n, \  \hat y_.^{n,i}:=y_.^{n+i}-y_.^n, \ \hat z_.^{n,i}:=z_.^{n+i}-z_.^n.$$
Then
\begin{align}\label{3.5}
\hat y_t^{n,i}=\hat{\xi}^{n,i}+\int_t^\tau\hat{g}^{n,i}(s,\hat y_s^{n,i}){\rm d}s-\int_t^\tau\hat z_s^{n,i}{\rm d}B_s, \ t\in[0,\tau],
\end{align}
where for each $y\in \R^k$, $$\hat{g}^{n,i}(s,y):=\overline{g}_{n+i}(s,y+y_s^n)-\overline{g}_n(s,y_s^n).$$
It follows from \ref{A:H4'} that for each $y\in \R^k$,
\begin{align*}
\begin{split}
\left<\hat{y},\hat{g}^{n,i}(t,y)\right>&=\left<\hat{y}, \overline{g}_{n+i}(t,y+y_t^n)-\overline{g}_{n+i}(t,y_t^n)+ \overline{g}_{n+i}(t,y_t^n)-\overline{g}_n(t,y_t^n)\right>\\
&\leq \mu_t|y| +\bigg|q_{(n+i)e^{-t}\alpha_{t}^2}(g(t,0)) -q_{ne^{-t}\alpha_{t}^2}(g(t,0))\bigg|, \ t\in[0,\tau].
\end{split}
\end{align*}
This implies that $\hat{g}^{n,i}$ satisfies assumption \ref{A:A} with $\bar\mu_t={\mu}_t$, $\bar\nu_t=0$ and $f_t=|g(t,0)|{\bf 1}_{|g(t,0)|>ne^{-t}\alpha_{t}^2}.$
By virtue of \eqref{2.03} in \cref{pro:1.1} with $r=t=0$ and \eqref{1.14}, we deduce that for each $n, \ i\geq1$, there exists a uniform constant $C>0$ such that
\begin{align}\label{3.6}
\begin{split}
&\E\left[\sup_{s\in[0,\tau]}\left(e^{2\beta \int_0^s\mu_r{\rm d}r}|\hat{y}_s^{n,i}|^2\right)\right]+\E\left[\int_{0}^{\tau}e^{2\beta \int_0^s\mu_r{\rm d}r}|\hat{z}_s^{n,i}|^2{\rm d}s\right]\\
&\ \ \leq  C\E\left[e^{2\beta \int_0^\tau \mu_r{\rm d}r}|\xi|^2{\bf 1}_{|\xi|>n\alpha_{\tau}^2}+\left(\int_0^\tau e^{\beta \int_0^s\mu_r{\rm d}r}|g(s,0)|{\bf 1}_{|g(s,0)|>ne^{-s}\alpha_{s}^2}{\rm d}s\right)^2\right].
\end{split}
\end{align}
Furthermore, by taking first the supremum with respect to $i$ and then taking the upper limit with respect to $n$ in both sides of \eqref{3.6} and by using Lebesgue's dominated convergence theorem, we know that $\{(y_\cdot^n,z_\cdot^n)\}^{+\infty}_{n=1}$ is a Cauchy sequence in $H_\tau^2(\beta\mu_\cdot;\R^{k}\times\R^{k\times d})$. Next, we denote by $(y_t,z_t)_{t\in[0,\tau]}$ the limit of the Cauchy sequence $\{(y_t^n,z_t^n)_{t\in[0,\tau]}\}_{n=1}^\infty$ in $H_\tau^2(\beta\mu_\cdot;\R^{k}\times\R^{k\times d})$, and pass to the limit under ucp for BSDE$(\xi_n,\tau,\overline{g}_n)$, considering the definitions  of $\xi_n$ and $\overline{g}_n$ along with the assumptions of $g$, to see that $(y_t,z_t)_{t\in[0,\tau]}$ solves BSDE \eqref{BSDE2.1}. Thus the proof of \cref{pro:2.1} is complete.
\end{proof}

\begin{rmk}\label{rmk:4.3}
As stated in the introduction, in order to establish the existence of solutions to BSDEs with random terminal time and generator $g$ having a general growth in $y$, the authors in \cite{Pardoux1999}, \cite{BriandCarmona2000}, \cite{Briand2003}, \cite{FanJiang2013}, \cite{Xiao2015} and \cite{XiaoandFan2017} developed successively the truncation technique, the approach proving convergence of the sequence via the established a priori estimate, the convolution approaching technique and the weak convergence method, most of which are systemically used to prove \cref{pro:2.1}. However, some new difficulties arise naturally under the assumptions of \cref{pro:2.1} due to the following facts:
\begin{enumerate}
\renewcommand{\theenumi}{a)}
\renewcommand{\labelenumi}{\theenumi}
\item\label{A:1} we want to find a solution in the space of $H_\tau^2(\beta\mu_\cdot;\R^{k}\times\R^{k\times d})$, not the usual space $H_\tau^2(0;\R^{k}\times\R^{k\times d})$;
\renewcommand{\theenumi}{b)}
\renewcommand{\labelenumi}{\theenumi}
\item\label{A:2} we do not suppose any moment integrability on $\int_0^\tau\mu_t{\rm d}t$;
\renewcommand{\theenumi}{c)}
\renewcommand{\labelenumi}{\theenumi}
\item\label{A:3} assumption \ref{A:H3} is strictly weaker than \ref{A:H3*} used in existing literature due to the presence of $\alpha_\cdot$.
\end{enumerate}
For that, we develop some innovative ideas to overcome these difficulties. Notably, except for the a priori estimate-\cref{pro:1.1}, the construction of process $\theta_r^{\alpha_\cdot}$ defined in \eqref{Matrix} and the truncation way for $\xi$, $g(t,0)$ and $g(t,\cdot)$ used in \eqref{2.6} and \eqref{22.6} are the key points. It ensure that for each $n\geq1$, the solution $(y_\cdot^n,z_\cdot^n)$ of BSDE$(\xi,\tau,g^n)$ constructed in the second step belongs to the same space $H_\tau^2(\beta\mu_\cdot;\R^{k}\times\R^{k\times d})$. And, the computations in \eqref{4.16**} and \eqref{22.101} also play crucial roles in the proof of \cref{pro:2.1} in order to obtain the uniform bound of the sequence of processes $y_\cdot^ne^{\beta\int_0^\cdot\mu_s{\rm d}s}$ so that $(y_\cdot^n,z_\cdot^n)$ is also the unique solution of BSDE$(\xi,\tau,\tilde{g}^n)$. To the best of our knowledge, these arguments mentioned above are totally new, compared to those in existing results.
\end{rmk}

\subsection{The general case}

In this subsection, with \cref{pro:2.1} in hand we can prove the general case of \cref{thm:3.1}.

\begin{proof}[\bf Proof of the existence part of Theorem \ref{thm:3.1}.]
The proof is divided into two steps.\vspace{0.2cm}

{\bf First Step:} Assume that $g$ satisfies assumptions \ref{A:H1}-\ref{A:H5} with processes $\mu_\cdot$ and $\nu_\cdot$. Recalling that  $a_t:=\beta\mu_t+\frac{\rho}{2}\nu_t^2$ satisfies $\int_0^\tau a_t{\rm d}t<+\infty$. In this step, we will use the conclusion of \cref{pro:2.1} to prove that for each $\xi\in L_\tau^2(a_\cdot;\R^k)$ and $V_\cdot\in M_\tau^2(a_\cdot;\R^{k\times d})$, the following BSDE admits a unique weighted $L^2$-solution in $H_\tau^2(a_\cdot;\R^{k}\times\R^{k\times d})$:
\begin{align}\label{BSDE3.1}
  y_t=\xi+\int_t^\tau g(s,y_s,V_s){\rm d}s-\int_t^\tau z_s{\rm d}B_s, \ \ t\in[0,\tau].
\end{align}
We first prove that  $g(t,y,V_t)$ satisfies assumptions \ref{A:H1'}-\ref{A:H4'} appearing in \cref{pro:2.1}. In fact, it is clear that $g(t,y,V_t)$ satisfies \ref{A:H2'} and \ref{A:H4'}. Moreover, by \ref{A:H1} and \ref{A:H5} of $g$ and H${\rm \ddot{o}}$lder's inequality we derive that
\begin{align}\label{0.35}
&\E\left[\left(\int_0^\tau e^{\beta \int_0^s{\mu}_r{\rm d}r}|g(s,0,V_s)|{\rm d}s\right)^2\right]\leq\E\left[\left(\int_0^\tau e^{\beta \int_0^s{\mu}_r{\rm d}r}\left(|g(s,0,0)|+{\nu}_s|V_s|\right){\rm d}s\right)^2\right]\nonumber\\
&\ \ \leq 2\E\left[\left(\int_0^\tau e^{\beta \int_0^s{\mu}_r{\rm d}r}|g(s,0,0)|{\rm d}s\right)^2\right]+2\E\left[\left(\int_0^\tau e^{ \int_0^sa_r{\rm d}r}|V_s|e^{- \int_0^s\frac{\rho}{2} {\nu}_r^2{\rm d}r}{\nu}_s{\rm d}s\right)^2\right]\nonumber\\
&\ \ \leq 2\E\left[\left(\int_0^\tau e^{\beta\int_0^s{\mu}_r{\rm d}r}|g(s,0,0)|{\rm d}s\right)^2\right]+2\E\left[\int_0^\tau e^{2 \int_0^sa_r{\rm d}r}|V_s|^2{\rm d}s\right]\E\left[\int_0^\tau e^{-\int_0^s\rho {\nu}_r^2{\rm d}r}{\nu}_s^2{\rm d}s\right]\nonumber\\
&\ \ \leq 2\E\left[\left(\int_0^\tau e^{\beta\int_0^s{\mu}_r{\rm d}r}|g(s,0,0)|{\rm d}s\right)^2\right]+\frac{2}{\rho} \E\left[\int_0^\tau e^{2 \int_0^sa_r{\rm d}r}|V_s|^2{\rm d}s\right]<+\infty.
\end{align}
Hence \ref{A:H1'} is true for $g(t,y,V_t)$. It follows from \ref{A:H5} that for each $n\geq1$, $r\in \R_+$ and $t\in[0,\tau]$,
\begin{align}\label{334}
\begin{split}
\overline{\psi}_{r}^{\alpha_\cdot}(t):&=\sup_{|y|\leq r\alpha_t}\left\{ \left|g(t,y,V_t)-g(t,0,V_t)\right|\right\}\\
&=\sup_{|y|\leq r\alpha_t}\left\{ \left|g(t,y,V_t)-g(t,y,0)+g(t,y,0)-g(t,0,0)+g(t,0,0)-g(t,0,V_t)\right|\right\}\\
&\leq2\nu_t|V_t|+\sup_{|y|\leq r\alpha_t}\left\{ \left|g(t,y,0)-g(t,0,0)\right|\right\}=2\nu_t|V_t|+{\psi}_{r}^{\alpha_\cdot}(t).
\end{split}
\end{align}
Thus combining \eqref{334}, \eqref{0.35} and \ref{A:H3} of $g$, we deduce that \ref{A:H3'} is also true for $g(t,y,V_t)$. According to \cref{pro:2.1}, BSDE \eqref{BSDE3.1} has a unique weighted $L^2$-solution $(y_t,z_t)_{t\in[0,\tau]}$ in the space of $H_\tau^2(\beta\mu_\cdot;\R^{k}\times\R^{k\times d})$, and blow we prove that $(y_t,z_t)_{t\in[0,\tau]}$ belongs also to the space $H_\tau^2(a_\cdot;\R^{k}\times\R^{k\times d})$.\vspace{0.2cm}

Applying It\^{o}'s formula to $e^{\beta\int_{0}^{t}\mu_s{\rm d}s}|y_t|$ yields that
\begin{align}\label{0.37}
\begin{split}
e^{\beta\int_{0}^{t\wedge\tau}\mu_s{\rm d}s}|y_{t\wedge\tau}|
\leq& e^{\beta\int_{0}^{\tau}\mu_s{\rm d}s}|\xi|+\int_{t\wedge\tau}^{\tau} e^{\beta\int_{0}^{s}\mu_r{\rm d}r}\left(\left<\hat{y}_s,g(s,y_s,V_s)\right>-\beta \mu_s|y_s|\right){\rm d}s\\
&-\int_{t\wedge\tau}^{\tau} e^{\beta \int_{0}^{s}\mu_r{\rm d}r}\langle \hat{y}_s,z_s{\rm d}B_s\rangle, \  \ t\geq0.
\end{split}
\end{align}
Using assumptions \ref{A:H4} and \ref{A:H5} we deduce that
\begin{align}\label{0.38}
\begin{split}
\left<\hat{y}_t,g(t,y_t,V_t)\right>-\beta \mu_t|y_t|\leq \nu_t|V_t|+|g(t,0,0)|, \ t\in[0,\tau].
\end{split}
\end{align}
By virtue of $(y_t,z_t)_{t\in[0,\tau]}\in H_\tau^2(\beta\mu_\cdot;\R^{k}\times\R^{k\times d})$ and the BDG inequality, we can derive that
$$\left(\int_0^{t\wedge\tau} e^{\beta \int_{0}^{s}\mu_r{\rm d}r}\langle \hat{y}_s,z_s{\rm d}B_s\rangle\right)_{t\geq0}$$
is a uniformly integrable martingale. In fact, in view of Theorem 1 in \cite{Ren2008BDG}, we deduce that
\begin{align}\label{2.06*}
\begin{split}
&\E\left[\sup_{t\in[0,\tau]}\left|\int_0^{t\wedge\tau} e^{\beta \int_{0}^{s}\mu_r{\rm d}r}\langle \hat{y}_s,z_s{\rm d}B_s\rangle\right|\right]
\leq 2\sqrt{2}\E\left[\left(\int_0^\tau e^{2\beta \int_{0}^{s}\mu_r{\rm d}r}|z_s|^2{\rm d}s\right)^{\frac{1}{2}}\right]<+\infty.
\end{split}
\end{align}
It follows from \eqref{0.37}, \eqref{0.38} and \eqref{2.06*} that for each $t\geq0$,
\begin{align}\label{0.39}
\begin{split}
e^{\beta\int_{0}^{t\wedge\tau}\mu_s{\rm d}s}|y_{t\wedge\tau}|\leq\E\left[e^{\beta\int_{0}^{\tau}\mu_s{\rm d}s}|\xi|\bigg|\F_{t\wedge\tau}\right] +\E\left[\int_{t\wedge\tau}^{\tau}e^{\beta\int_{0}^{s}\mu_r{\rm d}r}\left(\nu_s|V_s|+|g(s,0,0)|\right){\rm d}s\bigg|\F_{t\wedge\tau}\right].
\end{split}
\end{align}
Multiplying  $e^{\int_{0}^{t\wedge\tau}\frac{\rho}{2}\nu_s^2{\rm d}s}$ at both sides of the above equation and using H${\rm \ddot{o}}$lder's inequality, we have for each $t\geq0$,
\begin{align}\label{0.40}
\begin{split}
&e^{\int_{0}^{t\wedge\tau}a_s{\rm d}s}|y_{t\wedge\tau}|=e^{\int_{0}^{t\wedge\tau}(\beta{\mu}_s+\frac{\rho}{2}\nu_s^2){\rm d}s}|y_{t\wedge\tau}|
\\
&\ \ \leq \E\left[e^{\int_{0}^{\tau} a_s{\rm d}s}|\xi|\bigg|\F_{t\wedge\tau}\right]
+\E\left[\int_{t\wedge\tau}^{\tau} e^{\beta\int_{0}^{s}\mu_r{\rm d}r}\left(\nu_s|V_s|+|g(s,0,0)|\right){\rm d}s\cdot e^{\int_{0}^{t\wedge\tau}\frac{\rho}{2}\nu_s^2{\rm d}s}\bigg|\F_{t\wedge\tau}\right]\\
&\ \ \leq \E\left[e^{\int_{0}^{\tau} a_s{\rm d}s}|\xi|\bigg|\F_{t\wedge\tau}\right]+\E\left[\int_0^\tau e^{\int_{0}^{s} a_r{\rm d}r}|g(s,0,0)|{\rm d}s\bigg|\F_{t\wedge\tau}\right]\\
&\hspace{0.6cm}+\E\left[\sqrt{\int_{t\wedge\tau}^{\tau}e^{2\int_{0}^{s} a_r{\rm d}r}|V_s|^2{\rm d}s\int_{t\wedge\tau}^{\tau}e^{-\int_{0}^{s} \rho \nu_r^2{\rm d}r}\nu_s^2{\rm d}s\cdot e^{\int_{0}^{t\wedge\tau} \rho\nu_s^2{\rm d}s}}\bigg|\F_{t\wedge\tau}\right]\\
&\ \ \leq \E\left[e^{\int_{0}^{\tau} a_s{\rm d}s}|\xi|\bigg|\F_{t\wedge\tau}\right]+\E\left[\int_0^\tau e^{\int_{0}^{s} a_r{\rm d}r}|g(s,0,0)|{\rm d}s\bigg|\F_{t\wedge\tau}\right]\\
&\hspace{0.6cm}+\E\left[\sqrt{\frac{1}{\rho} \int_{0}^{\tau}e^{2\int_{0}^{s} a_r{\rm d}r}|V_s|^2{\rm d}s}\bigg|\F_{t\wedge\tau}\right].
\end{split}
\end{align}
Then, using the last inequality, the assumptions of $\xi$, $V_\cdot$, and $g(t,0,0)$, and Doob's martingale inequality, we can deduce that
\begin{align*}
\begin{split}
&\E\left[\sup_{s\in[0,\tau]}\left(e^{2\int_0^s a_r{\rm d}r}|y_s|^2\right)\right]\\
&\ \ \leq3\E\left[e^{2\int_{0}^{\tau} a_s{\rm d}s}|\xi|^2+\left(\int_0^\tau e^{\int_{0}^{s} a_r{\rm d}r}|g(s,0,0)|{\rm d}s\right)^2+\frac{1}{\rho}\int_{0}^{\tau}e^{2\int_{0}^{s} a_r{\rm d}r}|V_s|^2{\rm d}s\right]<+\infty.
\end{split}
\end{align*}
Thus, $y_\cdot\in S^2_\tau(a_\cdot;\R^{k})$. Furthermore, in view of \eqref{0.38} together with $y_\cdot\in S^2_\tau(a_\cdot;\R^{k})$ and the assumptions of $V_\cdot$ and $g(t,0,0)$, similar to the proof of \eqref{2.2} and \eqref{0032.2} in \cref{pro:1.1} with $r=t=0$, we can obtain that for each $n\geq1$,
\begin{align}\label{32.82}
\E\left[\int_{0}^{\tau_n}e^{2 \int_{0}^{s}a_r{\rm d}r}|z_s|^2{\rm d}s\right]
\leq&
\E\left[\sup_{s\in[0,\tau_n]}\left(e^{2 \int_{0}^{s}a_r{\rm d}r}|y_s|^2\right)\right] +\frac{1}{\rho} \E\left[\int_{0}^{\tau_n}e^{2\int_0^s a_r{\rm d}r}|V_s|^2{\rm d}s\right]\nonumber\\
&+2\E\left[\int_{0}^{\tau_n}e^{2\int_0^s a_r{\rm d}r}|y_s||g(s,0,0)|{\rm d}s\right]\nonumber\\
\leq&
2\E\left[\sup_{s\in[0,\tau]}\left(e^{2\int_0^s a_r{\rm d}r}|y_s|^2\right)\right]+\frac{1}{\rho} \E\left[\int_{0}^{\tau}e^{2\int_0^s a_r{\rm d}r}|V_s|^2{\rm d}s\right]\nonumber\\
&+\E\left[\left(\int_{0}^{\tau}e^{\int_0^s a_r{\rm d}r}|g(s,0,0)|{\rm d}s\right)^2\right]<+\infty.
\end{align}
Finally, by taking the limit with respect to $n$ in the above inequality, the use of Levi's lemma yields $z_\cdot\in M_\tau^2(a_\cdot;\R^{k\times d})$. Hence, $(y_\cdot,z_\cdot)\in H_\tau^2(a_\cdot;\R^{k}\times\R^{k\times d})$.\vspace{0.2cm}

{\bf Second Step:} By the fixed point theorem, we will prove that under assumptions \ref{A:H1}-\ref{A:H5} of $g$, for each $\xi\in L_\tau^2(a_\cdot;\R^k)$, BSDE \eqref{BSDE1.1} admits a unique weighted $L^2$-solution $(y_\cdot,z_\cdot)\in H_\tau^2(a_\cdot;\R^{k}\times\R^{k\times d})$. In fact, based on the proof of the first step, we know that for any given $V_\cdot\in M_\tau^2(a_\cdot;\R^{k\times d})$, there exists a weighted $L^2$-solution $(y_t,z_t)_{t\in[0,\tau]}\in S^2_\tau(a_\cdot;\R^{k})\times M_\tau^2(a_\cdot;\R^{k\times d})$ of BSDE \eqref{BSDE3.1}. Thus, we choose the $z_\cdot$ in the solution to be the image of $V_\cdot$ and  construct a mapping:
\begin{align*}
\begin{split}
\Phi: M_\tau^2(a_\cdot;\R^{k\times d})&\rightarrow M_\tau^2(a_\cdot;\R^{k\times d})\\
V_\cdot&\rightarrow z_\cdot.
\end{split}
\end{align*}
Now, suppose that for each $i=1,2$, $V_\cdot^i\in M_\tau^2(a_\cdot;\R^{k\times d})$, and $z_\cdot^i:=\Phi(V_\cdot^i)$. Denote
\begin{align*}
\begin{split}
&\overline{V_t}:=V_t^1-V_t^2, \ \overline{y}_t:=y_t^1-y_t^2, \ \overline{z}_t:=z_t^1-z_t^2, \ t\in [0,\tau].
\end{split}
\end{align*}
Applying It\^{o}'s formula to $e^{2 \int_{0}^{t}a_s{\rm d}s}|\overline{y}_t|^2$ and using \ref{A:H4} and \ref{A:H5}, we deduce that
\begin{align}\label{33.9}
\begin{split}
&|\overline{y}_0|^2+\int_0^{\tau} e^{2\int_0^s a_r{\rm d}r}|\overline{z}_s|^2{\rm d}s+2\int_0^{\tau} e^{2 \int_0^s a_r{\rm d}r}a_s|\overline{y}_s|^2{\rm d}s\\
&\ \ = 2\int_0^{\tau} e^{2\int_0^s a_r{\rm d}r}\langle \overline{y}_s,g(s,y_s^1,V_s^1)-g(s,y_s^2,V_s^2)\rangle{\rm d}s-2\int_0^{\tau} e^{2\int_0^s a_r{\rm d}r}\langle \overline{y}_s,\overline{z}_s{\rm d}B_s\rangle\\
&\ \ \leq \int_0^{\tau} e^{2\int_0^s a_r{\rm d}r}\left(2\mu_s|\overline{y}_s|^2+2|\overline{y}_s|\nu_s|\overline{V_s}|\right){\rm d}s-2\int_0^{\tau} e^{2\int_0^s a_r{\rm d}r}\langle \overline{y}_s,\overline{z}_s{\rm d}B_s\rangle.
\end{split}
\end{align}
Considering $2ab\leq \rho a^2+\frac{1}{\rho} b^2$, we know that for each $s\in[0,\tau]$,
\begin{align}\label{3.9}
2|\overline{y}_s|\nu_s|\overline{V_s}|\leq\rho \nu^2_s|\overline{y}_s|^2+\frac{1}{\rho}|\overline{V_s}|^2.
\end{align}
Moreover, it then follows from the BDG inequality that $$\left(\int_0^{t\wedge\tau} e^{2\int_0^s a_r{\rm d}r}\langle \overline{y}_s,\overline{z}_s{\rm d}B_s\rangle\right)_{t\geq0}$$ is a uniformly integrable martingale. Indeed, in view of Theorem 1 in \cite{Ren2008BDG}, we have
\begin{align}\label{2.06}
\begin{split}
&2\E\left[\sup_{t\in[0,\tau]}\left|\int_0^{t\wedge\tau} e^{2\int_0^s a_r{\rm d}r}\langle \overline{y}_s,\overline{z}_s{\rm d}B_s\rangle\right|\right]
\leq 4\sqrt{2}\E\left[\left(\int_0^\tau e^{4 \int_{0}^{s}a_r{\rm d}r}|\overline{y}_s|^2|\overline{z}_s|^2{\rm d}s\right)^{\frac{1}{2}}\right]\\
&\ \ \leq  \frac{1}{2}\E\left[\sup_{s\in[0,\tau]}\left(e^{2 \int_{0}^{s}a_r{\rm d}r}|\overline{y}_s|^2\right)\right]+64\E\left[\int_0^\tau e^{2\int_0^s a_r{\rm d}r}|\overline{z}_s|^2{\rm d}s\right]<+\infty.
\end{split}
\end{align}
According to \eqref{3.9} and \eqref{2.06}, by taking the mathematical expectation in both sides of \eqref{33.9} we get
\begin{align*}\label{339}
\begin{split}
\E\left[\int_0^{\tau} e^{2\int_0^s a_r{\rm d}r}|\overline{z}_s|^2{\rm d}s\right]+(2\beta-2) \E\left[\int_0^{\tau} e^{2\int_0^s a_r{\rm d}r}\mu_s|\overline{y}_s|^2{\rm d}s\right]
\leq  \frac{1}{\rho}\E\left[ \int_0^{\tau} e^{2\int_0^s a_r{\rm d}r}|\overline{V_s}|^2{\rm d}s\right].
\end{split}
\end{align*}
Hence, $\Phi$ is a strict contraction in $M_\tau^2(a_\cdot;\R^{k\times d})$, and then has a unique fixed point $z_\cdot\in M_\tau^2(a_\cdot;\R^{k\times d})$, i.e., $\Phi(z_\cdot)=z_\cdot$. It implies that according to the definition of mapping $\Phi$ and the assertion in first step, there must exist $y_\cdot\in S^2_\tau(a_\cdot;\R^{k})$ such that $(y_\cdot,z_\cdot)\in S^2_\tau(a_\cdot;\R^{k})\times M_\tau^2(a_\cdot;\R^{k\times d})$ is the unique weighted $L^2$-solution of BSDE \eqref{BSDE1.1}. The existence part of Theorem \ref{thm:3.1} is then proved.
\end{proof}

\begin{rmk}\label{rmk:4.4}
In \cref{thm:3.1}, we want to find a solution of BSDE \eqref{BSDE1.1} in the weighted space of $H_\tau^2(a_\cdot;\R^{k}\times\R^{k\times d})$, but based on \cref{pro:2.1} we can only conclude that BSDE \eqref{BSDE3.1} admits a solution $(y_\cdot,z_\cdot)$ in the weighted space of $H_\tau^2(\beta\mu_\cdot;\R^{k}\times\R^{k\times d})$ when $\xi\in L^2_\tau(a_\cdot;\R^{k})$ and $V_\cdot\in M_\tau^2(a_\cdot;\R^{k\times d})$. Consequently, how to prove that $(y_\cdot,z_\cdot)$ also belongs to $H_\tau^2(a_\cdot;\R^{k}\times\R^{k\times d})$ is the key point in first step of the proof. This is quite different from that in usual case of $\int_0^\tau a_t{\rm d}t\leq M$ for some constant $M>0$, in which both spaces $H_\tau^2(a_\cdot;\R^{k}\times\R^{k\times d})$ and  $H_\tau^2(\beta\mu_\cdot;\R^{k}\times\R^{k\times d})$ are identical to the space $H_\tau^2(0;\R^{k}\times\R^{k\times d})$. Consequently, the arguments from \eqref{0.37} and \eqref{32.82} is necessary and innovative. In the second step of the proof, by the assertion in the first step, we successfully construct a contract mapping $\Phi$ from $M_\tau^2(a_\cdot;\R^{k\times d})$ to itself such that $\Phi(V_\cdot)=z_\cdot$, rather than a contract mapping $\Psi$ from $H_\tau^2(a_\cdot;\R^{k}\times\R^{k\times d})$ to itself such that $\Psi(U_\cdot,V_\cdot)=(y_\cdot,z_\cdot)$ as usually done in existing literature. This difference leads directly to enlargement of rang of $\beta$, i.e., $\beta\geq1$ in our framework rather than $\beta>\frac{1+\sqrt{5}}{2}$ in \cite{Li2023}.
\end{rmk}
\setlength{\bibsep}{2pt}
\bibliographystyle{model5-names}
\biboptions{authoryear}

\end{document}